\documentclass[11pt]{article}
\usepackage[margin=0.7in]{geometry}
\usepackage[utf8]{inputenc}
\usepackage{amsfonts}
\usepackage{amsthm}
\usepackage{amsmath}
\usepackage{appendix}
\usepackage{enumerate}
\usepackage{amssymb}
\usepackage{standalone}
\usepackage{xfrac} 
\usepackage[dvipsnames]{xcolor}
\usepackage[colorlinks,citecolor=green, linkbordercolor=blue,linkcolor=blue, urlcolor=blue,bookmarks=false,hypertexnames=true]{hyperref} 
\usepackage{graphicx}
\usepackage[T1]{fontenc}
\usepackage{tgtermes}
\graphicspath{ {./images/} }

\theoremstyle{plain}
\newtheorem{thm}{Theorem}[section]
\newtheorem{lem}[thm]{Lemma}
\newtheorem{prop}[thm]{Proposition}
\newtheorem{cor}[thm]{Corollary}

\theoremstyle{definition}

\numberwithin{equation}{section}
\numberwithin{thm}{section}
\numberwithin{defn}{section}

\title{A Polynomial Rate of Convergence for the Dirichlet Problem on Orthodiagonal Maps}
\author{David Pechersky}
\date{\today}

\begin{document}

\maketitle

\begin{abstract}

\noindent We extend recent work of Gurel-Gurevich, Jerison and Nachmias \cite{GJN20} and Gwynne and Bou-Rabee \cite{BG24} by showing that as the mesh of our lattice tends to $0$, we have a polynomial rate of convergence for the Dirichlet problem on orthodiagonal maps with H\"older boundary data to its continuous counterpart. \\ \\
The key idea is that the convolution of a discrete harmonic function on an orthodiagonal map with a smooth mollifier has small Laplacian and so is ``almost harmonic." This also allows us to show that discrete harmonic functions on orthodiagonal maps are Lipschitz in the bulk on a mesoscopic scale.

\end{abstract}

\tableofcontents{}

\section{Introduction}
\label{sec: Introduction}

\subsection{A Polynomial Rate of Convergence for the Dirichlet Problem on Orthodiagonal Maps}
\label{subsec: A Polynomial Rate of Convergence for the Dirichlet Problem on Orthodiagonal Maps}

Due to the ubiquity of diffusion phenomena in the physical world, the Dirichlet problem is one of the most important partial differential equations in mathematical physics. Given a domain $\Omega\subseteq{\mathbb{R}^{n}}$ and boundary data $g:\partial{\Omega}\rightarrow{\mathbb{R}}$, we say that a function $h\in{C^{2}(\Omega)\cap{C(\overline{\Omega})}}$ solves the Dirichlet problem on $\Omega$ with boundary data given by $g$ if: 
\begin{align*}
    \Delta{h}(x)&=0 \hspace{10pt} \text{for all $x\in{\mathbb{R}^{n}}$} \\
    h(x)&=g(x) \hspace{10pt} \text{for $x\in{\partial{\Omega}}$}
\end{align*}
In the absence of an explicit analytic solution, to solve the Dirichlet problem, we need to solve it numerically. One natural way to do this is to discretize. Namely, we replace  our partial differential equation on $\Omega$ by a finite difference equation on a finite grid $\Gamma$ which is a good approximation of the continuous domain $\Omega$. Here we are exploiting the fact that this discrete Dirichlet problem converges to the corresponding continuous Dirichlet problem as the mesh of our grid tends to $0$. This approach appeared in the literature as early as the 1920's with the work of Wiener and Phillips \cite{PW23} and Courant, Friedrichs and Lewy \cite{CFL28}. For an example of prior work where we have an explicit rate for the convergence of the discrete Dirichlet problem to its continuum analogue, see Theorem 2 in Chapter 4.3 of \cite{theory of difference schemes}.  \\ \\ 
More recently, the problem of verifying the convergence of the discrete Dirichlet problem to its continuum analogue, for a wide class of discretizations of continuous 2D space, has attracted attention due to the connection between this problem and questions about the behavior of critical statistical physics in two dimensions. In brief, a very successful approach for proving rigorous statements about critical 2D lattice models has been to identify discrete holomorphic observables in these models and to show that for finer and finer discretizations of 2D space, these discrete holomorphic observables converge to the corresponding continuous holomorphic functions in the limit. From here, the convergence of these observables is leveraged to prove convergence of random processes or observables of interest. Some of the landmark results of 21$^{\text{st}}$ century probability theory that have been proven in this way include the proof of Cardy's formula for crossing probabilities in critical percolation \cite{Sm01}, the convergence of interfaces in the critical Ising model to $SLE_{3}$ \cite{Ising}, and the convergence of critical percolation interfaces to $SLE_{6}$ \cite{CN07}. For details, see \cite{KS07}. More recently, this approach has also allowed mathematicians to prove explicit rates of convergence for observables and random processes in critical 2D lattice models. For instance, an explicit polynomial rate of convergence has been established for the convergence of crossing probabilities in critical percolation to their continuum limit \cite{MNS14, BCL15}, the convergence of loop-erased random walk to $SLE_{2}$ \cite{V15}, and the convergence of percolation interfaces to $SLE_{6}$ \cite{BR24}. Since the real and imaginary parts of a discrete holomorphic function are discrete harmonic, it is not hard to see why understanding the convergence properties of the Dirichlet problem would be useful for understanding the convergence of discrete holomorphic functions, which can then be leveraged to gain information about critical 2D lattice models.  
\\ \\ Orthodiagonal maps are a large class of discretizations of 2D space that admit a notion of discrete complex analysis. That is, we can define a discrete analogue of the Cauchy- Riemann equations for functions defined on the vertices of an orthodiagonal map. The functions that satisfy these discrete Cauchy Riemann equations are said to be discrete holomorphic and have many of the properties that are characteristic of continuous holomorphic functions: we have a discrete analogue of Morera's theorem, the real and imaginary parts of a discrete holomorphic function are discrete harmonic, a priori regularity estimates, etc. In \cite{GJN20}, Gurel-Gurevich, Jerison and Nachmias showed that solutions to the Dirichlet problem on orthodiagonal maps converge to the solution of the corresponding continuous Dirichlet problem. This improves on prior work of Chelkak and Smirnov \cite{CS11}, Skopenkov \cite{Sk13} and Werness \cite{W15}, where this result is proven for the Dirichlet problem on orthodiagonal maps, subject to various additional regularity assumptions on the underlying lattice. In particular, Theorem 1.1 of \cite{GJN20} provides an explicit rate of convergence for the Dirichlet problem on orthodiagonal maps to the corresponding continuous Dirichlet problem for $C^{2}$ boundary data. More recently, Bou-Rabee and Gwynne established a sharp polynomial rate of convergence for solutions to the discrete Dirichlet problem on orthodiagonal maps to the corresponding continuous Dirichlet problem (see Theorem 3.3 of \cite{BG24}). Furthermore, their methods don't just apply to the two dimensional orthodiagonal setting: their result works for discrete harmonic functions on a large class of tilings of d-dimensional space for $d\geq{2}$, including sphere packings and Delaunay triangulations (see Theorem B, 3.5 and 3.6 of \cite{BG24}). However, in both the two dimensional orthodiagonal case and the more general d-dimensional setting, for Theorems 3.3, 3.5, and 3.6 of \cite{BG24} to give an effective estimate of the rate of convergence of the discrete Dirichlet problem to its continuum analogue, the second derivative of the solution to the corresponding continuous Dirichlet problem must be uniformly bounded. To ensure that this is the case, one must make additional assumptions as to the regularity of the boundary data, as well as the regularity of the boundary of our domain. For instance, Theorem 6.14 of \cite{GT01} tells us that if our domain is bounded and the boundary data and the boundary of our domain are both $C^{2}$, then the second derivative of the solution to the corresponding continuous Dirichlet problem extends continuous to the boundary and so is bounded uniformly. By using an entirely different approach, in Section \ref{sec: proof of polynomial rate of convergence for Dirichlet problem on OD maps}, we extend the work of Bou-Rabee and Gwynne by showing that for any simply connected domain, as long as our boundary data is H\"older, we have a polynomial rate of convergence for the Dirichlet problem on orthodiagonal maps to the corresponding continuous Dirichlet problem (see Theorem \ref{thm: polynomial rate of convergence for Dirichlet problem on OD maps}). Unlike in \cite{BG24}, our rate of convergence is not sharp. 

\subsection{Lipschitz Regularity in the Bulk on a Mesoscopic Scale for Harmonic Functions on Orthodiagonal Maps}
\label{subsec: Lipschitz Regularity in the Bulk on a Mesoscopic Scale for Harmonic Functions on Orthodiagonal Maps}

Suppose $\Omega$ is a subdomain of $\mathbb{R}^{2}$ and $h:\Omega\rightarrow{\mathbb{R}}$ is harmonic. Then the classical Harnack estimate says that if $x,y\in{\Omega}$ and $d=\text{dist}(x,\partial{\Omega})\wedge{\text{dist}(y,\partial{\Omega})}$, then: 
\begin{equation}
\label{eqn: Harnack estimate for continuous harmonic functions}
    |h(y)-h(x)|\leq{2\hspace{1pt}\|h\|_{L^{\infty}(\Omega)}\Big(\frac{|x-y|}{d}\Big)}
\end{equation}
This result follows from the mean value property for harmonic functions. Namely, if $x\in{\Omega}$ and $r<{\text{dist}(x,\partial{\Omega})}$, we have that: 
\begin{equation}
\label{eqn: integral MVP for harmonic functions}
    h(x)=\frac{1}{\pi{r}^{2}}\int_{B(x,r)}h(u)dA(u)
\end{equation}
where ``$dA(u)$" refers to integration with respect to area on $\Omega$. Insofar as orthodiagonal maps are good approximations of continuous 2D space, we should expect that something like this is true for discrete harmonic functions on orthodiagonal maps. Indeed, in the more restricted setting of isoradial graphs with angles uniformly bounded away from $0$ and $\pi$, Chelkak and Smirnov show that an anologue of the Harnack estimate holds (see Corollary 2.9 of \cite{CS11}). Just as in the continuous setting, the result follows from an analogue of the integral mean value property in Equation \ref{eqn: integral MVP for harmonic functions} for discrete harmonic functions (see Proposition A.2 of \cite{CS11}). The proof of this mean value property for discrete harmonic functions on isoradial graphs requires asymptotics for the discrete Green's function on isoradial graphs, proven by Kenyon (see Theorem 7.3 of \cite{Ke02}). It is expected that these estimates should also hold for the discrete Green's function on general orthodiagonal maps, however, this has yet to be proven. In light of this, to prove Harnack-type estimates for discrete harmonic functions on orthodiagonal maps, we will use a different approach.\\ \\
Namely, in Section \ref{sec: Lipschitz Regularity on a Mesoscopic Scale for Harmonic Functions on Orthodiagonal Maps}, we show that if you convolve a discrete harmonic function with a smooth mollifier, the resulting continuous function is ``almost" harmonic in that its Laplacian is  small in a precise quantitative sense. It turns out that to have a Harnack-type estimate like the one in Equation \ref{eqn: Harnack estimate for continuous harmonic functions}, we do not need our function to be harmonic. Being almost harmonic is enough. Thus, we have a Harnack-type estimate for the convolution of a discrete harmonic function with a smooth mollifier.\\ \\ Regularity estimates for discrete harmonic functions on orthodiagonal maps tell us that that our original discrete harmonic function is close to its convolution with a smooth mollifier, provided the support of this smooth mollifier is small. Since discrete harmonic functions on orthodiagonal maps are close to continuous functions that satisfy a Harnack-type estimate, we conclude that we also have a Harnack-type estimate for discrete harmonic functions on orthodiagonal maps, at least on a mesoscopic scale. 

\subsection{Conventions and Notation}
\begin{itemize}
    \item If $a,b\in{\mathbb{R}}$,
    \begin{align*}
        a\wedge{b}&:=\min\{a,b\} & a\vee{b}&:=\max\{a,b\}
    \end{align*}
    \item If $v\in{\mathbb{R}^{n}}$, 
    $$
    |v|:=\sqrt{v_{1}^{2}+v_{2}^{2}+...+v_{n}^{2}}
    $$
    \item If $f:\Omega\rightarrow{\mathbb{R}^{m}}$, where $\Omega$ is a subdomain of $\mathbb{R}^{n}$, then:
    \begin{equation*}
        \|f\|:=\sup\{|f(x)|: x\in{\Omega}\}
    \end{equation*}
    \item Given a $2\times{2}$ matrix $A=\begin{bmatrix} a_{1,1} & a_{1,2} \\ a_{2,1} & a_{2,2} \end{bmatrix}$, $\|A\|_{2}$ is the Frobenius norm of $A$: 
    $$
    \|A\|_{2}=\sqrt{a_{1,1}^{2}+a_{1,2}^{2}+a_{2,1}^{2}+a_{2,2}^{2}}
    $$ 
    \item In particular, if $f:\Omega\rightarrow{\mathbb{R}}$ is $C^{2}$, where $\Omega$ is a subdomain of $\mathbb{R}^{2}$, then: 
    \begin{equation*}
        \|D^{2}f(x)\|_{2}=\sqrt{(\partial_{1}^{2}f(x))^{2}+2(\partial_{1}\partial_{2}f(x))^{2}+(\partial_{2}^{2}f(x))^{2}}
    \end{equation*}
    and: 
    \begin{equation*}
        \|D^{2}f\|:=\sup\{\|D^{2}f(x)\|_{2}:x\in{\Omega}\}
    \end{equation*}
    If $f$
    \item Similarly, if $f:\Omega\rightarrow{\mathbb{R}}$ is $C^{3}$, where $\Omega$ is a subdomain of $\mathbb{R}^{2}$, then:
    \begin{equation*}
        \|D^{3}f(x)\|_{2}=\sqrt{(\partial_{1}^{3}f(x))^{2}+(\partial_{1}^{2}\partial_{2}f(x))^{2}+(\partial_{1}\partial_{2}^{2}f(x))^{2}+(\partial_{2}^{3}f(x))^{2}}
    \end{equation*}
    and: 
    \begin{equation*}
        \|D^{3}f\|:=\sup\{\|D^{2}f(x)\|_{2}:x\in{\Omega}\}
    \end{equation*}
    
    \item Recall that for any matrix, its Frobenius norm is an upper bound for its operator norm. In particular, if $v,w\in{\mathbb{R}^{2}}$ and $A$ is a $2\times{2}$ matrix,
    \begin{equation*}
        |v^{T}Aw|\leq{|v|\cdot{\|A\|_{2}\cdot|w|}}
    \end{equation*}
\end{itemize}

\section*{Acknowledgements} We thank Ilia Binder for many valuable discussions over many years. In particular, Ilia was the one who suggested that we look at \cite{Ch22} and \cite{BR21}. The ideas in these papers formed the basis for the proof of our main result, Theorem \ref{thm: polynomial rate of convergence for Dirichlet problem on OD maps}. We would also like to thank Dmitry Chelkak and Marianna Russkikh for discussing the results of \cite{CLR23} with us and commenting on the state-of-the-art in discrete complex analysis more broadly.

\section{A Brief Introduction to Orthodiagonal Maps}
\label{sec: A Brief Introduction to Orthodiagonal Maps}

Before we can precisely state our results, we first need to establish some terminology and remind the reader of some basic facts. Our exposition largely follows that of Section 1.1 of \cite{GJN20} and Sections 2.2 and 2.4 of \cite{BP24}.\\ \\
A \textbf{finite network} is a finite graph $G=(V,E)$ with positive edge weights $(c(e))_{e\in{E}}$. For any edge $e\in{E}$, we say that $c(e)$ is the \textbf{conductance} of that edge. Given a function $f:V\rightarrow{\mathbb{R}}$, its \textbf{Laplacian} $\Delta{f}:V\rightarrow{\mathbb{R}}$ is given by: 
\begin{equation*}
    \Delta{f}(x)=\sum_{y:y\sim{x}}c(x,y)\big(f(y)-f(x)\big)
\end{equation*}
where the expression $y\sim{x}$ indicates that $y$ and $x$ are neighbors in $G$. If $\Delta{f(x)}=0$, we say that $f$ is \textbf{harmonic} at $x$. Equivalently, $f$ is harmonic at $x$ if: 
\begin{equation*}
    f(x)=\frac{1}{\pi_{x}}\sum_{y:y\sim{x}}c(x,y)f(y)
\end{equation*}
where: 
\begin{equation*}
    \pi_{x}=\sum_{y:y\sim{x}}c(x,y)
\end{equation*}
This is reminiscent of the mean value property for continuous harmonic functions. From this formula, it is immediate that discrete harmonic functions satisfy the maximum principle: 
\begin{prop}
    If $G=(V,E,c)$ is a finite network, $U\subsetneq{V}$, and $f:V\rightarrow{\mathbb{R}}$ is harmonic on $U$. Define: 
    \begin{equation*}
        \partial{U}=\{w\in{V}\setminus{U}:w\sim{u} \hspace{5pt}\text{for some}\hspace{5pt} u\in{U} \}
    \end{equation*}
    Then: 
    \begin{equation*}
        \max_{u\in{U}}h(u)\leq{\max_{v\in{\partial{U}}}}h(v)
    \end{equation*} 
\end{prop}
\noindent Given a finite network $G=(V,E,c)$, a \textbf{simple random walk} on $G$ is a discrete-time Markov process $(S_{n})_{n\geq{0}}$ with transition probabilities: 
\begin{equation*}
    \mathbb{P}(S_{n}=y\hspace{1pt}|\hspace{1pt}S_{n-1}=x)=\frac{c(x,y)}{\pi_{x}}1_{(y\sim{x})}
\end{equation*}
In particular, notice that if $f:V\rightarrow{\mathbb{R}}$ is harmonic on $U\subseteq{V}$ and $T_{V\setminus{U}}$ is the hitting time of $V\setminus{U}$ by our simple random walk, then $(f(S_{n\wedge{T_{V\setminus{U}}}}))_{n\geq{0}}$ is a martingale. 
\\ \\
A \textbf{plane graph} is a graph $G=(V,E)$ with a fixed, proper embedding in the plane. We frequently identify the vertices and edges of $G$ with the corresponding points and curves in the plane in our embedding. The \textbf{faces} of a plane graph $G$ are the connected components of $\mathbb{R}^{2}$ minus the edges and vertices of $G$. All but one of its faces will be bounded: these are called \textbf{inner faces}. The unbounded face is known as the \textbf{outer face}. Given a plane graph $G$, let $\widehat{G}$ denote the subdomain of $\mathbb{C}$ formed by taking the interior of the union of the faces, vertices, and edges of $G$.   \\ \\ 
A finite \textbf{orthodiagonal map}\footnote{Our definition of an orthodiagonal map agrees with the one in \cite{GJN20} but differs from the one in \cite{BP24}, where an orthodiagonal map can have multiple outer faces, in addition to the unbounded face. These additional outer faces, which need not be quadrilaterals, are thought of as ``holes" in our orthodiagonal map. Using the language of \cite{BP24}, what we call an orthodiagonal map is a simply connected orthodiagonal map.} is a finite, connected plane graph such that: 
\begin{itemize}
    \item all edges are straight line segments.
    \item all faces are quadrilaterals with orthogonal diagonals.
    \item the boundary of the outer face is a simple closed curve. 
\end{itemize}
We allow non- convex quadrilaterals whose diagonals do not intersect. Intuitively, an orthodiagonal map consists of a graph and its dual, embedded in the plane so that the primal and dual edges are orthogonal. To see what we mean by this, observe first that any finite orthodiagonal map is bipartite: if our orthodiagonal map contained an odd cycle, there would be a face of odd degree inside this cycle. Hence, we have a bipartition of the vertices $V=V^{\bullet}\sqcup{V^{\circ}}$. The vertices of $V^{\bullet}$ are known as the \textbf{primal} vertices of $G$. The vertices of $V^{\circ}$ are known as the \textbf{dual} vertices. These give rise to the primal and dual graphs $G^{\bullet}=(V^{\bullet}, E^{\bullet})$ and $G^{\circ}=(V^{\circ}, E^{\circ})$. $G^{\bullet}$ is formed by connecting any pair of primal vertices that share an inner face in $G$. Similarly, $G^{\circ}$ is formed by connecting any pair of dual vertices that share an inner face in $G$. Since the inner faces of $G$ are all quadrilaterals, each inner face corresponds to one primal and one dual edge. In this way, there is a natural correspondence between primal and dual edges. Furthermore, our condition that the inner faces of $G$ have orthogonal diagonals is equivalent to requiring that corresponding primal and dual edges are orthogonal. Given an inner face $Q$ of $G$, let $e^{\bullet}_{Q}$ be the edge of $G^{\bullet}$ corresponding to the primal diagonal of $Q$. Similarly, $e^{\circ}_{Q}$ is the edge of $G^{\circ}$ corresponding to the dual diagonal of $Q$. Conversely, if $\{u,v\}$ is a primal (resp. dual) edge of our orthodiagonal map, $Q=Q_{u,v}$ denotes the quadrilateral in our orthodiagonal map that has $\{u,v\}$ as its primal (resp. dual) diagonal. Equivalently, we write $Q=[u,r,v,s]$, where $u,v,r,s$ are the vertices of $Q$ listed in counterclockwise order so that $u\in{V^{\bullet}}$. Notice that using this notation, $[u,r,v,s]=[v,s,u,r]$.   \\ \\
Observe that the square lattice, the triangular lattice and the hexagonal lattice all have this property that primal and dual edges are orthogonal. More generally, isoradial graphs, which have been studied widely in the context of critical statistical physics in 2D (i.e. see \cite{Ke02} and \cite{CS11}), are precisely orthodiagonal maps whose inner faces are all rhombii. Furthermore, by the double circle packing theorem, a wide variety of planar graphs admit an orthodiagonal embedding (see Section 2 of \cite{GJN20}). \\ \\
The boundary vertices of $G^{\bullet}$, which we denote by $\partial{V^{\bullet}}$, are the vertices of $V^{\bullet}$ that lie along the boundary of the outer face of $G$. The remaining vertices of $G^{\bullet}$ are called interior vertices and are denoted by $\text{Int}(V^{\bullet})$. $\partial{V^{\circ}}$ and $\text{Int}(V^{\circ})$, the boundary and interior vertices of $G^{\circ}$, are defined analogously. The graphs $G^{\bullet}$ and $G^{\circ}$ are both connected plane graphs (see Lemma 3.2 of \cite{GJN20}). Furthermore, as we alluded to earlier, $G^{\bullet}$ and $G^{\circ}$ are almost duals of each other, except the outer face of $G^{\bullet}$ contains all of the boundary vertices of $G^{\circ}$ and the outer face of $G^{\circ}$ contains all of the boundary vertices of $G^{\bullet}$. \\ \\
A function $F:V\rightarrow{\mathbb{C}}$ is \textbf{discrete holomorphic} if for every inner face $Q=[v_{1},w_{1},v_{2},w_{2}]$ of $G$, $F$ satisfies the discrete Cauchy- Riemann equations:
\begin{equation}
\label{eqn: discrete Cauchy-Riemann}
    \frac{F(v_{2})-F(v_{1})}{v_{2}-v_{1}}=\frac{F(w_{2})-F(w_{1})}{w_{2}-w_{1}}
\end{equation}
Here, we think of $v_{1},v_{2},w_{1},w_{2}$ as points in the complex plane. Notice that if $F|_{V^{\bullet}}$ is strictly real then, up to a constant, $F|_{V^{\circ}}$ is purely imaginary. This is typical of applications of discrete complex analysis. That is, the real part of our discrete holomorphic function typically lives on the primal vertices, while the imaginary part lives on the dual vertices. Given an orthodiagonal map $G=(V,E)$, a \textbf{contour} in $G$ is a sequence of directed edges $\gamma=(\vec{e_{1}}, \vec{e_{2}}, \vec{e_{3}},...,\vec{e_{m}})$ that forms a nearest neighbour path. That is, if $\vec{e_{i}}=(e_{i}^{-}, e_{i}^{+})$, then for each $i=1,2,3,...,m-1$, we have that $e_{i}^{+}=e_{i+1}^{-}$. A contour $\gamma=(\vec{e_{1}}, \vec{e_{2}}, \vec{e_{3}},...,\vec{e_{m}})$ is \textbf{closed} if the corresponding path starts and ends at the same point: $e_{m}^{+}=e_{1}^{-}$. Given a function $F:V\rightarrow{\mathbb{C}}$, its integral over a contour $\gamma$ is given by: 
\begin{equation*}
    \sum_{\substack{\vec{e}\in{\gamma} \\ \vec{e}=(e^{-},e^{+})}}\big(F(e^{-})+F(e^{+})\big)(e^{+}-e^{-})
\end{equation*}
It is not hard to see that if $F$ is discrete holomorphic, then $F$ integrates to $0$ over any closed contour in $G$: observe that $F$ satisfies the discrete Cauchy-Riemann equations for an inner face $Q$ of $G$ iff $F$ integrates to $0$ over the closed contour that traces out the boundary of $Q$, counterclockwise. From here, notice that any integral of $F$ over a closed contour $\gamma$ is an integer linear combination of contour integrals of $F$ over boundaries of inner faces of $G$ enclosed by $\gamma$. This is the discrete analogue of Morera's theorem. In particular, from Equation \ref{eqn: discrete Cauchy-Riemann} it is clear that the function $f(z)=z$ is discrete holomorphic.\\ \\
Orthodiagonal maps come with a natural set of edge weights $c:E^{\bullet}\sqcup{E^{\circ}}\rightarrow{(0,\infty)}$ defined as follows: 
\begin{align*}
    c(e^{\bullet}_{Q})&=\frac{|e^{\circ}_{Q}|}{|e^{\bullet}_{Q}|} & c(e^{\circ}_{Q})=\frac{|e^{\bullet}_{Q}|}{|e^{\circ}_{Q}|}
\end{align*}
where for any edge $e\in{E^{\bullet}\sqcup{E^{\circ}}}$, $|e|$ is the length of that edge. 
These edge weights were first introduced by Duffin \cite{D68} and independently by Dubejko \cite{D99}. Let $c^{\bullet}$ and $c^{\circ}$ be the restriction of $c$ to $E^{\bullet}$ and $E^{\circ}$ respectively. With these edge weights, $G^{\bullet}=(V^{\bullet},E^{\bullet},c^{\bullet})$ and $G^{\circ}=(V^{\circ}, E^{\circ}, c^{\circ})$ are finite networks. Let $\Delta^{\bullet}$ and $\Delta^{\circ}$ denote the discrete Laplacian on $G^{\bullet}$ and $G^{\circ}$ respectively. A function
$f:V^{\bullet}\rightarrow{\mathbb{R}}$ is \textbf{harmonic} if it is harmonic at all points $x\in{\text{Int}(V^{\bullet})}$. Harmonicity for functions on $G^{\circ}$ is defined analogously. By the discrete Cauchy-Riemann equations (\ref{eqn: discrete Cauchy-Riemann}), if $F:V\rightarrow{\mathbb{C}}$ is holomorphic, the real and imaginary parts of $F$ are both harmonic. That is, $Re(F)|_{V^{\bullet}}$, $Im(F)|_{V^{\bullet}}$ are both discrete harmonic functions on $G^{\bullet}$ and $Re(F)|_{V^{\circ}}$, $Im(F)|_{V^{\circ}}$ are both discrete harmonic functions on $G^{\circ}$. Conversely, if $h:V^{\bullet}\rightarrow{\mathbb{R}}$ is discrete harmonic, up to an additive constant, there is a unique function $\widetilde{h}:V^{\circ}\rightarrow{\mathbb{R}}$ so that the function $F=h+i\widetilde{h}:V\rightarrow{\mathbb{C}}$, whose restriction to $V^{\bullet}$ is given by $h$ and whose restriction to $V^{\circ}$ is given by $i\widetilde{h}$ is discrete holomorphic.   \\ \\ 
 We saw earlier that on any orthodiagonal map $G$, the function $f(z)=z$ is discrete holomorphic. Hence, the functions $z\mapsto{\text{Re}(z)}$ and $z\mapsto{\text{Im}(z)}$ are both discrete harmonic on $G^{\bullet}$ and $G^{\circ}$ and so simple random walks on both $G^{\bullet}$ and $G^{\circ}$ are martingales.

\section{The Main Result and an Outline of its Proof}
\label{sec: outline of main theorem proof}

 As we alluded to in Section \ref{subsec: A Polynomial Rate of Convergence for the Dirichlet Problem on Orthodiagonal Maps}, our goal is to prove the following: 
\begin{thm}
\label{thm: polynomial rate of convergence for Dirichlet problem on OD maps}
    Suppose $\Omega\subseteq{\mathbb{R}^{2}}$ is a bounded simply connected domain, $g$ is a $\alpha$-H\"older on $\mathbb{R}^{2}$ and $G=(V^{\bullet}\sqcup{V^{\circ}},E)$ is an orthodiagonal map with edges of length at most $\varepsilon$ so that for each point $z\in{\partial{\widehat{V}}}$, $\text{dist}(z,\partial{\Omega})\leq{\varepsilon}$. Let $h$ be the solution to the continuous Dirichlet problem on $\Omega$ with boundary data given by $g$: 
    \begin{align*}
        \Delta{h}(x)&=0 \hspace{24pt}\text{for all $x\in{\Omega}$} \\ 
        h(x)&=g(x) \hspace{10pt}\text{for all $x\in{\partial{\Omega}}$}
    \end{align*}  
    Let $h^{\bullet}$ be the solution to the discrete Dirichlet problem on $G^{\bullet}$ with boundary data given by $g$:
    \begin{align*}
        \Delta^{\bullet}{h^{\bullet}}(v)&=0 \hspace{24pt}\text{for all $v\in{\text{Int}(V^{\bullet})}$} \\ 
        h^{\bullet}(v)&=g(v) \hspace{10pt}\text{for all $v\in{\partial{V^{\bullet}}}$}
    \end{align*}  
    Let $\beta\in{(0,1)}$ be the absolute constant from Lemma \ref{lem: weak Beurling on OD maps}. Then for any $v\in{V^{\bullet}}$, we have that: 
    \begin{equation*}
        |h^{\bullet}(v)-h(v)|\leq{\begin{cases}
            \big(C_{1}\|g\| + C_{2}\frac{\beta}{\beta-\alpha}\|g\|_{\alpha}\text{diam}(\Omega)^{\alpha}\big)\big(\frac{\varepsilon}{\text{diam}(\Omega)}\big)^{\lambda(\alpha,\beta)} & \text{if $\alpha\in{(0,\beta)}$} \\ \\
            \big(C_{1}\|g\| + C_{2}\|g\|_{\alpha}\text{diam}(\Omega)^{\alpha}\big)\big(\log\big(\frac{\text{diam}(\Omega)}{\varepsilon}\big)\big)\big(\frac{\varepsilon}{\text{diam}(\Omega)}\big)^{\lambda(\alpha,\beta)} & \text{if $\alpha=\beta$} \\ \\
            \big(C_{1}\|g\| + C_{2}\frac{\alpha}{\beta-\alpha}\|g\|_{\alpha}\text{diam}(\Omega)^{\alpha}\big)\big(\frac{\varepsilon}{\text{diam}(\Omega)}\big)^{\lambda(\alpha,\beta)} & \text{if $\alpha\in{(\beta,1)}$}
        \end{cases}}
    \end{equation*}
    where $C_{1},C_{2}>0$ are absolute constants, $\|g\|_{\alpha}$ is the $\alpha$-H\"older norm of $g$: 
    \begin{equation*}
        \|g\|_{\alpha}=\sup\limits_{\substack{x,y\in{\mathbb{R}^{2}}\\ x\neq{y}}}\frac{|g(y)-g(x)|}{|x-y|^{\alpha}}
    \end{equation*}
     and the function $\lambda:(0,1)^{2}\rightarrow{(0,\infty)}$ is given by: 
    \begin{equation*}
        \lambda(\alpha,\beta)=\min_{r\in{[0,1]}}\max_{s\in{(r,1)}}\Xi(\alpha,\beta,r,s), 
    \end{equation*}
    where: 
    \begin{equation*}
        \Xi(\alpha,\beta,r,s)=
        \max\{(\alpha\wedge{\beta})r, \min\Big\{\beta(s-r),\frac{\beta}{1+\beta}-(2+\frac{\beta}{1+\beta})s+\frac{r}{2}, (\alpha\wedge{\beta})s\Big\}\}
    \end{equation*}
\end{thm}
\noindent The idea behind the proof is as follows. Fix a point $z\in{V^{\bullet}}$. Our goal is to show that $h(z)$ and $h^{\bullet}(z)$ are close. To do this, we will consider two cases.

\subsection{Proof Outline: $z$ Near the Boundary}
\label{subsec: outline of proof near boundary}

If $z$ is close to the boundary of $\Omega$ (and therefore the boundary of $\widehat{G}$), the H\"older regularity of $g$ tells us that, near the boundary, the value of $h(z)$ is close to the value of the boundary data $g$ at nearby points of $\partial{\Omega}$. Similarly, the H\"older regularity of $g$ tells us that the value of $h^{\bullet}(z)$ is close to the value of the boundary data $g$ at nearby points of $\partial{V^{\bullet}}$. Since $\partial{V^{\bullet}}$ is close to $\partial{\Omega}$ and $g$ is $\alpha$-H\"older, we conclude that $h(z)$ and $h^{\bullet}(z)$ are close. Thus, for $z$ near the boundary, the fact that $h(z)$ and $h^{\bullet}(z)$ are close follows from estimates for the modulus of continuity of solutions to the discrete and continuous Dirichlet problems with H\"older boundary data.   
\\ \\
The key ingredient in the proof of these modulus of continuity estimates for the continuous Dirichlet problem is the following estimate for planar Brownian motion: 
\begin{thm}
\label{thm: cts Beurling estimate}
(Beurling's Estimate) Let $B(0,R)\subseteq{\mathbb{R}^{2}}$ be the ball of radius $R$ about the origin in $\mathbb{R}^{2}$, $z\in{B(0,R)}$ and $K\subseteq{\mathbb{R}^{2}}$ is a connected, compact subset of the plane so that $0\in{K}$ and $K\cap{\partial{B(0,R)}}\neq{\emptyset}$. Let $(B_{t})_{t\geq{0}}$ be a planar Brownian motion and let $T_{\partial{B(0,R)}}$ and $T_{K}$ be the hitting times of $\partial{B(0,R)}$ and $K$ by this Brownian motion. Then: 
\begin{equation*}
    \mathbb{P}^{z}(T_{\partial{B(0,R)}}<T_{K})\leq{C\Big(\frac{|z|}{R}\Big)^{1/2}}
\end{equation*}
where $C>0$ is an absolute constant. 
\end{thm}
\noindent For a proof of Theorem \ref{thm: cts Beurling estimate}, see Section 3.8 of \cite{CIPbook}. In plain language, Theorem \ref{thm: cts Beurling estimate} gives us an upper bound for the probability that a planar Brownian motion, started at $z$, escapes the ball $B(0,R)$ without hitting $K$. 
This is known as the strong Beurling estimate. The word ``strong" here alludes to the fact that the exponent of $1/2$ in this theorem is sharp. This can be seen by evaluating the probability that Brownian motion started at $r\in{(0,1)}$ exits the unit disk before hitting the line segment $[-1,0]$. In fact, Theorem \ref{thm: cts Beurling estimate} is a direct consequence of a stronger result, known as the Beurling projection theorem (see Theorem 9.2 in Chapter III of \cite{GM05}), which tells us that given a Brownian motion started at a point $z\in{B(0,R)}$, a line segment stretching from $0$ to $Re^{-i\hspace{1pt}\text{arg}(z)}$ is the connected, compact set containing $0$ and intersecting $\partial{B(0,R)}$, that maximizes the probability that a planar Brownian motion, started at $z$, escapes the ball $B(0,R)$ without hitting $K$. An equivalent reformulation of the strong Beurling estimate is as follows: 
\begin{prop}
\label{prop: cts Beurling estimate}
Suppose $\Omega\subsetneq{\mathbb{C}}$ is a simply connected domain, $(B_{t})_{t\geq{0}}$ is a planar Brownian motion and $T_{\partial{\Omega}}$ is the hitting of $\partial{\Omega}$ by this planar Brownian motion. Then for any positive real number $r>0$ and any $z\in{\Omega}$ we have that: 
\begin{equation*}
    \mathbb{P}^{z}(|B_{T_{\partial{\Omega}}}-z|\geq{r})\leq{C\Big(\frac{d_{z}}{r}\Big)^{1/2}}
\end{equation*}
where $C>0$ is an absolute constant and $d_{z}=\text{dist}(z,\partial{\Omega})$.
\end{prop}
\noindent It is not hard to show that the statement of Proposition \ref{prop: cts Beurling estimate} is equivalent to Theorem \ref{thm: cts Beurling estimate}. Writing the strong Beurling estimate in this way will make it clearer how it is we are applying this result, when we use it to prove Theorem \ref{thm: polynomial rate of convergence for Dirichlet problem on OD maps}.\\ \\ Insofar as orthodiagonal maps are good approximations of continuous 2D space, an analogous estimate should be true for simple random walks on orthodiagonal maps. Indeed, in \cite{K87}, Kesten proves an analogue of the strong Beurling estimate for simple random walks on the square grid. For a nice exposition of this result, see Section 2.5 of \cite{IRWbook}. Later, in \cite{LL04}, Lawler and Limic prove an analogue of the strong Beurling estimate for a large class of random walks on periodic lattices. Establishing a strong Beurling estimate for general orthodiagonal maps, or even for the more restricted setting of isoradial graphs, is currently an open problem. However, in this work, we will show that simple random walks on orthodiagonal maps satisfy a weak Beurling estimate:
\begin{lem}
\label{lem: weak Beurling on OD maps}
(Weak Beurling Estimate) There exist absolute constants $C,K,\beta>0$ such that if $G=(V^{\bullet}\sqcup{V^{\circ}},E)$ is an orthodiagonal map with edges of length at most $\varepsilon$, $u\in{\text{Int}(V^{\bullet})}$, $r>0$ is a real number, $(S_{n})_{n\geq{0}}$ is a simple random walk on $G^{\bullet}$, and $\tau_{\partial{V^{\bullet}}}$ is the hitting time of $\partial{V^{\bullet}}$ by this random walk, then:
\begin{equation*}
    \mathbb{P}^{u}(|S_{\tau_{\partial{V^{\bullet}}}}-u|\geq{r})\leq{C\Big(\frac{d_{u}\vee{K\varepsilon}}{r}\Big)^{\beta}}
\end{equation*}
where $d_{u}=\text{dist}(u,\partial{V^{\bullet}})$.
\end{lem}
\noindent Just as in the continuous setting, this weak Beurling estimate gives us estimates for the modulus of continuity of solutions to the Dirichlet problem on orthodiagonal maps. The word ``weak" here refers to the fact that the exponent $\beta$ in this estimate is not sharp. As a consequence, we have the following weak Harnack-type estimate: 
\begin{lem}
\label{lem: discrete harmonic functions are beta Holder in the bulk}
There exist absolute constants $C',K,\beta>0$ so that if $G=(V^{\bullet}\sqcup{V^{\circ}},E)$ is an orthodiagonal map with edges of length at most $\varepsilon$, $h:V^{\bullet}\rightarrow{\mathbb{R}}$ is a harmonic function on $G^{\bullet}$ and $d=d_{x}\wedge{d_{y}}=\text{dist}(x,\partial{V^{\bullet}})\wedge{\text{dist}(y,\partial{V^{\bullet}})}$ for some vertices $x,y\in{V^{\bullet}}$, then:
\begin{equation*}
    |h(y)-h(x)|\leq{C'\hspace{1pt}\|h\|_{\infty}\Big(\frac{|y-x|\vee{K\varepsilon}}{d}\Big)^{\beta}}    
\end{equation*}
\end{lem}
\noindent The standard Harnack estimate for continuous harmonic functions can be interpreted as telling us that bounded harmonic functions are Lipschitz in the bulk (away from the boundary of our domain). Hence, our weak Harnack estimate in Lemma \ref{lem: discrete harmonic functions are beta Holder in the bulk} can be interpreted as saying that discrete harmonic functions on orthodiagonal maps are $\beta$-H\"older in the bulk.\\ \\
To our knowledge, Lemma \ref{lem: weak Beurling on OD maps} is not stated explicitly in this level of generality anywhere in the literature. However, it follows immediately from Lemma 6.7 of \cite{CLR23}. In the same paper, Chelkak, Laslier and Russkikh prove Lemma \ref{lem: discrete harmonic functions are beta Holder in the bulk} in far greater generality. Namely, Proposition 6.13 of \cite{CLR23} says that bounded discrete harmonic and discrete holomorphic functions on $t$- embeddings whose corresponding origami map is $\kappa$- Lipschitz on large scales for some $\kappa\in{(0,1)}$, are $\beta$- H\"older in the bulk for some absolute constant $\beta=\beta(\kappa)\in{(0,1)}$. The condition on the origami map here is known as ``$Lip(\kappa,\delta)$" where $\delta>0$ is the scale on which our origami map is $\kappa$- Lipschitz. The t- embeddings of \cite{CLR23} are a strict generalization of the setting we are working in. Namely, every orthodiagonal map is a t-embedding (for details, see Section 8.1 of \cite{CLR23}). Furthermore, it is not difficult to show that for any $\kappa\in{(0,1)}$, there exists an absolute constant $c=c(\kappa)>1$ so that if $G$ is an orthodiagonal map with edges of length at most $\varepsilon$, then $G$ satisfies the assumption ``$Lip(\kappa, c\varepsilon)$" of \cite{CLR23} (see Assumption 1.1 of \cite{CLR23}). Hence, all of the results of \cite{CLR23} that are proven for t-embeddings satisfying the assumption $Lip(\kappa,\delta)$ for some $\kappa\in{(0,1)}$ and $\delta>0$ transfer over immediately to the orthodiagonal setting. \\ \\
For the convenience of the reader, we provide self-contained proofs of Lemmas \ref{lem: weak Beurling on OD maps} and \ref{lem: discrete harmonic functions are beta Holder in the bulk} in Appendix \ref{sec: Random Walks and A Priori Regularity Estimates for Harmonic Functions on Orthodiagonal Maps}. Our arguments are identical to those in \cite{CLR23} and its predecessor \cite{Ch16}, with the notable exception that we use Theorem 1.1 of \cite{GJN20} to prove the analogue of Lemma 6.7 of \cite{CLR23} in the orthodiagonal setting. 

\subsection{Proof Outline: $z$ Away from the Boundary}
\label{subsec: proof outline away from the boundary}

Away from the boundary, we use the same strategy used by Chelkak to prove Theorem 4.1 of \cite{Ch19}. This theorem establishes a polynomial rate of convergence, mesoscopically far away from the boundary,  for certain observables on s-embeddings satisfying the regularity conditions ``Unif($\delta$)" and ``Flat($\delta$)" (see Section 1.3 of \cite{Ch19} for details). To see this same argument, written out in the simpler setting of isoradial graphs, see Proposition 4.4.14 of \cite{BR21}. The idea is that if a function $f$ is almost harmonic in the sense that $\Delta{f}\approx{0}$, then $f$ is close to the harmonic function with the same boundary data. More precisely, suppose $\Omega\subseteq{\mathbb{R}^{2}}$ is a bounded, simply connected domain, $g\in{C^{0}(\mathbb{R}^{2})}$, and $h$ is the solution to the Dirichlet problem on $\Omega$ with boundary data given by $g$. That is: 
\begin{align*}
    \Delta{h}(x)&=0 \hspace{18pt}\text{for all $x\in{\Omega}$} \\ 
    h(x)&=g(x) \hspace{5pt}\text{for all $x\in{\partial{\Omega}}$}
\end{align*}
If $f$ is any other function in $C^{2}_{b}(\Omega)\cap{C(\overline{\Omega})}$ with the same boundary data,
\begin{equation}
\label{eqn: small Laplacian means almost harmonic}
|f(x)-h(x)|=\big|\int_{\Omega}\Delta{f}(y)G_{\Omega}(y,x)dA(y)\big|
\end{equation}
for any $x\in{\Omega}$, where $G_{\Omega}(\cdot,x)$ is the Green function on $\Omega$ centered at $x$, and ``$dA(y)$" is integration with respect to area on $\Omega$. As an immediate consequence of this formula, if $f$ has the same boundary behavior as $h$ and $\Delta{f}$ is small, $f$ must be close to $h$. \\ \\
To apply the formula in Equation \ref{eqn: small Laplacian means almost harmonic} to the problem of estimating the difference between $h(z)$ and $h^{\bullet}(z)$, we need to replace $h^{\bullet}$ with a smooth function. To this effect:  
\begin{enumerate}
    \item We convolve $h^{\bullet}$ with a smooth mollifier $\phi_{\delta}$, supported on a ball of radius $\delta$ about $0$, where $\delta$ is small.
    \item As long as $\delta\ll{d_{z}}$, Lemma \ref{lem: discrete harmonic functions are beta Holder in the bulk} tells us that $h^{\bullet}(z)$ is close to $\phi_{\delta}\ast{h^{\bullet}}(z)$.
    \item Since $h^{\bullet}$ is discrete harmonic, the convolution $\phi_{\delta}\ast{h^{\bullet}}$ is almost harmonic in the sense that $\Delta(\phi_{\delta}\ast{h^{\bullet}})\approx{0}$. Verifying this is the most subtle part of the whole argument. This is the subject of Section \ref{sec: The Convolution of a Discrete Harmonic Function with a Smooth Mollifier is Almost Harmonic}.
    \item  From here, we are in a position to apply our formula in Equation \ref{eqn: small Laplacian means almost harmonic} to show that the solution to our continuous Dirichlet problem, $h$, is close to $\phi_{\delta}\ast{h^{\bullet}}$ and therefore $h^{\bullet}$.
\end{enumerate}   
There is a minor technicality that $\phi_{\delta}\ast{h^{\bullet}}$ is not defined on all of $\Omega$, so this argument actually plays out on some smaller subdomain of $\Omega$. Furthermore, the boundary behavior of $\phi\ast{h^{\bullet}}$ doesn't quite agree with the boundary behavior of $h$, but morally, this is what's going on.

\section{The Convolution of a Discrete Harmonic Function with a Smooth Mollifier is Almost Harmonic}
\label{sec: The Convolution of a Discrete Harmonic Function with a Smooth Mollifier is Almost Harmonic}

In the section, we will show that the convolution of a discrete harmonic function with a smooth mollifier has small Laplacian. This is the key estimate that will allow us to compare our mollified discrete harmonic function to the corresponding continuous harmonic function. We will also use this result in Section \ref{sec: Lipschitz Regularity on a Mesoscopic Scale for Harmonic Functions on Orthodiagonal Maps} to show that discrete harmonic functions are Lipschitz in the bulk on a mesoscopic scale. To prove this, we'll need the following analogue of Proposition 3.12 of \cite{Ch22}: 

\begin{prop}
\label{prop: continuous Laplacian is average of discrete Laplacian}
    Suppose $S\subseteq{\mathbb{R}^{2}}$ is a square with side length $l$, $f\in{C_{b}^{3}(S)}$, and $G=(V^{\bullet}\sqcup{V^{\circ}}, E)$ is an orthodiagonal map with edges of length at most $\varepsilon$ such that $S\subseteq{\widehat{G}}$ and $l\gtrsim{\varepsilon}$. Then: 
    $$
    \sum_{v\in{V^{\bullet}\cap{S}}}\Delta^{\bullet}f(v)=\int_{S}\Delta{f}(x)dA(x)+O(\varepsilon\cdot{l}\cdot\|D^{2}f\|)+O(l^{3}\cdot\|D^{3}f\|)
    $$
\end{prop}
\noindent For context, notice that if our orthodiagonal map $G=(V^{\bullet}\sqcup{V^{\circ}},E)$ is isoradial, with edges of length at most $\varepsilon$ and $f\in{C^{3}_{b}(\widehat{G})}$, a straightforward computation (see Lemma 2.2 of \cite{CS11}) tells us that for any vertex $x\in{\text{Int}(V^{\bullet})}$, we have that:
\begin{align}
    \Delta^{\bullet}f(x)&=\Delta{f}(x)\cdot\text{Area}(N_{x})+O(\varepsilon\cdot\text{Area}(N_{x})\cdot\|D^{3}f\|) \\ &=\Delta{f}(x)\cdot\text{Area}(N_{x})+O(\varepsilon^{3}\cdot\|D^{3}f\|) \label{eqn: discrete and cts Laplacian agree for isoradial graphs}
\end{align}
where:
\begin{equation*}
    \text{Area}(N_{x})=\frac{1}{2}\sum_{\substack{y\in{V^{\bullet}} \\ y\sim{x}}}\text{Area}(Q_{\{x,y\}})
\end{equation*}
Recall that for any pair of neighboring vertices $x,y\in{V^{\bullet}}$, $Q_{\{x,y\}}$ is the face of $G$ with primal diagonal $\{x,y\}$. The factor of $\frac{1}{2}$ here, comes from the fact that every inner face of $G$ has two primal vertices, so it is natural that the area of this face should be split evenly between them. Equation \ref{eqn: discrete and cts Laplacian agree for isoradial graphs} tells us that on isoradial graphs, at any vertex, the discrete Laplacian of a smooth function looks like the continuous Laplacian. In particular, any continuous harmonic function defined in a neighbourhood of $\widehat{G}$ is ``almost" discrete harmonic in that its discrete Laplacian is small. \\ \\
Repeating this computation for a general orthodiagonal map, one finds that this result is no longer true: at any fixed vertex, the discrete Laplacian does not look like the continuous Laplacian. However, Proposition \ref{prop: continuous Laplacian is average of discrete Laplacian} tells us that, at least on average, the discrete Laplacian does indeed look like the continuous Laplacian.

\begin{proof}
    By Lemma 4.3 of \cite{BP24}, we can pick a suborthodiagonal map $S'$ of G so that: 
    \begin{itemize}
        \item $d_{\text{Haus}}(\partial{\widehat{S'}},\partial{S})=O(\varepsilon)$.
        \item If $x_{1}, x_{2}, ..., x_{m}$ are the vertices of $\partial{V_{S'}^{\circ}}$ listed in counterclockwise order, then these vertices form a contour. That is, $\{x_{i}, x_{i+1}\}$ is an edge of $(S')^{\circ}$ for all $i$, where the indices $i$ are being considered modulo $m$. It follows that if $x\in{\partial{V_{S'}^{\bullet}}}$, then $x$ has exactly one neighbouring vertex $y$ in $(S')^{\bullet}$, where $y\in{\text{Int}(V_{S'}^{\bullet})}$. 
        \item For $x_{1}, x_{2}, ..., x_{m}$ as above, $\sum\limits_{i=1}^{m}|x_{i+1}-x_{i}|\asymp{l}$.
    \end{itemize}
    In other words, $S'$ is an orthodiagonal approximation to $S$ that is close in Hausdorff distance and whose perimeter (at least in the dual lattice) is comparable to that of $S$. We will see why this is important later. Then: 
    \begin{align*}
    \sum_{v\in{V^{\bullet}\cap{S}}}\Delta^{\bullet}f(v)&=\sum_{v\in{V^{\bullet}\cap{S}}}\sum_{u:u\sim{v}}c(u,v)\big(f(u)-f(v))\big) \\ &=\sum_{v\in{V^{\bullet}\cap{S}}}\sum_{u:u\sim{v}}c(u,v)\big(\nabla{f}(v)(u-v)+\frac{1}{2}(u-v)^{T}D^{2}f(v)(u-v)+O(\|D^{3}{f}\|\cdot|u-v|^{3})\big) 
    \end{align*}
    Since $\nabla{f}(v)\cdot{u}$ is a linear function of $u$ for fixed $v$ and linear functions are discrete harmonic, 
$$
\sum_{u:u\sim{v}}c(u,v)\nabla{f}(v)\cdot(u-v)=0
$$
for all $v\in{\text{Int}(V^{\bullet})}$. Since our quadrilaterals all have orthogonal diagonals, given a quadrilateral $Q=[u,r,v,s]$ in $G$, we have that:
$$
\text{Area}([u,r,v,s])=\frac{1}{2}|u-v||r-s|=\frac{1}{2}c(u,v)|u-v|^{2}
$$
With this in mind, 
\begin{align*}
    &\sum_{v\in{V^{\bullet}\cap{S}}}\sum_{u:u\sim{v}}c(u,v)O(\|D^{3}{f}\|\cdot|u-v|^{3})=\sum_{v\in{V^{\bullet}\cap{S}}}\sum_{u:u\sim{v}}O(\|D^{3}{f}\|\cdot\text{Area}(Q_{u,v})\cdot|u-v|) \\ 
    =&O(\|D^{3}f\|\cdot{\varepsilon}\sum_{v\in{V^{\bullet}\cap{S}}}\sum_{u:u\sim{v}}\text{Area}(Q_{u,v}))\stackrel{(*)}{=}O(\|D^{3}f\|\cdot{\varepsilon}\cdot\text{Area}(S))=O(\|D^{3}f\|\cdot{\varepsilon}\cdot{l^{2}}) 
\end{align*}
The equality, $(*)$, follows from the fact that any quadrilateral of $G$ that lies within $\varepsilon$ of $S$ has at most two corresponding primal vertices in $V^{\bullet}\cap{S}$ and so is counted at most twice in our sum. Thus:
$$
\sum_{v\in{V^{\bullet}\cap{S}}}\Delta^{\bullet}f(v)=\frac{1}{2}\sum_{v\in{V^{\bullet}\cap{S}}}\sum_{u:u\sim{v}}c(u,v)\big((u-v)^{T}D^{2}f(v)(u-v)\big)+O(\|D^{3}f\|\cdot{\varepsilon}\cdot{l^{2}}) 
$$
By the same reasoning,
$$
\sum_{v\in{\text{Int}(V^{\bullet}_{S'})}}\Delta^{\bullet}f(v)=\frac{1}{2}\sum_{v\in{\text{Int}(V^{\bullet}_{S'})}}\sum_{u:u\sim{v}}c(u,v)\big((u-v)^{T}D^{2}f(v)(u-v)\big)+O(\|D^{3}f\|\cdot{\varepsilon}\cdot{l^{2}}) 
$$
Hence: 
\begin{align*}
    \big|\sum_{v\in{V^{\bullet}\cap{S}}}\Delta^{\bullet}f(v)-\sum_{v\in{\text{Int}(V^{\bullet}_{S'})}}\Delta^{\bullet}f(v)\big|&=\frac{1}{2}\sum_{\substack{v\in{V^{\bullet}\cap{S}} \\ v\notin{\text{Int}(V^{\bullet}_{S'})}}}\sum_{u:u\sim{v}}c(u,v)\big((u-v)^{T}D^{2}f(v)(u-v)\big)+O(\|D^{3}f\|\cdot{\varepsilon}\cdot{l^{2}}) \\ 
    &=O\big(\|D^{2}f\|\sum_{\substack{v\in{V^{\bullet}\cap{S}} \\ v\notin{\text{Int}(V^{\bullet}_{S'})}}}\sum_{u:u\sim{v}}c(u,v)|u-v|^{2})\big)+O(\|D^{3}f\|\cdot{\varepsilon}\cdot\text{Area}(S)) \\ 
    &=O\big(\varepsilon\cdot{l}\cdot\|D^{2}f\|\big)+O(\|D^{3}f\|\cdot{\varepsilon}\cdot{l^{2}})
\end{align*}
In other words, we see that we can approximate the sum of $\Delta^{\bullet}f$ over $V^{\bullet}\cap{S}$ by the sum of $\Delta^{\bullet}f$ over $\text{Int}(V^{\bullet}_{S'})$, and the error we incur when we do this is small. In summa, we have that:
$$
\sum_{v\in{V^{\bullet}\cap{S}}}\Delta^{\bullet}f(v)=\frac{1}{2}\sum_{v\in{\text{Int}(V^{\bullet}_{S'})}}\sum_{u:u\sim{v}}c(u,v)\big((u-v)^{T}D^{2}f(v)(u-v)\big)+O\big(\varepsilon\cdot{l}\cdot\|D^{2}f\|\big)+O(\varepsilon\cdot{l^{2}}\cdot\|D^{3}f\|)
$$
For any $v\in{V^{\bullet}\cap{S}}$, Taylor- expanding the second derivatives of $f$ about $z_{S}$, the center of the square $S$, we have that: 
\begin{align*}
    \partial_{1}^{2}f(v)&=\partial_{1}^{2}f(z_{S})+O(l\cdot\|D^{2}f\|) \\ 
    \partial_{1}\partial_{2}f(v)&=\partial_{1}\partial_{2}f(z_{S})+O(l\cdot\|D^{2}f\|) \\
    \partial_{2}^{2}f(v)&=\partial_{2}^{2}f(z_{S})+O(l\cdot\|D^{2}f\|)
\end{align*}
Applying these estimates to the sum above (effectively, we are treating the second derivatives of $f$ as roughly constant on each square) we have that: 
\begin{align*}
    &\sum_{v\in{\text{Int}(V^{\bullet}_{S'})}}\sum_{u:u\sim{v}}c(u,v)\big((u-v)^{T}D^{2}f(v)(u-v)\big)= \\ &=\sum_{v\in{\text{Int}(V^{\bullet}_{S'})}}\sum_{u:u\sim{v}}c(u,v)\big(\partial_{1}^{2}f(v)(u_{1}-v_{1})^{2}+2\partial_{1}\partial_{2}f(v)(u_{1}-v_{1})(u_{2}-v_{2})+\partial_{2}^{2}f(v)(u_{2}-v_{2})^{2}\big)= \\ 
    &=\sum_{v\in{\text{Int}(V^{\bullet}_{S'})}}\partial_{1}^{2}f(v)\sum_{u:u\sim{v}}c(u,v)(u_{1}-v_{1})^{2} + 2\sum_{v\in{\text{Int}(V^{\bullet}_{S'})}}\partial_{1}\partial_{2}f(v)\sum_{u:u\sim{v}}c(u,v)(u_{1}-v_{1})(u_{2}-v_{2}) \\ 
    &\hspace{15pt}+\sum_{v\in{\text{Int}(V^{\bullet}_{S'})}}\partial_{2}^{2}f(v)\sum_{u:u\sim{v}}c(u,v)(u_{2}-v_{2})^{2}= \\
    &=\partial_{1}^{2}f(z_{S})\Big(\sum_{v\in{\text{Int}(V^{\bullet}_{S'})}}\sum_{u:u\sim{v}}c(u,v)(u_{1}-v_{1})^{2}\Big)+O\big(l\cdot{\|D^{3}f\|}\sum_{v\in{\text{Int}(V^{\bullet}_{S'})}}\sum_{u:u\sim{v}}c(u,v)(u_{1}-v_{1})^{2}\big) \\ 
     &\hspace{15pt}+2\partial_{1}\partial_{2}f(z_{S})\Big(\sum_{v\in{\text{Int}(V^{\bullet}_{S'})}}\sum_{u:u\sim{v}}c(u,v)(u_{1}-v_{1})(u_{2}-v_{2})\Big) \\ &\hspace{15pt}+O\Big(l\cdot{\|D^{3}f\|}\sum_{v\in{\text{Int}(V^{\bullet}_{S'})}}\sum_{u:u\sim{v}}c(u,v)|u_{1}-v_{1}||u_{2}-v_{2}|\Big) \\ 
     &\hspace{15pt}+\partial_{2}^{2}f(z_{S})\Big(\sum_{v\in{\text{Int}(V^{\bullet}_{S'})}}\sum_{u:u\sim{v}}c(u,v)(u_{2}-v_{2})^{2}\Big)+O\big(l\cdot{\|D^{3}f\|}\sum_{v\in{\text{Int}(V^{\bullet}_{S'})}}\sum_{u:u\sim{v}}c(u,v)(u_{2}-v_{2})^{2}\big)
\end{align*}
Observe that $(u_{1}-v_{1})^{2}, (u_{2}-v_{2})^{2}, |u_{1}-v_{1}||u_{2}-v_{2}|\leq{|u-v|^{2}}$. Furthermore, using the fact that $c(u,v)|u-v|^{2}=2\hspace{1pt}\text{Area}(Q_{u,v})$ and observing that every quadrilateral $Q$ within $O(\varepsilon)$ of $S$ is counted at most twice in our sum, we have that each of our error terms is of size at most $O(l^{3}\cdot{\|D^{3}f\|})$. Thus: 
\begin{align*}
    &\sum_{v\in{\text{Int}(V^{\bullet}_{S'})}}\sum_{u:u\sim{v}}c(u,v)\big((u-v)^{T}D^{2}f(v)(u-v)\big)= \\
    &=\partial_{1}^{2}f(z_{S})\Big(\underbrace{\sum_{v\in{\text{Int}(V^{\bullet}_{S'})}}\sum_{u:u\sim{v}}c(u,v)(u_{1}-v_{1})^{2}}_{(1)}\Big)+2\partial_{1}\partial_{2}f(z_{S})\Big(\underbrace{\sum_{v\in{\text{Int}(V^{\bullet}_{S'})}}\sum_{u:u\sim{v}}c(u,v)(u_{1}-v_{1})(u_{2}-v_{2})}_{(2)}\Big) \\ 
    &\hspace{15pt}+\partial_{2}^{2}f(z_{S})\Big(\underbrace{\sum_{v\in{\text{Int}(V^{\bullet}_{S'})}}\sum_{u:u\sim{v}}c(u,v)(u_{2}-v_{2})^{2}}_{(3)}\Big)+O(l^{3}\cdot{\|D^{3}f\|})
\end{align*}
To complete the proof of Proposition \ref{prop: continuous Laplacian is average of discrete Laplacian}, we need to understand the behavior of terms (1), (2) and (3). The idea is to use discrete integration by parts to show that each of (1), (2), (3) is equal to the discretization of a certain contour integral. The fact that the perimeter of $S'$ in the dual lattice is comparable to that of $S$ will allow us to show that this discrete contour integral is close to the corresponding continuous contour integral. Consider term (1). Rewriting this as a sum over directed edges we have that: 
\begin{align*}
    &\sum_{v\in{\text{Int}(V^{\bullet}_{S'})}}\sum_{u:u\sim{v}}c(u,v)(u_{1}-v_{1})^{2}=\sum_{\substack{\vec{e}\in{\vec{E^{\bullet}_{S'}}} \\ e^{-}\in{\text{Int}(V^{\bullet}_{S'})}}}c(e)(e^{+}_{1}-e^{-}_{1})^{2}= \\ 
    &=\underbrace{\sum_{\vec{e}\in{\vec{E^{\bullet}_{S'}}}}c(e)(e^{+}_{1}-e^{-}_{1})^{2}}_{(1a)}+\underbrace{\sum_{\substack{\vec{e}\in{\vec{E}}\\ e^{-}\in{\partial{V_{S'}^{\bullet}}}}}c(e)(e_{1}^{+}-e_{1}^{-})^{2}}_{(1b)}
\end{align*}
Term (1a) is just the discrete Dirichlet energy of the function $z\mapsto{\text{Re}(z)}$ on $S'$. Since $\text{Re}(z)$ is discrete harmonic, applying discrete integration by parts we have that: 
\begin{align*}
    \sum_{\vec{e}\in{\vec{E^{\bullet}_{S'}}}}c(e)(e^{+}_{1}-e^{-}_{1})^{2}&= \sum_{\vec{e}\in{\vec{E^{\bullet}_{S'}}}}c(e)e_{1}^{+}(e_{1}^{+}-e_{1}^{-})-\sum_{\vec{e}\in{\vec{E^{\bullet}_{S'}}}}e_{1}^{-}(e_{1}^{+}-e_{1}^{-}) \\ 
    &=\sum_{\vec{e}\in{\vec{E^{\bullet}_{S'}}}}c(e)e_{1}^{+}(e_{1}^{+}-e_{1}^{-})-\sum_{\vec{e}\in{\vec{E^{\bullet}_{S'}}}}e_{1}^{+}(e_{1}^{-}-e_{1}^{+}) \\
    &=2\sum_{\vec{e}\in{\vec{E^{\bullet}_{S'}}}}c(e)e_{1}^{+}(e_{1}^{+}-e_{1}^{-})=-2\sum_{\vec{e}\in{\vec{E^{\bullet}_{S'}}}}c(e)e_{1}^{-}(e_{1}^{+}-e_{1}^{-}) \\
    &= -2\sum_{\substack{\vec{e}\in{\vec{E^{\bullet}_{S'}}}\\ e^{-}\in{\text{Int}(V^{\bullet}_{S'})}}}c(e)e_{1}^{-}(e_{1}^{+}-e_{1}^{-})-2\sum_{\substack{\vec{e}\in{\vec{E^{\bullet}_{S'}}}\\ e^{-}\in{\partial{V^{\bullet}_{S'}}}}}c(e)e_{1}^{-}(e_{1}^{+}-e_{1}^{-}) \\ 
    &=-2\sum_{v\in{\text{Int}(V^{\bullet}_{S'})}}\sum_{u:u\sim{v}}c(u,v)v_{1}(u_{1}-v_{1})-2\sum_{\substack{\vec{e}\in{\vec{E^{\bullet}_{S'}}}\\ e^{-}\in{\partial{V^{\bullet}_{S'}}}}}c(e)e_{1}^{-}(e_{1}^{+}-e_{1}^{-}) \\
    &=-2\sum_{v\in{\text{Int}(V_{S'}^{\bullet})}}\text{Re}(v)\Delta^{\bullet}(\text{Re}(\cdot))(v)-2\sum_{\substack{ \vec{e}\in{\vec{E^{\bullet}_{S'}}} \\ e^{-}\in{\partial{V_{S'}^{\bullet}}}}}c(e)\text{Re}(e^{-})(\text{Re}(e^{+})- \text{Re}(e^{-})) \\ 
    &=-2\sum_{\substack{ \vec{e}\in{\vec{E^{\bullet}_{S'}}} \\ e^{-}\in{\partial{V_{S'}^{\bullet}}}}}c(e)e^{-}_{1}(e^{+}_{1}-e^{-}_{1})
\end{align*}
Suppose $Q=[e^{-},f^{-}, e^{+}, f^{+}]$ is a quadrilateral face of $G$, where the vertices $e^{-}, f^{-}, e^{+}, f^{-}$ are listed in counterclockwise order. Since $\text{Im}(z)$ is the discrete harmonic conjugate of $\text{Re}(z)$ (up to an additive constant), we have that: 
$$
c(e)(e^{+}_{1}-e^{-}_{1})=(f_{2}^{+}-f_{2}^{-})
$$
More simply, since the diagonals of $Q$ are orthogonal, 
$$
\frac{e^{+}-e^{-}}{|e^{+}-e^{-}|}=-i\frac{f^{+}-f^{-}}{|f^{+}-f^{-}|} \hspace{5pt}\Longleftrightarrow\hspace{5pt} \frac{e^{+}_{1}-e^{-}_{1}}{|e^{+}-e^{-}|}= \frac{f^{+}_{2}-f^{-}_{2}}{|f^{+}-f^{-}|}, \hspace{5pt} \frac{e^{+}_{2}-e^{-}_{2}}{|e^{+}-e^{-}|}= -\frac{f^{+}_{1}-f^{-}_{1}}{|f^{+}-f^{-}|}
$$
Since $c(e)=\frac{|f^{+}-f^{-}|}{|e^{+}-e^{-}|}$, $f_{1}^{+}-f_{1}^{-}=-c(e)(e^{+}_{2}-e^{-}_{2})$, $f^{+}_{2}-f^{-}_{2}=c(e)(e^{+}_{1}-e^{-}_{1})$. Applying this to the problem at hand, we have that: 
$$
-2\sum_{\substack{ \vec{e}\in{\vec{E^{\bullet}_{S'}}} \\ e^{-}\in{\partial{V_{S'}^{\bullet}}}}}c(e)e^{-}_{1}(e^{+}_{1}-e^{-}_{1})=-2\sum_{\substack{ \vec{e}\in{\vec{E^{\bullet}_{S'}}} \\ e^{-}\in{\partial{V_{S'}^{\bullet}}} \\ Q_{e^{-},e^{+}}=[e^{-}, f^{-}, e^{+}, f^{+}]}}e_{1}^{-}(f_{2}^{+}-f_{2}^{-})
$$
Let $\partial{\vec{E}^{\bullet}_{S'}}$ denote the set of directed edges of $S'$ that go from a vertex of $\partial{V^{\bullet}_{S'}}$ to a vertex of $\text{Int}(V^{\bullet}_{S'})$. Because of how we defined $S'$, for all directed edges $\vec{e}=(e^{-},e^{+})$ under consideration in the sum above, $e^{-}\in\partial{V_{S'}^{\bullet}}$, $e^{+}\in{\text{Int}(V^{\bullet}_{S'})}$ and $f^{-},f^{+}\in{\partial{V^{\circ}_{S'}}}$ where $Q_{e^{-},e^{+}}=[e^{-}, f^{-}, e^{+}, f^{+}]$. Furthermore, the directed edges $\vec{f}=(f^{-},f^{+})$ dual to the directed edges $\vec{e}\in{\partial{E^{\bullet}_{S'}}}$, form a closed contour in $(S')^{\circ}$, oriented clockwise, with length is comparable to $l$. We use $\partial^{\circlearrowright}{E^{\circ}_{S'}}$ to denote the set of directed edges in this contour. Intuitively, the sum above is a discretization of the integral of $xdy$ over the contour $\partial^{\circlearrowright}{E^{\circ}_{S'}}$. We now make this intuition precise. \\ \\
Suppose $\vec{f}\in{\partial^{\circlearrowright}{E^{\circ}_{S'}}}$ and $Q_{f^{-},f^{+}}=[e^{-},f^{-},e^{+},f^{+}]$. Then: 
$$
|e_{1}^{-}(f_{2}^{+}-f_{2}^{-})-\oint_{f^{-}}^{f^{+}}xdy|\leq{\varepsilon|f^{+}-f^{-}|}
$$
Summing over directed edges $\vec{f}\in{\partial^{\circlearrowright}{E^{\circ}_{S'}}}$,
\begin{align*}
    \big|\sum_{\substack{ \vec{e}\in{\vec{E^{\bullet}_{S'}}} \\ e^{-}\in{\partial{V_{S'}^{\bullet}}} \\ Q_{e^{-},e^{+}}=[e^{-}, f^{-}, e^{+}, f^{+}]}}e_{1}^{+}(f_{2}^{+}-f_{2}^{-})-\oint_{\partial^{\circlearrowright}{E^{\circ}_{S'}}}xdy\big|\leq{\varepsilon\cdot\text{length}(\partial^{\circlearrowright}{E^{\circ}_{S'}})}\lesssim{\varepsilon\cdot{l}}
\end{align*}
Thus, we see that term (1a) is close to the contour integral $\oint_{\partial^{\circlearrowright}{E^{\circ}_{S'}}}xdy$. By Green's theorem: 
$$
\oint_{\partial^{\circlearrowright}{E^{\circ}_{S'}}}xdy=-l^{2}+O(\varepsilon\cdot{l})
$$
Thus: 
$$
(1a)= 2l^{2}+O(\varepsilon\cdot{l})
$$
Term (1b) can be dealt with in the same way as all the error terms we saw previously:
\begin{align*}
    \sum_{\substack{\vec{e}\in{\vec{E}}\\ e^{-}\in{\partial{V_{S'}^{\bullet}}}}}c(e)(e_{1}^{+}-e_{1}^{-})^{2}\leq{\sum_{\substack{\vec{e}\in{\vec{E}}\\ e^{-}\in{\partial{V_{S'}^{\bullet}}}}}c(e)|e^{+}-e^{-}|^{2}}=2\sum_{\substack{\vec{e}\in{\vec{E}}\\ e^{-}\in{\partial{V_{S'}^{\bullet}}}}}\text{Area}(Q_{e})=O(\varepsilon\cdot{l})
\end{align*}
Putting all this together, we get that: 
$$
(1)=2l^{2}+O(\varepsilon\cdot{l})
$$
A similar story plays out in the case of terms (2) and (3). Rewriting terms (2) and (3) as sums over directed edges we have that: 
\begin{align*}
    (2)&=\sum_{v\in{\text{Int}(V^{\bullet}_{S'})}}\sum_{u:u\sim{v}}c(u,v)(u_{1}-v_{1})(u_{2}-v_{2})=\sum_{\substack{\vec{e}\in{\vec{E^{\bullet}_{S'}}} \\ e^{-}\in{\text{Int}(V^{\bullet}_{S'})}}}c(e)(e^{+}_{1}-e^{-}_{1})(e^{+}_{2}-e^{-}_{2})= \\ 
    &=\underbrace{\sum_{\vec{e}\in{\vec{E^{\bullet}_{S'}}}}c(e)(e^{+}_{1}-e^{-}_{1})(e^{+}_{2}-e^{-}_{2})}_{(2a)}+\underbrace{\sum_{\substack{\vec{e}\in{\vec{E}}\\ e^{-}\in{\partial{V_{S'}^{\bullet}}}}}c(e)(e_{1}^{+}-e_{1}^{-})(e_{2}^{+}-e_{2}^{-})}_{(2b)}
\end{align*}
\begin{align*}
    (3)&=\sum_{v\in{\text{Int}(V^{\bullet}_{S'})}}\sum_{u:u\sim{v}}c(u,v)(u_{2}-v_{2})^{2}=\sum_{\substack{\vec{e}\in{\vec{E^{\bullet}_{S'}}} \\ e^{-}\in{\text{Int}(V^{\bullet}_{S'})}}}c(e)(e^{+}_{2}-e^{-}_{2})^{2}= \\ 
    &=\underbrace{\sum_{\vec{e}\in{\vec{E^{\bullet}_{S'}}}}c(e)(e^{+}_{2}-e^{-}_{2})^{2}}_{(3a)}+\underbrace{\sum_{\substack{\vec{e}\in{\vec{E}}\\ e^{-}\in{\partial{V_{S'}^{\bullet}}}}}c(e)(e_{2}^{+}-e_{2}^{-})^{2}}_{(3b)}
\end{align*}
By the same argument we used to handle term (1b), terms (2b) and (3b) are of size $O(\varepsilon\cdot{l})$. Applying discrete integration by parts to terms (2a) and (3a) and using the fact that $f^{+}_{1}-f^{-}_{1}=-c(e)(e^{+}_{2}-e^{-}_{2})$ and $f^{+}_{2}-f^{-}_{2}=c(e)(e^{+}_{1}-e^{-}_{1})$ for any quadrilateral $Q=[e^{-},f^{-},e^{+},f^{+}]$ in $G$, we have that: 
\begin{align*}
    (2a)&=\sum_{\vec{e}\in{\vec{E^{\bullet}_{S'}}}}c(e)(e^{+}_{1}-e^{-}_{1})(e^{+}_{2}-e^{-}_{2})=-2\sum_{v\in{\text{Int}(V^{\bullet}_{S'})}}\text{Re}(v)\Delta^{\bullet}(\text{Im}(\cdot))(v)-2\sum_{\substack{\vec{e}\in{\vec{E^{\bullet}_{S'}}} \\ e^{-}\in{\partial{V^{\bullet}_{S'}}}}}c(e)e_{1}^{-}(e_{2}^{+}-e_{2}^{-}) \\
    &=-2\sum_{\substack{\vec{e}\in{\vec{E^{\bullet}_{S'}}} \\ e^{-}\in{\partial{V^{\bullet}_{S'}}}}}c(e)e_{1}^{-}(e_{2}^{+}-e_{2}^{-})=2\sum_{\substack{\vec{f}\in{\partial^{\circlearrowright}{E^{\circ}_{S'}}}\\ Q_{f^{-},f^{+}}=[e^{-},f^{-},e^{+},f^{+}]}}e_{1}^{-}(f_{1}^{+}-f_{1}^{-})
\end{align*}
\begin{align*}
    (3a)&=\sum_{\vec{e}\in{\vec{E^{\bullet}_{S'}}}}c(e)(e^{+}_{2}-e^{-}_{2})^{2}=-2\sum_{v\in{\text{Int}(V^{\bullet}_{S'})}}\text{Im}(v)\Delta^{\bullet}(\text{Im}(\cdot))(v)-2\sum_{\vec{e}\in{\vec{E^{\bullet}_{S'}}}}c(e)e_{2}^{-}(e_{2}^{+}-e_{2}^{-}) \\
    &=-2\sum_{\vec{e}\in{\vec{E^{\bullet}_{S'}}}}c(e)e_{2}^{-}(e_{2}^{+}-e_{2}^{-})=2\sum_{\substack{\vec{f}\in{\partial^{\circlearrowright}{E^{\circ}_{S'}}}\\ Q_{f^{-},f^{+}}=[e^{-},f^{-},e^{+},f^{+}]}}e_{2}^{-}(f_{1}^{+}-f_{1}^{-})
\end{align*}
where: 
\begin{equation*}
    \big|\sum_{\substack{\vec{f}\in{\partial^{\circlearrowright}{E^{\circ}_{S'}}}\\ Q_{f^{-},f^{+}}=[e^{-},f^{-},e^{+},f^{+}]}}e_{1}^{-}(f_{1}^{+}-f_{1}^{-})-\oint_{\partial^{\circlearrowright}{E^{\circ}_{S'}}}xdx\big|\lesssim{\varepsilon\cdot{l}}
\end{equation*}
\begin{equation*}
    \big|\sum_{\substack{\vec{f}\in{\partial^{\circlearrowright}{E^{\circ}_{S'}}}\\ Q_{f^{-},f^{+}}=[e^{-},f^{-},e^{+},f^{+}]}}e_{2}^{-}(f_{1}^{+}-f_{1}^{-})-\oint_{\partial^{\circlearrowright}{E^{\circ}_{S'}}}ydx\big|\lesssim{\varepsilon\cdot{l}}
\end{equation*}
By Green's theorem: 
\begin{equation*}
    \oint_{\partial^{\circlearrowright}{E^{\circ}_{S'}}}xdx=0
\end{equation*}
\begin{equation*}
    \oint_{\partial^{\circlearrowright}{E^{\circ}_{S'}}}ydx=l^{2}+O(\varepsilon\cdot{l})
\end{equation*}
And so: 
\begin{equation*}
    (2)=O(\varepsilon\cdot{l}), \hspace{25pt} (3)=2l^{2}+O(\varepsilon\cdot{l})
\end{equation*}
Armed with these estimates for terms (1), (2) and (3) in our sum from earlier, we have that: 
\begin{align*}
\sum_{V^{\bullet}\cap{S}}\Delta^{\bullet}f(v)&=\partial_{1}^{2}f(z_{S})\big(l^{2}+O(\varepsilon\cdot{l})\big)+\partial_{1}\partial_{2}f(z_{S})\cdot{O(\varepsilon\cdot{l})}+\partial_{2}^{2}f(z_{S})(l^{2}+O(\varepsilon\cdot{l}))+O(l^{3}\|D^{3}f\|) \\ 
&=\Delta{f}(z_{S})\cdot{l^{2}}+O(\varepsilon\cdot{l}\cdot\|D^{2}f\|)+O(l^{3}\cdot\|D^{3}f\|) \\
&=\int_{S}\Delta{f}(x)dA(x)+O(\varepsilon\cdot{l}\cdot\|D^{2}f\|)+O(l^{3}\cdot\|D^{3}f\|)
\end{align*}
\end{proof}
\noindent Having shown that for a smooth function, the continuous Laplacian agrees with the discrete Laplacian, on average, we are ready to prove the main result of this section: 
\begin{prop}
\label{prop: convolution has small Laplacian}
Suppose $G=(V^{\bullet}\sqcup{V^{\circ}}, E)$ is an orthodiagonal map with edges of length at most $\varepsilon$, $h^{\bullet}:V^{\bullet}\rightarrow{\mathbb{R}}$ is harmonic, and $\varphi$ is a smooth mollifier supported on the unit ball $B(0,1)\subseteq{\mathbb{R}^{2}}$. We can think of $h^{\bullet}$ as a function on $\widehat{G}$ by extending $h^{\bullet}$ to $\widehat{G}$ in any sort of sensible way. For instance, we could triangulate the faces of $G^{\bullet}$ and then define $h^{\bullet}$ on each face of $G^{\bullet}$ by linear interpolation. By our a priori regularity estimates for discrete harmonic functions on orthodiagonal maps (see Lemma \ref{lem: discrete harmonic functions are beta Holder in the bulk}), the exact details of how we choose to extend $h^{\bullet}$ to $\widehat{G}$ don't matter. \\ \\  For any $\delta>0$, define $\varphi_{\delta}(z):=\delta^{-2}\varphi(\delta^{-1}z)$. Then $\varphi_{\delta}$ is a smooth mollifier supported on the ball $B(0,\delta)\subseteq{\mathbb{R}^{2}}$. Fix $\delta>0$, $z\in{\widehat{G}}$ so that $\varepsilon\lesssim\delta\leq{\frac{d}{2}}$ where $d=d_{z}=\text{dist}(z,\partial{\widehat{G}})$. Then: 
\begin{equation*}
    \Delta(\varphi_{\delta}\ast{h^{\bullet}})(z)=O(\varepsilon^{\frac{1}{2}}\cdot\delta^{-\frac{5}{2}}\cdot(\|D^{2}\varphi\|+\|D^{3}\varphi\|)\cdot{\|h^{\bullet}\|})+O(\varepsilon^{\frac{\beta}{1+\beta}}\cdot\delta^{-2}\cdot{d^{-\frac{\beta}{1+\beta}}}\cdot(\|D^{2}\varphi\|+\|D^{3}\varphi\|)\cdot{\|h^{\bullet}\|})
\end{equation*}
In particular, if $\delta\geq{\varepsilon^{\frac{1-\beta}{1+\beta}}d^{\frac{2\beta}{1+\beta}}}$, we have that:
\begin{equation*}
    \Delta(\varphi_{\delta}\ast{h^{\bullet}})(z)=O(\varepsilon^{\frac{1}{2}}\cdot\delta^{-\frac{5}{2}}\cdot(\|D^{2}\varphi\|+\|D^{3}\varphi\|)\cdot{\|h^{\bullet}\|})
\end{equation*}
Otherwise: 
\begin{equation*}
    \Delta(\varphi_{\delta}\ast{h^{\bullet}})(z)=O(\varepsilon^{\frac{\beta}{1+\beta}}\cdot\delta^{-2}\cdot{d^{-\frac{\beta}{1+\beta}}}\cdot(\|D^{2}\varphi\|+\|D^{3}\varphi\|)\cdot{\|h^{\bullet}\|})
\end{equation*}
\end{prop}

\begin{proof}
    We compute: 
    \begin{align*}
        \Delta(\varphi_{\delta}\ast{h^{\bullet}})(z)&=\int_{\mathbb{R}^{2}}\Delta_{z}\varphi_{\delta}(z-w)h^{\bullet}(w)dA(w)=\int_{\mathbb{R}^{2}}\Delta_{w}\varphi_{\delta}(z-w)h^{\bullet}(w)dA(w) \\ &=\int_{B(z,\delta)}\Delta_{w}\varphi_{\delta}(z-w)h^{\bullet}(w)dA(w)
    \end{align*}
    Let $\mathcal{S}$ be a family of pairwise disjoint squares of side length $\ell$ that cover $B(z,\delta)$ up to a region of area $O(\delta{\ell})$. Here $\ell$ is a real parameter such that   $\varepsilon\ll{\ell}\ll{\delta}$, whose exact value will be specified later. Since each square has area $\ell^{2}$, $\mathcal{S}$ consists of $O((\delta/\ell)^{2})$ squares of side length $\ell$. Then: 
    \begin{align*}
        \Delta(\varphi_{\delta}\ast{h^{\bullet}})(z)=&\int_{B(z,\delta)}\Delta_{w}\varphi_{\delta}(z-w)h^{\bullet}(w)dA(w) \\ 
        =&\underbrace{\sum_{S\in{\mathcal{S}}}\Big(\int_{S}\Delta_{w}\varphi_{\delta}(z-w)h^{\bullet}(w)dA(w)\Big)}_{(1)}-\underbrace{\int_{(\cup_{S\in{\mathcal{S}}}S)\setminus{B(z,\delta)}}\Delta_{w}\varphi_{\delta}(z-w)h^{\bullet}(w)dA(w)}_{(2)}
    \end{align*}
    where: 
    \begin{equation*}
        |(2)|=O(\delta\cdot{\ell}\cdot{\|D^{2}\varphi_{\delta}\|\cdot{\|h^{\bullet}\|}})=O(\ell\cdot{\delta^{-3}}\cdot{\|D^{2}\varphi\|\cdot{\|h^{\bullet}\|}})
    \end{equation*}
    To handle term (1), fix $S\in{\mathcal{S}}$ and select a point $w_{S}\in{S\cap{V^{\bullet}}}$. Then we write: 
    \begin{equation}
    \label{summands in term (1) split in two}
        \int_{S}\Delta_{w}\varphi_{\delta}(z-w)h^{\bullet}(w)dA(w)=\underbrace{h^{\bullet}(w_{S})\int_{S}\Delta_{w}\varphi_{\delta}(z-w)dA(w)}_{(a)}+\underbrace{\int_{S}\Delta_{w}\varphi_{\delta}(z-w)\big(h^{\bullet}(w)-h^{\bullet}(w_{S})\big)dA(w)}_{(b)} 
    \end{equation}
    Since the square $S$ has area $\ell^{2}$ and Lemma \ref{lem: discrete harmonic functions are beta Holder in the bulk} tells us that $|h^{\bullet}(w)-h^{\bullet}(w_{S})|=O(\|h^{\bullet}\|\ell^{\beta}d^{-\beta})$ for any $w\in{S}$, it follows that: 
    \begin{equation*}
        (b)=\int_{S}\Delta_{w}\varphi_{\delta}(z-w)\big(h^{\bullet}(w)-h^{\bullet}(w_{S})\big)dA(w)=O(\ell^{2}\cdot{\|D^{2}\varphi_{\delta}\|}\cdot{\|h^{\bullet}\|(\ell/\delta)^{\beta}})=O(\ell^{2+\beta}\cdot{\delta^{-4}}\cdot{d^{-\beta}}\|D^{2}\varphi\|\cdot\|h^{\bullet}\|)
    \end{equation*}
    As far as handling term (a) in Equation \ref{summands in term (1) split in two}:
    \begin{align*}
        (a)=&h^{\bullet}(w_{S})\int_{S}\Delta_{w}\varphi_{\delta}(z-w)dA(w)\\ =&h^{\bullet}(w_{S})\Big(\Big(\sum_{w\in{V^{\bullet}\cap{S}}}\Delta_{w}^{\bullet}\varphi_{\delta}(z-w)\Big)+O(\varepsilon\cdot{\ell}\cdot{\|D^{2}\varphi_{\delta}\|})+O(\ell^{3}\cdot{\|D^{3}\varphi_{\delta}\|})\Big) \\
        =&\Big(\sum_{w\in{V^{\bullet}\cap{S}}}\Delta_{w}^{\bullet}\varphi_{\delta}(z-w)h^{\bullet}(w)\Big)+\Big(\sum_{w\in{V^{\bullet}\cap{S}}}\Delta_{w}^{\bullet}\varphi_{\delta}(z-w)\big(h^{\bullet}(w_{S}-h^{\bullet}(w)\big)\Big) \\ &+O(\varepsilon\cdot{\ell}\cdot{\delta^{-4}}\cdot{\|D^{2}\varphi\|}\cdot{\|h^{\bullet}\|})+O(\ell^{3}\cdot{\delta^{-5}}\cdot{\|D^{3}\varphi\|\cdot{\|h^{\bullet}\|}}) \\
        =&\Big(\sum_{w\in{V^{\bullet}\cap{S}}}\Delta_{w}^{\bullet}\varphi_{\delta}(z-w)h^{\bullet}(w)\Big)+O(\ell^{2+\beta}\cdot{\delta^{-4}}d^{-\beta}\cdot{\|D^{2}\varphi\|\cdot{\|h^{\bullet}\|}})\\ &+O(\varepsilon\cdot{\ell}\cdot{\delta^{-4}}\cdot{\|D^{2}\varphi\|}\cdot{\|h^{\bullet}\|})+O(\ell^{3}\cdot{\delta^{-5}}\cdot{\|D^{3}\varphi\|\cdot{\|h^{\bullet}\|}})
    \end{align*}
     Briefly, 
     \begin{itemize}
         \item the second equality follows from Proposition \ref{prop: continuous Laplacian is average of discrete Laplacian}.
         \item the fourth equality follows from the following crude estimate: 
         \begin{equation*}
             \Big|\sum_{w\in{V^{\bullet}\cap{S}}}\Delta_{w}^{\bullet}\varphi_{\delta}(z-w)\big(h^{\bullet}(w_{S}-h^{\bullet}(w)\big)\Big|\leq{\sum_{w\in{V^{\bullet}\cap{S}}}|\Delta_{w}^{\bullet}}\varphi_{\delta}(z-w)|\cdot\|h^{\bullet}\|\cdot\ell^{\beta}\cdot{d^{-\beta}}
         \end{equation*}
         where: 
         \begin{align*}
            |\Delta_{w}^{\bullet}\varphi_{\delta}(z-w)|&={\big|\sum_{\substack{u\in{V^{\bullet}} \\ u\sim{w}}}c(u,w)(\varphi_{\delta}(z-u)-\varphi_{\delta}(z-w))\big|}\\ &={\Big|\sum_{\substack{u\in{V^{\bullet}} \\ u\sim{w}}}c(u,w)\big(-\nabla{\varphi_{\delta}}(z-w)(w-u)+O(\|D^{2}\varphi_{\delta}\|\cdot|u-w|^{2})\big)\Big|}\\
            &=\Big|\sum_{\substack{u\in{V^{\bullet}} \\ u\sim{w}}}c(u,w)\nabla{\varphi_{\delta}}(z-w)(w-u)\Big|+O\Big(\|D^{2}\varphi_{\delta}\|\sum_{\substack{u\in{V^{\bullet}} \\ u\sim{w}}}\text{Area}(Q_{u,v})\Big) \\
            &=O\Big(\delta^{-4}\cdot\|D^{2}\varphi\|\sum_{\substack{u\in{V^{\bullet}} \\ u\sim{w}}}\text{Area}(Q_{u,v})\Big)
         \end{align*}
         The first term on the third line above vanishes, since any linear function is discrete harmonic. Summing over $w\in{V^{\bullet}\cap{S}}$, every quadrilateral face $Q$ of $G$ that lies within $\varepsilon$ of $S$ is counted at most twice, where $\varepsilon\ll{\ell}$. Hence: 
         \begin{equation*}
             \sum_{w\in{V^{\bullet}\cap{S}}}|\Delta_{w}^{\bullet}\varphi_{\delta}(z-w)|=O(\ell^{2}\cdot{\delta^{-4}}\cdot{\|D^{2}\varphi\|})
         \end{equation*}
     \end{itemize}
   Summing over $S\in{\mathcal{S}}$, since $\mathcal{S}$ consists of $O(\delta^{2}\cdot\ell^{-2})$ squares, we have that: 
    \begin{align*}
        (1)=&\sum_{S\in{\mathcal{S}}}\Big(\sum_{w\in{V^{\bullet}\cap{S}}}\Delta_{w}^{\bullet}\varphi_{\delta}(z-w)h^{\bullet}(w)\Big)+O(\varepsilon\cdot{\ell^{-1}}\cdot{\delta^{-2}}\cdot{\|D^{2}\varphi\|}\cdot{\|h^{\bullet}\|})\\ &+O(\ell^{\beta}\cdot{\delta^{-2}}d^{-\beta}\cdot{\|D^{2}\varphi\|\cdot{\|h^{\bullet}\|}})+O(\ell\cdot{\delta^{-3}}\cdot{\|D^{3}\varphi\|\cdot{\|h^{\bullet}\|}})
    \end{align*}
    Since the squares in $\mathcal{S}$ cover $B(z,\delta)$ and $\phi_{\delta}(z-w)$ is supported on $B(z,\delta)$ (as a function of $w$), we have that: 
    \begin{equation*}
        \sum_{S\in{\mathcal{S}}}\Big(\sum_{w\in{V^{\bullet}\cap{S}}}\Delta_{w}^{\bullet}\varphi_{\delta}(z-w)h^{\bullet}(w)\Big)=\sum_{w\in{\text{Int}(V^{\bullet})}}\Delta_{w}^{\bullet}\varphi_{\delta}(z-w)h^{\bullet}(w)=\sum_{w\in{\text{Int}(V^{\bullet})}}\varphi_{\delta}(z-w)\underbrace{\Delta^{\bullet}h^{\bullet}(w)}_{=0}=0
    \end{equation*}
    Putting all this together, we get that: 
    \begin{equation}
    \label{eqn: Laplacian estimate before optimizing in ell}
        \Delta(\varphi_{\delta}\ast{h^{\bullet}})(z)=O(\varepsilon\cdot{\ell^{-1}}\cdot{\delta^{-2}}\cdot{\|D^{2}\varphi\|}\cdot{\|h^{\bullet}\|})+O(\ell^{\beta}\cdot{\delta^{-2}}d^{-\beta}\cdot{\|D^{2}\varphi\|\cdot{\|h^{\bullet}\|}})+O(\ell\cdot{\delta^{-3}}{(\|D^{3}\varphi\|+\|D^{2}\varphi\|){\|h^{\bullet}\|}})
    \end{equation}
    In the estimate above, $d, \delta$ and $\varepsilon$ are given to us, whereas $\ell$ is just some parameter satisfying $\varepsilon\ll{\ell}\ll{\delta}$. To complete our proof, the last thing we need to do is optimize in $\ell$. \\ \\
    First observe that taking $\ell\asymp{\delta}$ or $\ell\asymp{\varepsilon}$ doesn't give us an effective estimate. Hence, the optimal choice of $\ell$ must lie on some intermediate scale. In particular, the asymptotically optimal choice of $\ell$ corresponds to a critical point of the function $f(\ell)=\ell^{-1}\varepsilon+\ell^{\beta}d^{-\beta}+\ell\delta^{-1}$. 
    \begin{align*}
        f'(\ell)=\delta^{-1}+\beta\ell^{\beta-1}d^{-\beta}-\ell^{-2}\varepsilon=0\hspace{5pt}\Longleftrightarrow\hspace{5pt}\delta\varepsilon=\ell^{2}+\beta\ell^{1+\beta}\delta{d^{-\beta}}
    \end{align*}
    When this equality holds, we either have: 
    \begin{equation*}
        \ell^{2}\asymp{\delta\varepsilon} \hspace{5pt}\Longleftrightarrow\hspace{5pt}\ell\asymp{\sqrt{\delta\varepsilon}}
    \end{equation*}
    or: 
    \begin{equation*}
        \ell^{1+\beta}\delta{d^{-\beta}}\asymp{\delta\varepsilon} \hspace{5pt}\Longleftrightarrow\hspace{5pt} \ell\asymp{d^{\frac{\beta}{1+\beta}}\varepsilon^{\frac{1}{1+\beta}}}
    \end{equation*}
    which of these asymptotics holds at the critical point depends on which of the relevant quantities- $\sqrt{\delta\varepsilon}$ or $d^{\frac{\beta}{1+\beta}}\varepsilon^{\frac{1}{1+\beta}}$- is smaller. When $\sqrt{\delta\varepsilon}\leq{d^{\beta/(1+\beta)}\varepsilon^{1/(1+\beta)}} \hspace{5pt}\Longleftrightarrow\hspace{5pt}\delta\leq{d^{\frac{2\beta}{1+\beta}}\varepsilon^{\frac{1-\beta}{1+\beta}}}$, we have that $\ell\asymp{\sqrt{\delta\varepsilon}}$ near the critical point, giving us the estimate: 
    \begin{align*}
        \Delta(\varphi_{\delta}\ast{h^{\bullet}})(z)&=O(\varepsilon^{\frac{1}{2}}\cdot\delta^{-\frac{5}{2}}\cdot(\|D^{2}\varphi\|+\|D^{2}\varphi\|)\cdot{\|h^{\bullet}\|})+O(\varepsilon^{\frac{\beta}{2}}\cdot{\delta^{(\frac{\beta}{2}-2)}}\cdot{d^{-\beta}}\cdot{\|D^{2}\varphi\|}\cdot{\|h^{\bullet}\|}) \\ 
        &=O(\varepsilon^{\frac{1}{2}}\cdot\delta^{-\frac{5}{2}}\cdot(\|D^{2}\varphi\|+\|D^{3}\varphi\|)\cdot{\|h^{\bullet}\|})
    \end{align*}
    The second equality follows from the fact that when $\delta\leq{d^{\frac{2\beta}{1+\beta}}\varepsilon^{\frac{1-\beta}{1+\beta}}}$, we have that $\varepsilon^{\frac{1}{2}}\delta^{-\frac{5}{2}}\geq{\varepsilon^{\frac{\beta}{2}}{\delta^{(\frac{\beta}{2}-2)}}{d^{-\beta}}}$. When $\delta\geq{d^{\frac{2\beta}{1+\beta}}\varepsilon^{\frac{1-\beta}{1+\beta}}}$, we have that $\ell\asymp{d^{\frac{\beta}{1+\beta}}\varepsilon^{\frac{1}{1+\beta}}}$ at the critical point, giving us the estimate: 
    \begin{align*}
        \Delta(\varphi_{\delta}\ast{h^{\bullet}})(z)&=O(\varepsilon^{\frac{\beta}{1+\beta}}\cdot\delta^{-2}\cdot{d^{-\frac{\beta}{1+\beta}}}\cdot\|D^{2}\varphi\|\cdot{\|h^{\bullet}\|})+O(\varepsilon^{\frac{1}{1+\beta}}\cdot{\delta^{-3}}\cdot{d^{\frac{\beta}{1+\beta}}}(\|D^{2}\varphi\|+\|D^{3}\varphi\|)\cdot{\|h^{\bullet}\|}) \\
        &=O(\varepsilon^{\frac{\beta}{1+\beta}}\cdot\delta^{-2}\cdot{d^{-\frac{\beta}{1+\beta}}}\cdot(\|D^{2}\varphi\|+\|D^{3}\varphi\|)\cdot{\|h^{\bullet}\|})
    \end{align*}
    Again, the second equality follows from the fact that when $\delta\geq{d^{\frac{2\beta}{1+\beta}}\varepsilon^{\frac{1-\beta}{1+\beta}}}$, we have that $\varepsilon^{\frac{\beta}{1+\beta}}\delta^{-2}{d^{-\frac{\beta}{1+\beta}}}\geq{\varepsilon^{\frac{1}{1+\beta}}{\delta^{-3}}{d^{\frac{\beta}{1+\beta}}}}$. This completes our proof. 
\end{proof}
\section{Proof of Theorem \ref{thm: polynomial rate of convergence for Dirichlet problem on OD maps}}
\label{sec: proof of polynomial rate of convergence for Dirichlet problem on OD maps}
We briefly recall the setting we are working in. Suppose $\Omega\subset{\mathbb{R}^{2}}$ is a simply connected domain and $g:\mathbb{R}^{2}\rightarrow{\mathbb{R}}$ is $\alpha$-H\"older for some $\alpha\in{(0,1)}$. Let $h$ be the solution to the continuous Dirichlet problem on $\Omega$ with boundary data given by $g$:
\begin{align*}
\Delta{h}(x)&=0 \hspace{22pt}\text{for all $x\in{\Omega}$} \\ 
h(x)&=g(x) \hspace{10pt}\text{for all $x\in{\partial{\Omega}}$}
\end{align*}
Suppose $G=(V^{\bullet}\sqcup{V^{\circ}}, E)$ is an orthodiagonal map with edges of length at most $\varepsilon$ so that $\widehat{G}\subseteq{\Omega}$ and $\text{dist}(x,\partial{\Omega})\leq{\varepsilon}$ for all $x\in{\partial{V^{\bullet}}}$. Let $h^{\bullet}:V^{\bullet}\rightarrow{\mathbb{R}}$ be the solution to the discrete Dirichlet problem on $G^{\bullet}$ with boundary data given by $g$:
\begin{align*}
    \Delta^{\bullet}{h^{\bullet}}(x)&=0 \hspace{22pt}\text{for all $x\in{\text{Int}(V^{\bullet})}$} \\ 
h^{\bullet}(x)&=g(x) \hspace{10pt}\text{for all $x\in{\partial{V^{\bullet}}}$}
\end{align*}  
Fix $z_{0}\in{V^{\bullet}}$. This will be the point at which we compare the solutions to the continuous and discrete Dirichlet problems. For any $z\in{\Omega}$, let $d_{z}=\text{dist}(z,\partial{\Omega})$. As we discussed in Section \ref{sec: outline of main theorem proof}, we will now estimate the difference $|h(z_{0})-h^{\bullet}(z_{0})|$ in two different ways. One approach will give us a superior estimate when $z_{0}$ is close to $\partial{\Omega}$. The other will give a superior estimate when $z_{0}$ is far away from $\partial{\Omega}$.
\subsection{Case 1: $z_{0}$ is close to $\partial{\Omega}$}
\label{subsec: case 1}
Let $w$ be a point of $\partial{\Omega}$ so that $d=d_{z_{0}}=|z_{0}-w|$. By considering the point of intersection between $\partial{\widehat{G}}$ and the line segment from $z_{0}$ to $w$, it follows that we can find a point $w'\in{\partial{V^{\bullet}}}$ so that $|z_{0}-w'|\leq{d+\varepsilon}$ and $|w-w'|\leq{2\varepsilon}$. By the triangle inequality: 
\begin{equation*}
    |h(z_{0})-h^{\bullet}(z_{0})|\leq{|h(z_{0})-g(w)|+|g(w)-g(w')|+|g(w')-h^{\bullet}(z_{0})|}
\end{equation*}
Since $g$ is $\alpha$-H\"older:
\begin{equation*}
    |g(w)-g(w')|\leq{2^{\alpha}\|g\|_{\alpha}\varepsilon^{\alpha}}
\end{equation*}
To estimate $|h(z_{0})-g(w)|$, we write this quantity as an expectation. Using the layer-cake representation of this expectation along with the strong Beurling estimate (Proposition \ref{prop: cts Beurling estimate}), we have that: 
\begin{align*}
    |h(z_{0})-g(w)|&=|\mathbb{E}^{z_{0}}g(B_{T_{\partial{\Omega}}})-g(w)|\leq{\mathbb{E}^{z_{0}}|g(B_{T_{\partial{\Omega}}})-g(w)|}=\int_{0}^{\infty}\mathbb{P}^{z_{0}}(|g(B_{T_{\partial{\Omega}}})-g(w)|\geq{\lambda})d\lambda \\ 
    &\leq{\int_{0}^{\infty}\mathbb{P}^{z_{0}}(\|g\|_{\alpha}|B_{T_{\partial{\Omega}}}-w|^{\alpha}\geq{\lambda})d\lambda}=\alpha\|g\|_{\alpha}\int_{0}^{\text{diam}(\Omega)}u^{\alpha-1}\mathbb{P}^{z_{0}}(|B_{T_{\partial{\Omega}}}-w|\geq{u})du \\ 
    &\leq{\alpha\|g\|_{\alpha}\int_{0}^{d}u^{\alpha}du+\alpha\|g\|_{\alpha}\int_{d}^{\text{diam}(\Omega)}u^{\alpha-1}C_{1}\Big(\frac{d}{u}\Big)^{1/2}du}\\ 
    &=\|g\|_{\alpha}d^{\alpha}+C_{1}\alpha\|g\|_{\alpha}d^{1/2}\int_{d}^{\text{diam}(\Omega)}u^{\alpha-3/2}du
\end{align*}
where $C_{1}>0$ is an absolute constant, $(B_{t})_{t\geq{0}}$ is a standard 2D Brownian motion, and $T_{\partial{\Omega}}$ is the hitting time of $\partial{\Omega}$ by this Brownian motion. Observe that: 
\begin{equation*}
    \int_{d}^{\text{diam}(\Omega)}u^{\alpha-3/2}du\leq{\begin{cases}
        \frac{2d^{\alpha-1/2}}{(1-2\alpha)} & \text{if $\alpha\in{(0,1/2)}$} \\ \\ 
        \log\big(\frac{\text{diam}(\Omega)}{d}\big) & \text{if $\alpha=1/2$} \\ \\
        \frac{2\text{diam}(\Omega)^{\alpha-1/2}}{(2\alpha-1)} & \text{if $\alpha\in{(1/2,1]}$}
        
    \end{cases}}
\end{equation*}
Hence: 
\begin{equation*}
    |h(z_{0})-g(w)|\lesssim{\begin{cases}
        \frac{\alpha}{1-2\alpha}\|g\|_{\alpha}d^{\alpha}  & \text{if $\alpha\in{(0,1/2)}$} \\ \\
        \alpha\|g\|_{\alpha}d^{\alpha}\log\big(\frac{\text{diam}(\Omega)}{d}\big) & \text{if $\alpha=1/2$} \\ \\ 
        \frac{\alpha}{2\alpha-1}\|g\|_{\alpha}\text{diam}(\Omega)^{\alpha}\big(\frac{d}{\text{diam}(\Omega)}\big)^{1/2} & \text{if $\alpha\in{(1/2,1]}$}
        
    \end{cases}}
\end{equation*}
Using the weak Beurling estimate for simple random walks on orthodiagonal maps (Lemma \ref{lem: weak Beurling on OD maps}) in place of the strong Beurling estimate for planar Brownian motion, we can use the same argument to estimate $|g(w')-h^{\bullet}(z_{0})|$:
\begin{align*}
    |h^{\bullet}(z_{0})-g(w')|&=|\mathbb{E}^{z_{0}}g(S_{T_{\partial{V^{\bullet}}}})-g(w')|\leq{\mathbb{E}^{z_{0}}|g(S_{T_{\partial{V^{\bullet}}}})-g(w')|}=\int_{0}^{\infty}\mathbb{P}^{z_{0}}(|g(S_{T_{\partial{V^{\bullet}}}})-g(w')|\geq{\lambda})d\lambda \\
    &\leq{\int_{0}^{\infty}\mathbb{P}^{z_{0}}(\|g\|_{\alpha}|S_{T_{\partial{V^{\bullet}}}}-w'|^{\alpha}\geq{\lambda})d\lambda}=\alpha\|g\|_{\alpha}\int_{0}^{\text{diam}(\Omega)}u^{\alpha-1}\mathbb{P}^{z_{0}}(|S_{T_{\partial{V^{\bullet}}}}-w'|\geq{u})du \\
    &\leq{\alpha\|g\|_{\alpha}\int_{0}^{2(d\vee{\varepsilon})}u^{\alpha-1}du+\alpha\|g\|_{\alpha}\int_{2(d\vee{\varepsilon})}^{\text{diam}(\Omega)}u^{\alpha-1}}C_{2}\Big(\frac{d\vee{\varepsilon}}{u}\Big)^{\beta}du \\
    &\lesssim{\|g\|_{\alpha}(d\vee{\varepsilon})^{\alpha}+\alpha\|g\|_{\alpha}(d\vee{\varepsilon})^{\beta}\int_{2(d\vee{\varepsilon})}^{\text{diam}(\Omega)}u^{\alpha-\beta-1}du}
\end{align*}
where $C_{2}>0$ is an absolute constant, $(S_{n})_{n\geq{0}}$ is a simple random walk on $G^{\bullet}$, and $T_{\partial{V^{\bullet}}}$ is the hitting time of $\partial{V^{\bullet}}$ by this random walk. The appearance of ``$(d\vee{\varepsilon})$" in our estimates comes from the fact that $|z_{0}-w'|\leq{d+\varepsilon}\leq{2(d\vee{\varepsilon})}$. Observe that: 
\begin{equation*}
    \int_{2(d\vee{\varepsilon})}^{\text{diam}(\Omega)}u^{\alpha-\beta-1}du\lesssim{\begin{cases}
        \frac{(d\vee{\varepsilon})^{\alpha-\beta}}{\beta-\alpha} & \text{if $\alpha\in{(0,\beta)}$} \\ \\
        \log\big(\frac{\text{diam}(\Omega)}{d\vee{\varepsilon}}\big) & \text{if $\alpha=\beta$} \\ \\
        \frac{\text{diam}(\Omega)^{\alpha-\beta}}{\alpha-\beta} & \text{if $\alpha\in{(\beta,1]}$}
    \end{cases}}
\end{equation*}
Hence: 
\begin{equation*}
    |h^{\bullet}(z_{0})-g(w')|\lesssim{\begin{cases}
        \frac{\alpha}{\beta-\alpha}\|g\|_{\alpha}(d\vee{\varepsilon})^{\alpha}  & \text{if $\alpha\in{(0,\beta)}$} \\ \\
        \|g\|_{\alpha}(d\vee{\varepsilon})^{\alpha}\log\big(\frac{\text{diam}(\Omega)}{d\vee{\varepsilon}}\big) & \text{if $\alpha=\beta$} \\ \\
        \frac{\alpha}{\alpha-\beta}\|g\|_{\alpha}\text{diam}(\Omega)^{\alpha}\big(\frac{d\vee{\varepsilon}}{\text{diam}(\Omega)}\big)^{\beta} & \text{if $\alpha\in{(\beta,1]}$}
    \end{cases}}
\end{equation*}
Since we used a strictly weaker version of the Beurling estimate, our estimate for $|h^{\bullet}(z_{0})-h(w')|$ is necessarily worse than our estimate for $|h(z_{0})-g(w)|$. Hence, putting all this together, we have that:
\begin{equation*}
    |h(z_{0})-h^{\bullet}(z_{0})|\lesssim{\begin{cases}
        \frac{\beta}{\beta-\alpha}\|g\|_{\alpha}(d\vee{\varepsilon})^{\alpha}  & \text{if $\alpha\in{(0,\beta)}$} \\ \\
        \|g\|_{\alpha}(d\vee{\varepsilon})^{\alpha}\log\big(\frac{\text{diam}(\Omega)}{d\vee{\varepsilon}}\big) & \text{if $\alpha=\beta$} \\ \\
        \frac{\alpha}{\alpha-\beta}\|g\|_{\alpha}\text{diam}(\Omega)^{\alpha}\big(\frac{d\vee{\varepsilon}}{\text{diam}(\Omega)}\big)^{\beta} & \text{if $\alpha\in{(\beta,1]}$}
    \end{cases}}
\end{equation*}
\subsection{Case 2: $z_{0}$ is far away from $\partial{\Omega}$}
\label{subsec: case 2}
Let $\delta>0$ be some mesoscopic scale whose exact value we will fix later. If $\phi$ is a radially symmetric smooth mollifier supported on the unit ball $B(0,1)\subseteq{\mathbb{R}^{2}}$, then $\phi_{\delta}(x)=\delta^{-2}\phi(\delta^{-1}x)$ is a radially symmetric smooth mollifer, supported on $B(0,\delta)$. Extend $h^{\bullet}:V^{\bullet}\rightarrow{\mathbb{R}}$ to a function on $\widehat{G}$ in any sort of sensible way. For instance, we could triangulate the faces of $G^{\bullet}$ and define $h^{\bullet}$ on each triangle by linear interpolation. In this way, we can think of $h^{\bullet}$ as a function on $\widehat{G}$. This allows us to consider the convolution $\phi_{\delta}\ast{h^{\bullet}}$. Notice that this is only well-defined for points of $\widehat{G}$ that are at least $\delta$ far away from $\partial{\widehat{G}}$. With this in mind, let $\Omega^{\delta}$ be a simply connected domain so that:
\begin{itemize}
    \item $z_{0}\in{\Omega^{\delta}}$
    \item $\Omega^{\delta}\subseteq{\widehat{G}}\subseteq{\Omega}$
    \item every point of $\partial{\Omega^{\delta}}$ is at least $2\delta$ far away from $\partial{\widehat{G}}$.
    \item every point of $\partial{\Omega^{\delta}}$ lies within $O(\delta)$ of $\partial{\Omega}$ and therefore $\partial{\widehat{G}}$. 
\end{itemize}
In this way, $\phi_{\delta}\ast{h^{\bullet}}$ is a well-defined function on $\Omega^{\delta}$. Let $\widetilde{h}$ be the solution to the continuous Dirichlet problem on $\Omega^{\delta}$ with boundary data given by $\phi_{\delta}\ast{h^{\bullet}}$. That is: 
\begin{align*}
    \Delta{\widetilde{h}}(z)&=0 \hspace{59pt} \text{for all $z\in{\Omega^{\delta}}$} \\
    \widetilde{h}(z)&=(\phi_{\delta}\ast{h^{\bullet}})(z) \hspace{10pt} \text{for all $z\in{\partial{\Omega^{\delta}}}$}
\end{align*}
By the triangle inequality: 
\begin{equation}
\label{eqn: triangle inequality}
    |h^{\bullet}(z_{0})-h(z_{0})|\leq{|h^{\bullet}(z_{0})-(\phi_{\delta}\ast{h^{\bullet}})(z_{0})|+|(\phi_{\delta}\ast{h^{\bullet}})(z_{0})-\widetilde{h}(z_{0})|+|\widetilde{h}(z)-h(z)|}
\end{equation}
By Lemma \ref{lem: discrete harmonic functions are beta Holder in the bulk}: 
\begin{equation*}
    |h^{\bullet}(z_{0})-(\phi_{\delta}\ast{h^{\bullet}})(z_{0})|\lesssim{\|g\|\Big(\frac{\delta}{d}\Big)^{\beta}}
\end{equation*}
To handle the second term in Equation \ref{eqn: triangle inequality}, observe that $(\phi_{\delta}\ast{h^{\bullet}})$ is a smooth function on $\Omega^{\delta}$ that extends continuously to $\partial{\Omega^{\delta}}$ and $\widetilde{h}$ is the harmonic function on $\Omega^{\delta}$ that agrees with $\phi_{\delta}\ast{h^{\bullet}}$ on $\partial{\Omega^{\delta}}$. Hence:
\begin{equation}
\label{eqn: difference between mollified harmonic function and its harmonic extension}
    |(\phi_{\delta}\ast{h^{\bullet}})(z_{0})-\widetilde{h}(z_{0})|=\big|\int_{\Omega_{\delta}}(\Delta(\phi_{\delta}\ast{h^{\bullet}}))(w)G_{\Omega^{\delta}}(w,z_{0})dA(w)\big|
\end{equation}
Proposition \ref{prop: convolution has small Laplacian} tells us that the convolution of a discrete harmonic function with a smooth mollifier is almost harmonic. Namely, we have that: 
\begin{equation}
\label{eqn: Laplacian estimate}
    |(\Delta(\phi_{\delta}\ast{h^{\bullet}}))(w)|\lesssim{\|g\|(\varepsilon^{\frac{1}{2}}\delta^{-\frac{5}{2}}+\varepsilon^{\frac{\beta}{1+\beta}}\delta^{-2}d_{w}^{-\frac{\beta}{1+\beta}})}\lesssim{\|g\|(\varepsilon^{\frac{1}{2}}\delta^{-\frac{5}{2}}+\varepsilon^{\frac{\beta}{1+\beta}}\delta^{-2-\frac{\beta}{1+\beta}})}\lesssim{\|g\|\hspace{1pt}\varepsilon^{\frac{\beta}{1+\beta}}\delta^{-2-\frac{\beta}{1+\beta}}}
\end{equation}
where $d_{w}=\text{dist}(w,\partial{\widehat{G}})\asymp{\text{dist}(w,\partial{\Omega})}$. The second inequality in Equation \ref{eqn: Laplacian estimate} follows from the fact that we are only considering points $w\in{\Omega^{\delta}}$. Plugging our Laplacian estimate in Equation \ref{eqn: Laplacian estimate} into Equation \ref{eqn: difference between mollified harmonic function and its harmonic extension}, we have that:  
\begin{equation}
\label{eqn: estimate in terms of integral of the Green's function}
    |(\phi_{\delta}\ast{h^{\bullet}})(z_{0})-\widetilde{h}(z_{0})|\lesssim{\|g\|\hspace{1pt}\varepsilon^{\frac{\beta}{1+\beta}}\delta^{-2-\frac{\beta}{1+\beta}}\int_{\Omega^{\delta}}G_{\Omega^{\delta}}(w,z_{0})dA(w)} 
\end{equation}
Observe that the integral appearing on the RHS of the inequality above, can be interpreted probabilistically as the expected amount of time a planar Brownian motion, started at $z_{0}$, spends in $\Omega^{\delta}$ before hitting $\partial{\Omega^{\delta}}$. That is: 
\begin{equation*}
    \int_{\Omega^{\delta}}G_{\Omega^{\delta}}(w,z_{0})dA(w)=\mathbb{E}^{z_{0}}T_{\partial{\Omega^{\delta}}}\leq{\mathbb{E}^{z_{0}}T_{\partial{\Omega}}}
\end{equation*}
where $T_{\partial{\Omega}}$ and $T_{\partial{\Omega^{\delta}}}$ are the hitting times of $\partial{\Omega}$ and $\partial{\Omega^{\delta}}$ by our planar Brownian motion. Let $(B_{t})_{t\geq{0}}$ be a planar Brownian motion. Then the process $\big(|B_{t\wedge{T_{\partial{\Omega}}}}-z_{0}|^{2}-2t\wedge{T_{\partial{\Omega}}}\big)_{t\geq{0}}$ is a martingale. By the optional stopping theorem: 
\begin{equation*}
    \mathbb{E}^{z_{0}}|B_{T_{\partial{\Omega}}}-z_{0}|^{2}=2\hspace{1pt}\mathbb{E}^{z_{0}}T_{\partial{\Omega}}
\end{equation*}
Using the layer-cake representation of the expectation on the LHS along with the strong Beurling estimate (Proposition \ref{prop: cts Beurling estimate}), we have that: 
\begin{align*}
    \mathbb{E}^{z_{0}}|B_{T_{\partial{\Omega}}}-z_{0}|^{2}&=\int_{0}^{\infty}\mathbb{P}^{z_{0}}(|B_{T_{\partial{\Omega}}}-z_{0}|^{2}\geq{\lambda})d\lambda=2\int_{0}^{\infty}u\hspace{1pt}\mathbb{P}^{z_{0}}(|B_{T_{\partial{\Omega}}}-z_{0}|\geq{u})du \\
    &= 2\int_{0}^{d}u\hspace{1pt}du +  2\int_{d}^{\text{diam}(\Omega)}C\Big(\frac{d}{u}\Big)^{1/2}u\hspace{1pt}du\lesssim{d^{1/2}\text{diam}(\Omega)^{3/2}}
\end{align*}
where $C>0$ is an absolute constant. Plugging this estimate for $\mathbb{E}^{z_{0}}|B_{T_{\partial{\Omega}}}-z_{0}|^{2}=\mathbb{E}^{z_{0}}T_{\partial_{\Omega}}\geq{\mathbb{E}^{z_{0}}T_{\partial{\Omega^{\delta}}}}$ into Equation \ref{eqn: estimate in terms of integral of the Green's function}, we have that: 
\begin{equation*}
    |(\phi_{\delta}\ast{h^{\bullet}})(z_{0})-\widetilde{h}(z_{0})|\lesssim{\|g\|\hspace{1pt}\varepsilon^{\frac{\beta}{1+\beta}}\delta^{-2-\frac{\beta}{1+\beta}}d^{1/2}\text{diam}(\Omega)^{3/2}}
\end{equation*}
To estimate the third term in Equation \ref{eqn: triangle inequality} we use the maximum principle. Suppose $w\in{\partial{\Omega^{\delta}}}$. Then $\widetilde{h}(w)=(\phi_{\delta}\ast{h^{\bullet}})(w)$, since the boundary data of $\widetilde{h}$ on $\partial{\Omega^{\delta}}$ is given by $\phi_{\delta}\ast{h^{\bullet}}$. On the other hand, since $h$ is harmonic on $\Omega$, and the smooth mollifier $\phi_{\delta}$ is radially symmetric, $(\phi_{\delta}\ast{h})(w)=h(w)$. Hence: 
\begin{equation*}
    |h(w)-\widetilde{h}(w)|=|(\phi_{\delta}\ast{h})(w)-(\phi_{\delta}\ast{h^{\bullet}})(w)|\leq{\max_{w'\in{B(w,\delta)}}|h(w')-h^{\bullet}(w')|}
\end{equation*}
Since $\delta$ is small and the points $w'\in{B(w,\delta)}$ are $\delta$-close to the boundary of $\Omega$, by the same argument as in Case 1, we have that: 
\begin{equation*}
    |h(w)-\widetilde{h}(w)|\lesssim{\begin{cases}
        \frac{\beta}{\beta-\alpha}\|g\|_{\alpha}\delta^{\alpha}  & \text{if $\alpha\in{(0,\beta)}$} \\ \\
        \|g\|_{\alpha}\delta^{\alpha}\log\big(\frac{\text{diam}(\Omega)}{\delta}\big) & \text{if $\alpha=\beta$} \\ \\
        \frac{\alpha}{\alpha-\beta}\|g\|_{\alpha}\text{diam}(\Omega)^{\alpha}\big(\frac{\delta}{\text{diam}(\Omega)}\big)^{\beta} & \text{if $\alpha\in{(\beta,1]}$}
    \end{cases}}
\end{equation*}
Putting all this together, if $\alpha\in{(0,\beta)}$, we have that:  
\begin{equation*}
    |h(z_{0})-h^{\bullet}(z_{0})|=
        O(\|g\|_{\infty}\delta^{\beta}d^{-\beta})+O(\|g\|\hspace{1pt}\varepsilon^{\frac{\beta}{1+\beta}}\delta^{-2-\frac{\beta}{1+\beta}}d^{1/2}\text{diam}(\Omega)^{3/2})+O(\frac{\beta}{\beta-\alpha}\|g\|_{\alpha}\delta^{\alpha})
\end{equation*}
If $\alpha=\beta$:
\begin{equation*}
    |h(z_{0})-h^{\bullet}(z_{0})|=
        O(\|g\|_{\infty}\delta^{\beta}d^{-\beta})+O(\|g\|\hspace{1pt}\varepsilon^{\frac{\beta}{1+\beta}}\delta^{-2-\frac{\beta}{1+\beta}}d^{1/2}\text{diam}(\Omega)^{3/2})+O(\|g\|_{\alpha}\delta^{\alpha}\log\big(\frac{\text{diam}(\Omega)}{\delta}\big))
\end{equation*}
If $\alpha\in{(\beta,1)}$: 
\begin{equation*}
    |h(z_{0})-h^{\bullet}(z_{0})|=
        O(\|g\|_{\infty}\delta^{\beta}d^{-\beta})+O(\|g\|\hspace{1pt}\varepsilon^{\frac{\beta}{1+\beta}}\delta^{-2-\frac{\beta}{1+\beta}}d^{1/2}\text{diam}(\Omega)^{3/2})+O(\frac{\alpha}{\alpha-\beta}\|g\|_{\alpha}\text{diam}(\Omega)^{\alpha}\big(\frac{\delta}{\text{diam}(\Omega)}\big)^{\beta})
\end{equation*} 
\subsection{Choosing an Optimal $\delta$}
\label{subsec: optimizing delta}
In Section \ref{subsec: case 1}, we derived an estimate for $|h(z_{0})-h^{\bullet}(z_{0})|$ for $z_{0}$ close to the boundary. In Section \ref{subsec: case 2}, we derived an estimate for $|h(z_{0})-h^{\bullet}(z_{0})|$ for $z_{0}$ far away from the boundary. To complete our proof, we need to: 
\begin{enumerate}
    \item Find the optimal choice of $\delta$ for our estimate from Section \ref{subsec: case 2}.
    \item Combine these estimates to get an estimate that works for all $z_{0}\in{V^{\bullet}}$. 
\end{enumerate}
We will do this in detail for $\alpha\in{(0,\beta)}$. The corresponding estimates when $\alpha\in{[\beta,1]}$ follow by the same argument. Armed with the intuition that our rate of convergence should be polynomial in $\varepsilon$, we take $\delta=\varepsilon^{s}\hspace{1pt}\text{diam}(\Omega)^{1-s}$, $d=\varepsilon^{r}\hspace{1pt}\text{diam}(\Omega)^{1-r}$, where $0<r<s<1$. By our estimates from Sections \ref{subsec: case 1} and \ref{subsec: case 2} we have that: 
\begin{align*}
    |h(z_{0})-h^{\bullet}(z_{0})|\lesssim\Big(\|g\|+\frac{\beta}{\beta-\alpha}\|g\|_{\alpha}\text{diam}(\Omega)^{\alpha}\Big)\cdot\min\{&\varepsilon^{\alpha{r}}\text{diam}(\Omega)^{-\alpha{r}}, \varepsilon^{\beta(s-r)}\text{diam}(\Omega)^{-\beta(s-r)}+... \\ &...+\varepsilon^{\frac{\beta}{1+\beta}(1-s)-2s+\frac{r}{2}}\text{diam}(\Omega)^{-\frac{\beta}{1+\beta}(1-s)+2s-\frac{r}{2}}+... \\ &...+\varepsilon^{\alpha{s}}\text{diam}(\Omega)^{-\alpha{s}}\}
\end{align*}
for $\alpha\in{(0,\beta)}$. We want to find the value of $\delta$ that minimizes our error for a point which is distance $d$ from the boundary of $\Omega$. This amounts to finding a value of $s$ that minimizes our error for a fixed choice of $r$. In other words, we are interested in the maximum of the function: 
\begin{equation*}
    \Xi(\alpha,\beta,r,s)=
        \max\{\alpha{r}, \min\Big\{\beta(s-r),\frac{\beta}{1+\beta}-(2+\frac{\beta}{1+\beta})s+\frac{r}{2}, \alpha{s}\Big\}\}
\end{equation*}
in $s$, treating $\alpha$, $\beta$ and $r$ as constants: 
\begin{equation}
\label{eqn: maximum of Xi in s}
    \max_{s\in{(r,1)}}\Xi(\alpha,\beta,r,s)=\max\{\alpha{r},\min\Big\{\alpha\Big(\frac{\frac{r}{2}+\frac{\beta}{1+\beta}}{2+\alpha+\frac{\beta}{1+\beta}}\Big), \beta\Big(\Big(\frac{(\beta+\frac{1}{2})r+\frac{\beta}{1+\beta}}{2+\beta+\frac{\beta}{1+\beta}}\Big)-r\Big)\Big\}\}
\end{equation}
From here, we take the minimum of the resulting function, $\max\limits_{s\in{(r,1)}}\Xi(\alpha,\beta,r,s)$, in $r$, treating $\alpha$ and $\beta$ as constants. This corresponds to finding an estimate that works for all $d$. In this way, we conclude that:
\begin{equation*}
    |h(z_{0})-h^{\bullet}(z_{0})|\leq{\big(C_{1}\|g\| + C_{2}\frac{\beta}{\beta-\alpha}\|g\|_{\alpha}\text{diam}(\Omega)^{\alpha}\big)\left(\frac{\varepsilon}{\text{diam}(\Omega)}\right)^{\lambda(\alpha,\beta)}}
\end{equation*}
where: 
\begin{equation*}
    \lambda(\alpha,\beta)=\min_{r\in{[0,1]}}\max_{s\in{(r,1)}}\Xi(\alpha,\beta,r,s)
\end{equation*}
Our formula for the maximum of $\Xi(\alpha,\beta,r,s)$ in $s$ in Equation \ref{eqn: maximum of Xi in s} was already pretty cumbersome. While it is possible to write an explicit formula for $\lambda(\alpha,\beta)$, this formula is a mess so we will not do this here. Observe however that for any $\alpha\in{(0,\beta)}$ and $\beta\in{(0,1/2]}$, $\lambda(\alpha,\beta)>0$. To see this, notice that if we take $s=\frac{\beta}{4+6\beta}$ and $r\in{[\frac{\beta}{8+12\beta},1]}$, clearly: 
\begin{equation*}
    \Xi(\alpha,\beta,r,\frac{\beta}{4+6\beta})\geq{\frac{\beta\alpha}{8+12\beta}}
\end{equation*}
On the other hand, if $r\in{[0,\frac{\beta}{8+12\beta}]}$:
\begin{equation*}
    \Xi(\alpha,\beta,r,\frac{\beta}{4+6\beta})\geq{\min\{\beta\big(\frac{\beta}{4+6\beta}-r\big), \frac{\beta}{2+2\beta}+\frac{r}{2}, \frac{\beta\alpha}{4+6\beta}\}}\geq{\min\{\frac{\beta^{2}}{8+12\beta}, \frac{\beta}{2+\beta}, \frac{\alpha\beta}{4+6\beta}\}}
\end{equation*}
Hence: 
\begin{equation*}
    \lambda(\alpha,\beta)\geq{\frac{\alpha\beta}{8+12\beta}}
\end{equation*}
\qed

\section{Improving our Rate of Convergence}
\label{sec: improved rate of convergence for Dirichlet problem}

In the proof of Theorem \ref{thm: polynomial rate of convergence for Dirichlet problem on OD maps} we used the weak Harnack-type estimate, Lemma \ref{lem: discrete harmonic functions are beta Holder in the bulk}, twice. First to show that for a point $z_{0}\in{\Omega}$ mesoscopically far away from the boundary of $\Omega$, $\phi_{\delta}\ast{h^{\bullet}}(z_{0})$ is close to $h^{\bullet}(z_{0})$. And then when applying Proposition \ref{prop: convolution has small Laplacian}, since the proof of Proposition \ref{prop: convolution has small Laplacian} uses this lemma. All this takes place in Section \ref{subsec: case 2}. \\ \\
In Section \ref{sec: Lipschitz Regularity on a Mesoscopic Scale for Harmonic Functions on Orthodiagonal Maps}, we will prove that: 
\begin{enumerate}
    \item Discrete harmonic functions on orthodiagonal maps satisfy a strong Harnack-type estimate on a mesoscopic scale. (Theorem \ref{thm: optimal Lipschitz regularity on mesoscopic scales})
    \item Having a strong Harnack estimate on large scales implies that we have a refined weak Harnack estimate at any intermediate scale. (Proposition \ref{prop: refined Holder regularity on intermediate scales})
\end{enumerate} 
Since we always use Lemma \ref{lem: discrete harmonic functions are beta Holder in the bulk} on a mesoscopic scale in the proof of Theorem \ref{thm: polynomial rate of convergence for Dirichlet problem on OD maps}, replacing every application of Lemma \ref{lem: discrete harmonic functions are beta Holder in the bulk} by a combination of Theorem \ref{thm: optimal Lipschitz regularity on mesoscopic scales} and Proposition \ref{prop: refined Holder regularity on intermediate scales}, we can improve our rate of convergence in Theorem \ref{thm: polynomial rate of convergence for Dirichlet problem on OD maps}. That said, the formula for $\lambda(\alpha,\beta)$ (the exponent in our rate of convergence) that we get when we do this is even more complicated than the one in Theorem \ref{thm: polynomial rate of convergence for Dirichlet problem on OD maps}, so we won't bother writing this out.\\ \\
\noindent While we're unable to prove it at this time, since orthodiagonal maps are good approximations of continuous 2D space, we expect that discrete harmonic functions on orthodiagonal maps should satisfy a strong Harnack-type estimate of the kind that is known to hold in the continuum. That is, we expect that there exist absolute constants $C_{1}. C_{2}>0$ so that if $G=(V^{\bullet}\sqcup{V^{\circ}}, E)$ is an orthodiagonal map with edges of length at most $\varepsilon$ and $h:V^{\bullet}\rightarrow{\mathbb{R}}$ is harmonic, we have that: 
\begin{equation}
\label{eqn: strong Harnack on OD maps}
    |h(y)-h(x)|\leq{C_{1}\|h\|\Big(\frac{|x-y|}{d}\Big)}
\end{equation}
where $d=\text{dist}(x,\partial{\widehat{G}})\wedge{\text{dist}(y,\partial{\widehat{G}})}$ and $x$ and $y$ are points of $V^{\bullet}$ such that $|x-y|\geq{C_{2}\varepsilon}$. Assuming that such an estimate holds, we can replace every application of our weak Harnack estimate in the proof of Theorem \ref{thm: polynomial rate of convergence for Dirichlet problem on OD maps} with this strong Harnack estimate, thereby improving our rate of convergence:  
\begin{cor}
\label{cor: rate of convergence given weak Beurling and strong Harnack} 
Suppose $\Omega\subseteq{\mathbb{R}^{2}}$ is a bounded simply connected domain, $g$ is a $\alpha$-H\"older on $\mathbb{R}^{2}$ and $G=(V^{\bullet}\sqcup{V^{\circ}},E)$ is an orthodiagonal map with edges of length at most $\varepsilon$ belonging to a subcollection of orthodiagonal maps on which a strong Harnack-type estimate holds (see Equation \ref{eqn: strong Harnack on OD maps}). Suppose also that for each point $z\in{\partial{\widehat{V}}}$, $\text{dist}(z,\partial{\Omega})\leq{\varepsilon}$. Let $h$ be the solution to the continuous Dirichlet problem on $\Omega$ with boundary data given by $g$ and let $h^{\bullet}$ be the solution to the discrete Dirichlet problem on $G^{\bullet}$ with boundary data given by $g$, where $g$ is $\alpha$-H\"older for some $\alpha\in{(0,1]}$.  If $\beta\in{(0,1)}$ is the absolute constant from Lemma \ref{lem: weak Beurling on OD maps}, for any $v\in{V^{\bullet}}$, we have that: 
    \begin{equation*}
        |h^{\bullet}(v)-h(v)|\leq{\begin{cases}
            \big(C_{1}\|g\| + C_{2}\frac{\beta}{\beta-\alpha}\|g\|_{\alpha}\text{diam}(\Omega)^{\alpha}\big)\big(\frac{\varepsilon}{\text{diam}(\Omega)}\big)^{\theta(\alpha,\beta)} & \text{if $\alpha\in{(0,\beta)}$} \\ \\
            \big(C_{1}\|g\| + C_{2}\|g\|_{\alpha}\text{diam}(\Omega)^{\alpha}\big)\big(\log\big(\frac{\text{diam}(\Omega)}{\varepsilon}\big)\big)\big(\frac{\varepsilon}{\text{diam}(\Omega)}\big)^{\theta(\alpha,\beta)} & \text{if $\alpha=\beta$} \\ \\
            \big(C_{1}\|g\| + C_{2}\frac{\alpha}{\beta-\alpha}\|g\|_{\alpha}\text{diam}(\Omega)^{\alpha}\big)\big(\frac{\varepsilon}{\text{diam}(\Omega)}\big)^{\theta(\alpha,\beta)} & \text{if $\alpha\in{(\beta,1)}$}
        \end{cases}}
    \end{equation*}
    where $C_{1},C_{2}>0$ are absolute constants, $\|g\|_{\alpha}$ is the $\alpha$-H\"older norm of $g$, and the function $\theta$ is given by: 
    \begin{align*}
        \theta(\alpha,\beta)&=\min_{r\in{[0,1]}}\max_{s\in{(r,1)}}\big(\max\{(\alpha\wedge{\beta})r, \min\{s-r, \frac{1}{2}-\frac{5}{2}s+\frac{r}{2}, (\alpha\wedge{\beta})s\}\}\big) \\ 
        &=\min\Big\{\frac{\alpha\wedge{\beta}}{5+2(\alpha\wedge{\beta})}, \frac{\alpha\wedge{\beta}}{4+7(\alpha\wedge{\beta})}\Big\}
    \end{align*}
\end{cor}



\section{Lipschitz Regularity on a Mesoscopic Scale for Harmonic Functions on Orthodiagonal Maps}
\label{sec: Lipschitz Regularity on a Mesoscopic Scale for Harmonic Functions on Orthodiagonal Maps}

As we alluded to in Section \ref{sec: Lipschitz Regularity on a Mesoscopic Scale for Harmonic Functions on Orthodiagonal Maps}, in this Section, we will prove the following Harnack-type estimate for discrete harmonic functions on orthodiagonal maps: 
\begin{thm}
\label{thm: optimal Lipschitz regularity on mesoscopic scales}
    If $\beta\in{(0,\frac{1}{2})}$ is the absolute constant from Lemma \ref{lem: weak Beurling on OD maps}, for any $\alpha\in{(0,\frac{\beta}{1+3\beta})}$, we have an absolute constant $C_{\alpha}>0$ so that if $G=(V^{\bullet}\sqcup{V^{\circ}},E)$ is an orthodiagonal map with edges of length at most $\varepsilon$, $h:V^{\bullet}\rightarrow{\mathbb{R}}$ is harmonic, and $z,w\in{V^{\bullet}}$ satisfy $|z-w|\geq{d\big({\small \frac{\varepsilon}{d}}\big)^{\alpha}}$, where $d=\text{dist}(z,\partial{\widehat{G}})\wedge{\text{dist}(w,\partial{\widehat{G}})}$, then:
    \begin{equation*}
        |h(z)-h(w)|\leq{C_{\alpha}\|h\|_{\infty}\Big(\frac{|z-w|}{d}\Big)}
    \end{equation*}
\end{thm}
\noindent We say that this estimate holds on a mesoscopic scale because it requires that the points $z,w\in{V^{\bullet}}$ we're looking at are at least $d\big(\frac{\varepsilon}{d}\big)^{\alpha}$ apart, where $\alpha\in{(0,1)}$, implying that $\varepsilon\ll{d\big(\frac{\varepsilon}{d}\big)^{\alpha}}\ll{1}$. One interpretation of this result is that it tells us that discrete harmonic functions on orthodiagonal maps are Lipschitz in the bulk on a mesoscopic scale. We do not believe this result is sharp. Namely, since orthodiagonal maps are good approximations of continuous 2D space, the Harnack estimate should hold even on a microscopic scale. That is, for any pair of points $z,w\in{V^{\bullet}}$ that are at least $C\varepsilon$ apart, for some absolute constant $C>0$. As we remarked in Section \ref{sec: Lipschitz Regularity on a Mesoscopic Scale for Harmonic Functions on Orthodiagonal Maps}, this is known to be true for any isoradial graph. This includes subsets of the triangular, the hexagonal, and the square grid. Furthermore, in \cite{CLR23}, Chelkak, Laslier and Russkikh show that we have a Harnack estimate on microscopic scales for discrete harmonic functions on t-embeddings satisfying the assumptions ``$Lip(\kappa,\delta)$" and ``$Exp\text{-}Fat(\delta)$" for some $\kappa\in{(0,1)}$, $\delta>0$ (see Corollary 6.18 of \cite{CLR23}). For a precise definition of the assumptions ``$Lip(\kappa,\delta)$" and ``$Exp\text{-}Fat(\delta)$," see Section 1.2 of \cite{CLR23}. As we discussed in Section \ref{subsec: outline of proof near boundary}, for any $\kappa\in{(0,1)}$, there exists $c=c(\kappa)>0$ so that any orthodiagonal map of edge length at most $\varepsilon$, satisfies the assumption ``$Lip(c\hspace{1pt}\varepsilon, \varepsilon)$." In contrast, it is not known whether an arbitrary orthodiagonal map satisfies the condition ``$Exp\text{-}Fat(\delta)$" for some $\delta>0$ that only depends on the mesh of our orthodiagonal map. Thus, we do not have a Harnack estimate on microscopic scales for discrete harmonic functions on orthodiagonal maps as an immediate consequence of Corollary 6.18 of \cite{CLR23}.

\subsection{Lipschitz Regularity on a Mesoscopic Scale (Base Case)}
\label{subsec: Lipschitz Regularity on a Mesoscopic Scale (Base Case)}

The key idea behind the proof of Theorem \ref{thm: optimal Lipschitz regularity on mesoscopic scales} is the following regularity estimate for $C^{2}$ functions in terms of their norm and Laplacian: 
\begin{prop}
\label{prop: almost harmonic functions satisfy a Harnack type estimate}
     Suppose $\Omega$ is a simply connected domain and $h\in{C^{2}_{b}(\Omega)\cap{C(\overline{\Omega})}}$. Then: 
    \begin{equation}
        |h(x_{2})-h(x_{1})|\lesssim{\|h\|_{\infty}\Big(\frac{|x_{2}-x_{1}|}{d}\Big)+\|\Delta{h}\|_{\infty}d|x_{2}-x_{1}|}
    \end{equation}
     for any $x_{1},x_{2}\in{\Omega}$, where $d=d_{x_{1}}\wedge{d_{x_{2}}}=\text{dist}(x_{1},\partial{\Omega})\wedge{\text{dist}(x_{2},\partial{\Omega})}$.
\end{prop}

\begin{proof}
    Suppose $B(x,R)\subseteq{\Omega}$ is a ball contained in $\Omega$. By Green's identity applied to $h$ and $G_{B(x,R)}(y,x)=-\frac{1}{2\pi}\log\Big(\frac{|y-x|}{R}\Big)$, we have that:
\begin{equation*}
    h(x)=\frac{1}{2\pi{R}}\int_{\partial{B(x,R)}}h(y)d\sigma(y)-\frac{1}{2\pi}\int_{B(x,R)}\Delta{h}(y)\log\Big(\frac{|y-x|}{R}\Big)dA(y)
\end{equation*}
where ``dA(y)" denotes integration with respect to area in $\Omega$. Similarly, applying Green's identity with $h$ and $R^{2}-|y-x|^{2}$, we have that:
\begin{equation*}
    \frac{1}{2\pi{R}}\int_{\partial{B(x,R)}}h(y)d\sigma(y)=\frac{1}{\pi{R}^{2}}\int_{B(x,R)}h(y)dA(y)+\frac{1}{4\pi{R^{2}}}\int_{B(x,R)}\Delta{h}(y)\big(R^{2}-|y-x|^{2}\big)dA(y)
\end{equation*}
Putting all this together, we have that: 
\begin{equation*}
    h(x)=\frac{1}{\pi{R}^{2}}\int_{B(x,R)}h(y)dA(y)+\frac{1}{4\pi{R^{2}}}\int_{B(x,R)}\Delta{h}(y)\big(R^{2}-|y-x|^{2}\big)dA(y)-\frac{1}{2\pi}\int_{B(x,R)}\Delta{h}(y)\log\Big(\frac{|y-x|}{R}\Big)dA(y) 
\end{equation*}
Suppose $x_{1}, x_{2}$ are points of $\Omega$. Observe that if $|x_{2}-x_{1}|>\frac{d}{2}$, then trivially, $|h(x_{2})-h(x_{1})|\leq{4\|h\|_{\infty}\big(\frac{|x_{2}-x_{1}|}{d}\big)}$ and so the desired result follows. With this in mind, suppose that $|x_{2}-x_{1}|\leq{\frac{d}{2}}$. Then: 
\begin{equation*}
    h(x_{2})-h(x_{1})=(1)+(2)+(3)
\end{equation*}
where: 
\begin{align*}
    |(1)|=&\Big|\frac{1}{\pi{d^{2}}}\int_{B(x_{1},d)\Delta{B(x_{2},d)}}h(y)dA(y)\Big|\leq{\frac{5\pi|x_{2}-x_{1}|d}{\pi{d}^{2}}\|h\|_{\infty}}=5\|h\|_{\infty}\Big(\frac{|x_{2}-x_{1}|}{d}\Big) \\
    \vspace{5pt}
    |(2)|=&\frac{1}{4\pi}\Big|\int_{B(x_{1},d)\cap{B(x_{2},d)}}\Delta{h}(y)\Big(\frac{|x_{2}-y|^{2}-|x_{1}-y|^{2}}{d^{2}}\Big)dA(y)+\int_{B(x_{2},d)\setminus{B(x_{1},d)}}\Delta{h}(y)\Big(1-\frac{|x_{2}-y|^{2}}{d^{2}}\Big)dA(y) \\ &-\int_{B(x_{1},d)\setminus{B(x_{2},d)}}\Delta{h}(y)\Big(1-\frac{|x_{1}-y|^{2}}{d^{2}}\Big)dA(y)\Big| \\ 
    \leq&{\frac{15\|\Delta{h}\|_{\infty}d|x_{2}-x_{1}|}{8}} \\
    \vspace{5pt}
    |(3)|=&\frac{1}{2\pi}\Big|\int_{B(x_{1},d)\cap{B(x_{2},d)}}\Delta{h}(y)\log\Big(\frac{|y-x_{2}|}{|y-x_{1}|}\Big)dA(y)+\int_{B(x_{2},d)\setminus{B(x_{1},d)}}\Delta{h}(y)\log\Big(\frac{|x_{2}-y|}{d}\Big)dA(y) \\ &-\int_{B(x_{1},d)\setminus{B(x_{2},d)}}\Delta{h}(y)\log\Big(\frac{|x_{1}-y|}{d}\Big)dA(y)\Big| \\
    \leq&\frac{\|\Delta{h}\|_{\infty}}{2\pi}\Big(\int_{B(x_{2},d)}\log\Big(1+\frac{|x_{2}-x_{1}|}{|y-x_{2}|}\Big)dA(y)+\int_{B(x_{2},d)\setminus{B(x_{2},d-|x_{2}-x_{1}|)}}\log\Big(\frac{d}{|x_{2}-y|}\Big)dA(y) \\
    &+\int_{B(x_{1},d)\setminus{B(x_{1},d-|x_{2}-x_{1}|)}}\log\Big(\frac{d}{|x_{1}-y|}\Big)dA(y)\Big) \\
    =&\|\Delta{h}\|_{\infty}\Big(\int_{0}^{d}r\log\Big(1+\frac{|x_{1}-x_{2}|}{r}\Big)dr+2\int^{d}_{d-|x_{2}-x_{1}|}r\log\Big(\frac{d}{r}\Big)dr\Big)\lesssim{\|\Delta{h}\|_{\infty}d|x_{2}-x_{1}|}
\end{align*}
Putting all this together, we have that: 
\begin{equation*}
    |h(x_{2})-h(x_{1})|\lesssim{\|h\|_{\infty}\Big(\frac{|x_{2}-x_{1}|}{d}\Big)+\|\Delta{h}\|_{\infty}d|x_{2}-x_{1}|}
\end{equation*}
\end{proof}
\noindent Recall that Proposition \ref{prop: convolution has small Laplacian} tells us that if $h^{\bullet}$ is a discrete harmonic function, its convolution with a smooth mollifier, $\phi\ast{h^{\bullet}}$, is almost harmonic in that $\Delta{(\phi\ast{h^{\bullet}})}\approx{0}$.  Hence, taking $h=(\phi\ast{h^{\bullet}})$ in our estimate from Proposition \ref{prop: almost harmonic functions satisfy a Harnack type estimate}, we have that the convolution of a discrete harmonic function with a smooth mollifier, satisfies a Harnack-type estimate. As a consequence, we can recover a Harnack-type estimate for discrete harmonic functions on orthodiagonal maps, on a mesoscopic scale:
\begin{prop}
\label{prop: Lipschitz regularity on a mesoscopic scale (base case)}
There exists an absolute constant $C>0$ so that if $G=(V^{\bullet}\sqcup{V^{\circ}}, E)$ is an orthodiagonal map with edges of length at most $\varepsilon$, $h^{\bullet}:V^{\bullet}\rightarrow{\mathbb{R}}$ is harmonic, and $z,w\in{V^{\bullet}}$ are vertices of $V^{\bullet}$ so that $|z-w|\geq{d\big({\small \frac{\varepsilon}{d}}\big)^{{\frac{\beta^{2}}{2(1+\beta)}}}}$ where $d=\text{dist}(z,\partial{\widehat{G}})\wedge{\text{dist}(w,\partial{\widehat{G}})}$, then:
 \begin{equation*}
        |h^{\bullet}(z)-h^{\bullet}(w)|\leq{C\|h^{\bullet}\|\Big(\frac{|z-w|}{d}\Big)}
\end{equation*}
\end{prop}
\begin{proof}
    Let $\phi$ be a smooth mollifier supported on the unit ball $B(0,1)\subseteq{\mathbb{R}^{2}}$. Then for any $\delta>0$, $\phi_{\delta}(x):=\delta^{-2}\phi(\delta^{-1}x)$ is a smooth mollifier supported on $B(0,\delta)$. Suppose $z$ and $w$ are vertices of $V^{\bullet}$. By the triangle inequality: 
    \begin{equation*}
        |h^{\bullet}(w)-h^{\bullet}(z)|\leq{|h^{\bullet}(w)-(\phi_{\delta}\ast{h^{\bullet}})(w)|+|(\phi_{\delta}\ast{h^{\bullet}})(w)-(\phi_{\delta}\ast{h^{\bullet}})(z)|+|(\phi_{\delta}\ast{h^{\bullet}})(z)-h^{\bullet}(z)|}
    \end{equation*}
    By Lemma \ref{lem: discrete harmonic functions are beta Holder in the bulk}:
    \begin{equation}
    \label{eqn: difference between h- bullet and its mollification}
        |(\phi_{\delta}\ast{h^{\bullet}})(z)-h^{\bullet}(z)|, |(\phi_{\delta}\ast{h^{\bullet}})(w)-h^{\bullet}(w)|=O(\delta^{\beta}\cdot{d^{-\beta}}\cdot\|h^{\bullet}\|) 
    \end{equation}
    On the other hand, applying Proposition \ref{prop: almost harmonic functions satisfy a Harnack type estimate} with $h=(\phi_{\delta}\ast{h^{\bullet}})$ and using our estimate for the Laplacian of $(\phi_{\delta}\ast{h^{\bullet}})$ from Proposition \ref{prop: convolution has small Laplacian}, we have that: 
    \begin{align}
        \label{eqn: modulus of continuity of delta- mollified h- bullet}
        |(\phi_{\delta}\ast{h^{\bullet}})(w)-(\phi_{\delta}\ast{h^{\bullet}})(z)|=&O(\|h^{\bullet}\|\Big(\frac{|z-w|}{d}\Big))+O(\|\Delta(\varphi_{\delta}\ast{h^{\bullet}})\|\cdot{d}\cdot{|z-w|}) \notag \\ 
        =&O\big(\|h^{\bullet}\|\Big(\frac{|z-w|}{d}\Big)\big)+O\big(\varepsilon^{\frac{\beta}{1+\beta}}\delta^{-2}d^{2-\frac{\beta}{1+\beta}}\|h^{\bullet}\|\Big(\frac{|z-w|}{d}\Big)\big) \\
        &+O\big(\varepsilon^{\frac{1}{2}}\delta^{-\frac{5}{2}}d^{2}\|h^{\bullet}\|\Big(\frac{|z-w|}{d}\Big)\big) \notag
    \end{align}
    Note that having chosen a particular smooth mollifer $\phi$, we can disregard the $\|D^{2}\phi\|$ and $\|D^{3}\phi\|$ terms in Proposition \ref{prop: convolution has small Laplacian}, since they are just constants. Looking at the error term in Equation \ref{eqn: difference between h- bullet and its mollification}, to get the kind of estimate for $h^{\bullet}$ we are looking for, we need it to be the case that $\delta^{\beta}d^{-\beta}\leq{\big(\frac{|z-w|}{d}\big)} \hspace{5pt} \Longleftrightarrow \hspace{5pt} \delta\leq{d\big(\frac{|z-w|}{d}\big)^{1/\beta}}$. On the other hand, looking at the estimate for the modulus of continuity of $(\varphi_{\delta}\ast{h^{\bullet}})$ in Equation \ref{eqn: modulus of continuity of delta- mollified h- bullet}, all of the powers of $\delta$ are negative. Hence, to get the best estimate possible, we should take $\delta$ to be as large as possible. Namely, $\delta=d\big(\frac{|z-w|}{d}\big)^{1/\beta}$. Plugging this choice of $\delta$ into Equations \ref{eqn: difference between h- bullet and its mollification} and \ref{eqn: modulus of continuity of delta- mollified h- bullet} and putting all this together, we get that: 
    \begin{align}
        |h^{\bullet}(w)-h^{\bullet}(z)|&=O\big(\|h^{\bullet}\|\Big(\frac{|z-w|}{d}\Big)\big)+O\big(\varepsilon^{\frac{\beta}{1+\beta}}\cdot|z-w|^{-\frac{2}{\beta}}\cdot{d^{\frac{2}{\beta}-\frac{\beta}{1+\beta}}}\cdot\|h^{\bullet}\|\Big(\frac{|z-w|}{d}\Big)\big) \\
        &+O\big(\varepsilon^{\frac{1}{2}}\cdot{|z-w|^{-\frac{5}{2\beta}}}\cdot{d^{\frac{5}{2\beta}-\frac{1}{2}}}\cdot\|h^{\bullet}\|\Big(\frac{|z-w|}{d}\Big)\big) \notag
    \end{align}
    To get an effective estimate, we need it to be the case that:
    \begin{align*}
        \varepsilon^{\frac{\beta}{1+\beta}}\cdot|z-w|^{-\frac{2}{\beta}}\cdot{d^{\frac{2}{\beta}-\frac{\beta}{1+\beta}}}&\leq{1}  & &\Longleftrightarrow & |z-w|&\geq{d\Big({\small \frac{\varepsilon}{d}}\Big)^{\frac{\beta^{2}}{2(1+\beta)}}}\\
        \varepsilon^{\frac{1}{2}}\cdot{|z-w|^{-\frac{5}{2\beta}}}\cdot{d^{\frac{5}{2\beta}-\frac{1}{2}}}&\leq{1} & &\Longleftrightarrow & |z-w|&\geq{d\Big({\small \frac{\varepsilon}{d}}\Big)^{\frac{\beta}{5}}}
    \end{align*}
    Thus, as long as $|z-w|\geq{d\big({\small \frac{\varepsilon}{d}}\big)^{\frac{\beta}{5}\wedge{\frac{\beta^{2}}{2(1+\beta)}}}}$, we have that:
    \begin{equation*}
        |h^{\bullet}(z)-h^{\bullet}(w)|\leq{C\|h^{\bullet}\|\Big(\frac{|z-w|}{d}\Big)}
    \end{equation*}
   where $C>0$ is some absolute constant. Since the absolute constant $\beta>0$ from Lemma \ref{lem: discrete harmonic functions are beta Holder in the bulk} is small, $\frac{\beta^{2}}{2(1+\beta)}<\frac{\beta}{5}$.
\end{proof} 
\noindent Notice that Proposition \ref{prop: Lipschitz regularity on a mesoscopic scale (base case)} is a weaker version of Theorem \ref{thm: optimal Lipschitz regularity on mesoscopic scales} that requires a larger mesoscopic scale for our Harnack-type estimate to kick in. In Section \ref{subsec: bootstrap argument}, we will see how, using Proposition \ref{prop: Lipschitz regularity on a mesoscopic scale (base case)} as a starting point, we can use a bootstrap argument to successively improve the scale on which our Harnack estimate holds, giving us Theorem \ref{thm: optimal Lipschitz regularity on mesoscopic scales}.  

\subsection{Refining our Mesoscopic Scale}
\label{subsec: bootstrap argument}

In this section, we will refine our estimate in Proposition \ref{prop: Lipschitz regularity on a mesoscopic scale (base case)} by improving the mesoscopic scale on which our Harnack-type estimate holds. To do this, we first observe that discrete harmonic functions on orthodiagonal maps are $\beta$- H\"older in the bulk on small scales and Lipschitz in the bulk on large scales. This gives us improved H\"older regularity in the bulk on intermediate scales:
\begin{prop}
\label{prop: refined Holder regularity on intermediate scales}
(refined H\"older regularity on intermediate scales) Suppose that for any orthodiagonal map $G$ with edges of length at most $\varepsilon$ and any harmonic function $h:V^{\bullet}\rightarrow{\mathbb{R}}$, we have that:
\begin{itemize}
    \item (Lipschitz regularity on large scales) if $z,w\in{V^{\bullet}}$ satisfy $|z-w|\geq{d\big({\small \frac{\varepsilon}{d}}\big)^{\alpha}}$ for some fixed $\alpha\in{(0,1)}$, then:
    \begin{equation*}
        |h(z)-h(w)|\leq{C\|h\|\Big({\small \frac{|z-w|}{d}}\Big)}
    \end{equation*}
    where $C>0$ is an absolute constant. 
\end{itemize}
Then:
\begin{itemize}
    \item (Improved H\"older regularity on intermediate scales) for any $\gamma\in{(\alpha,1)}$, if $G$ is an orthodiagonal map with edges of length at most $\varepsilon$, $h:V^{\bullet}\rightarrow{\mathbb{R}}$ is harmonic, and $z,w\in{V^{\bullet}}$ satisfy $|z-w|=d\big({\small \frac{\varepsilon}{d} }\big)^{\gamma}$ where $d=\text{dist}(z,\partial{\widehat{G}})\wedge{\text{dist}(w,\partial{\widehat{G}})}$, then: 
    \begin{equation*}
        |h(z)-h(w)|\leq{C\|h\|\Big(\frac{|z-w|}{d}\Big)^{\beta+\frac{\alpha}{\gamma}(1-\beta)}}
    \end{equation*}
\end{itemize}

\end{prop}

\begin{proof}
    Suppose $G$ is an orthodiagonal map with edges of length at most $\varepsilon$ and $z,w\in{V^{\bullet}}$ satisfy $|z-w|=d\big({\small \frac{|z-w|}{d}}\big)^{\gamma}$ for some $\gamma\in{(\alpha,1)}$. Suppose $h:V^{\bullet}\rightarrow{\mathbb{R}}$ is harmonic and WLOG, $h(w)\leq{h(z)}$. Let $B=B_{G}(w,d\big({\small \frac{|z-w|}{d}}\big)^{\alpha})$ be the discrete ball of radius $d\big({\small\frac{|z-w|}{d}}\big)^{\alpha}$ centered at $w$ in $G$. That is, $B=(V_{B}^{\bullet}\sqcup{V_{B}^{\circ}},E_{B})$ is the suborthodiagonal map of $G$ consisting of all the quadrilaterals of $G$ with a primal vertex in the open ball $B(w,d\big({\small \frac{|z-w|}{d}}\big)^{\alpha})$. Then: 
    \begin{align*}
        \text{Int}(V_{B}^{\bullet})&=V^{\bullet}\cap{B(w,d\Big(\frac{|z-w|}{d}\Big)^{\alpha})} \\
        \partial{V_{B}}^{\bullet}&=\{v\in{V^{\bullet}}:v\notin{B(w,d\Big({\small \frac{|z-w|}{d}}\Big)^{\alpha})},\hspace{2pt}u\sim{v} \hspace{3pt}\text{for some  $u\in{V^{\bullet}\cap{B(w,d\Big({\small \frac{|z-w|}{d}}\Big)^{\alpha})}}$}\}
    \end{align*}
    By the maximum principle, we can find a nearest-neighbor path $\gamma=(w_{0}, w_{1}, ..., w_{m})$ of vertices in $G^{\bullet}$ so that $w_{0}=w$, $w_{m}\in{\partial{V_{B}^{\bullet}}}$ and $h(w_{i+1})\leq{h(w_{i})}$ for all $i$. In particular, it follows that $h(w_{i})\leq{h(w)}$ for all $i$. If $(S_{n})_{n\geq{0}}$ is a simple random walk on $G^{\bullet}$, and $\tau_{\gamma}$ and $\tau_{\partial{V_{B}^{\bullet}}}$ are the hitting times of $\gamma$ and $\partial{V_{B}^{\bullet}}$ by our random walk, by the optional stopping theorem, we have that: 
    \begin{align*}
        h(z)-h(w)&=\mathbb{E}^{z}\big(h(S_{\tau_{\gamma}\wedge{\tau_{\partial{V_{B}^{\bullet}}}}})-h(w)\big)=\mathbb{E}^{z}\big(\underbrace{(h(S_{\tau_{\gamma}})-h(w))}_{\leq{0}}1_{\tau_{\gamma}\leq{\tau_{\partial{V_{B}^{\bullet}}}}}\big)+\mathbb{E}^{z}\big((h(S_{\tau_{\partial{B}}})-h(w))1_{\tau_{\partial{V_{B}^{\bullet}}}<{\tau_{\gamma}}}\big) \\
        &\leq{\max_{u\in{\partial{V_{B}^{\bullet}}}}|h(u)-h(w)|\cdot{\mathbb{P}(\tau_{\partial{V_{B}^{\bullet}}}<\tau_{\gamma})}}\lesssim{\|h\|\Big(\frac{d\big(\frac{\varepsilon}{d}\big)^{\alpha}}{d}\Big)\Big(\frac{d\big(\frac{\varepsilon}{d}\big)^{\gamma}}{d\big(\frac{\varepsilon}{d}\big)^{\alpha}}\Big)^{\beta}} \\
        &\lesssim{\|h\|\Big(\frac{\varepsilon}{d}\Big)^{\alpha+\beta(\gamma-\alpha)}}=\|h\|\Big(\frac{d\big(\frac{\varepsilon}{d}\big)^{\gamma}}{d}\Big)^{\beta+\frac{\alpha}{\gamma}(1-\beta)}=\|h\|\Big(\frac{|z-w|}{d}\Big)^{\beta+\frac{\alpha}{\gamma}(1-\beta)}
    \end{align*}
\end{proof}

\noindent From here. the story is as follows: 
\begin{enumerate}
    \item Observe that in our estimate for $\Delta{(\phi_{\delta}\ast{h^{\bullet}})}(z)$ in Proposition \ref{prop: convolution has small Laplacian}, the $\frac{\beta}{1+\beta}$ exponents in the second term come from the fact that harmonic functions are $\beta$- H\"older in the bulk on scales $\asymp\ell$ where $\varepsilon\ll{\ell}\ll{\delta}\leq{\frac{d}{2}}$.
    \item Hence, if we use Proposition \ref{prop: refined Holder regularity on intermediate scales} in place of Lemma \ref{lem: discrete harmonic functions are beta Holder in the bulk}, we can improve our estimate for $\Delta{(\phi_{\delta}\ast{h^{\bullet}})}(z)$.
    \item However, the scale on which we have that discrete harmonic functions on orthodiagonal maps are Lipschitz in the bulk in Proposition \ref{prop: Lipschitz regularity on a mesoscopic scale (base case)}, comes from: 
    \begin{enumerate}
        \item the fact that discrete harmonic functions on orthodiagonal maps are $\beta$-H\"older in the bulk, which is used on the intermediate scale $\delta$.
        \item our estimate for $\Delta{(\phi_{\delta}\ast{h^{\bullet}})}(z)$ in Proposition \ref{prop: convolution has small Laplacian}.
    \end{enumerate}
    \item Thus, our improved H\"older regularity on intermediate scales in Proposition \ref{prop: refined Holder regularity on intermediate scales} can be used to improve the scale at which we can ensure discrete harmonic functions are Lipschitz in the bulk in Proposition \ref{prop: Lipschitz regularity on a mesoscopic scale (base case)}.
    \item But now we can repeat this process! That is, we know that discrete harmonic functions are Lipschitz on a scale smaller than the one specified in Proposition \ref{prop: Lipschitz regularity on a mesoscopic scale (base case)}. This improves our H\"older regularity on intermediate scales in Proposition \ref{prop: refined Holder regularity on intermediate scales}. Thus, repeating steps 2 through 4 with this improved H\"older regularity on intermediate scales gives us an even smaller scale on which discrete harmonic functions are Lipschitz. We can repeat this process indefinitely, taking the scale on which we have Lipschitz regularity and plugging it into our argument in steps 2 through 4 to make it even smaller.
\end{enumerate}
In short, we have a bootstrap argument for refining the scale at which we know that discrete harmonic functions are Lipschitz in the bulk. This is encapsulated in the following result: 
\begin{prop}
\label{prop: the bootstrap}
(the bootstrap) Suppose we know that for some $\alpha\in{(0,1)}$, there exists an absolute constant $C>0$ such that for any orthodiagonal map $G=(V^{\bullet}\sqcup{V^{\circ}}, E)$ with edges of length at most $\varepsilon$, any harmonic function $h:V^{\bullet}\rightarrow{\mathbb{R}}$, and any $z,w\in{V^{\bullet}}$ satisfying $|z-w|\geq{d\big({\small \frac{\varepsilon}{d}}\big)^{\alpha}}$, where $d=\text{dist}(z,\partial{\widehat{G}})\wedge{\text{dist}(w,\partial{\widehat{G}})}$, we have that: 
\begin{equation*}
    |h(z)-h(w)|\leq{C\|h\|\Big(\frac{|z-w|}{d}\Big)}
\end{equation*}
Then taking $\alpha'=(1-\beta)\alpha+\beta\min\{\frac{1}{5}, \frac{\beta+\alpha-\beta\alpha}{2(1+\beta)}\}$, there exists an absolute constant $C'>0$ so that for any orthodiagonal map $G=(V^{\bullet}\sqcup{V^{\circ}}, E)$ with edges of length at most $\varepsilon$, any harmonic function $h:V^{\bullet}\rightarrow{\mathbb{R}}$, and any $z,w\in{V^{\bullet}}$ satisfying $|z-w|\geq{d\big({\small \frac{\varepsilon}{d}}\big)^{\alpha'}}$, where $d=\text{dist}(z,\partial{\widehat{G}})\wedge{\text{dist}(w,\partial{\widehat{G}})}$, we have that:
\begin{equation*}
    |h(z)-h(w)|\leq{C'\|h\|\Big(\frac{|z-w|}{d}\Big)}
\end{equation*}
\end{prop}
\begin{proof}
     Let $\phi$ be a smooth mollifier supported on the unit ball $B(0,1)\subseteq{\mathbb{R}^{2}}$. Then for any $\delta>0$, $\phi_{\delta}(x)=\delta^{-2}\phi(\delta^{-1}x)$ is a smooth mollifier, supported on $B(0,\delta)$. Just as in the proof of Proposition \ref{prop: almost harmonic functions satisfy a Harnack type estimate}, we convolve our discrete harmonic function $h$ with a smooth mollifier $\phi_{\delta}$, where $\delta$ is mesoscopic. With this in mind, we write $\delta=d\big({ \small \frac{\varepsilon}{d} } \big)^{c}$ and $\ell=d\big({ \small \frac{\varepsilon}{d} } \big)^{\gamma}$, where $0<c<\gamma<1$. Here, $\ell$ is some intermediate scale. Using  Proposition \ref{prop: refined Holder regularity on intermediate scales} in place of Lemma \ref{lem: discrete harmonic functions are beta Holder in the bulk}, we have that:
    \begin{equation*}
        |(\phi_{\delta}\ast{h})(z)-h(z)|, |(\phi_{\delta}\ast{h})(w)-h(w)|=O(\|h\|\Big(\frac{\varepsilon}{d}\Big)^{\beta{c}+\alpha(1-\beta)})
    \end{equation*}
    if $c>{\alpha}$. If $c\leq{\alpha}$, we have that:
    \begin{equation*}
        |(\phi_{\delta}\ast{h})(z)-h(z)|, |(\phi_{\delta}\ast{h})(w)-h(w)|=O(\|h\|\Big(\frac{\varepsilon}{d}\Big)^{c})
    \end{equation*}  
    In particular, notice that if we repeat our argument in Proposition \ref{prop: Lipschitz regularity on a mesoscopic scale (base case)}, picking a value of $c$ that is less than or equal to $\alpha$ will gives us, at best, Lipschitz regularity on scale $d\big({ \small \frac{\varepsilon}{d}}\big)^{c}\geq{d\big({ \small \frac{\varepsilon}{d}}\big)^{\alpha}}$. In short, we end up with an estimate that is inferior to the one we started with. Thus, if we want to improve on our initial estimate, we need to choose $c$ and $\gamma$ so that $0<\alpha<c<\gamma<1$.\\ \\
    Repeating our argument in Proposition \ref{prop: convolution has small Laplacian} using Proposition \ref{prop: refined Holder regularity on intermediate scales} in place of Lemma \ref{lem: discrete harmonic functions are beta Holder in the bulk}, we have the following analogue of Equation \ref{eqn: Laplacian estimate before optimizing in ell}:
    \begin{equation*}
        \Delta(\phi_{\delta}\ast{h})(z)=O(\|h\|\Big({\small \frac{\varepsilon}{d}}\Big)^{1-\gamma-2c}d^{-2})+O(\|h\|\Big({\small \frac{\varepsilon}{d}}\Big)^{\beta\gamma+\alpha(1-\beta)-2c}d^{-2})+O(\|h\|\Big({\small \frac{\varepsilon}{d}}\Big)^{\gamma-3c}d^{-2})
    \end{equation*}
   Optimizing in $\ell=d\big({ \small \frac{\varepsilon}{d} } \big)^{\gamma}$ for fixed $\delta=d\big({ \small \frac{\varepsilon}{d} } \big)^{c}$, we get that:
    \begin{equation}
    \label{eqn: booptstrap Laplacian estimate}
        \Delta(\phi_{\delta}\ast{h})(z)=O(\|h\|\Big({\small \frac{\varepsilon}{d}}\Big)^{\frac{1-5c}{2}}d^{-2})+O(\|h\|\Big({\small \frac{\varepsilon}{d}}\Big)^{\frac{\beta+\alpha-\alpha\beta}{(1+\beta)}-2c}d^{-2})
    \end{equation}
    Applying the estimate in Proposition \ref{prop: almost harmonic functions satisfy a Harnack type estimate} to $\phi_{\delta}\ast{h}$ with this Laplacian estimate, we have that: 
    \begin{align*}
        |(\phi_{\delta}\ast{h})(z)-(\phi_{\delta}\ast{h})(w)|=O\big(\|h\|\Big(\frac{|z-w|}{d}\Big)\big)+O\big(\|h\|\Big(\frac{|z-w|}{d}\Big)\Big( \frac{\varepsilon}{d}\Big)^{(\frac{1-5c}{2})\wedge{(\frac{\beta+\alpha-\alpha\beta}{(1+\beta)}}-2c)}\big)
    \end{align*}
    Combining this with our estimate for the difference between $h$ and $\phi_{\delta}\ast{h}$ from earlier, we have that: 
    \begin{equation*}
        |h(z)-h(w)|=O\big(\|h\|\Big(\frac{|z-w|}{d}\Big)\big)+O\big(\|h\|\Big(\frac{|z-w|}{d}\Big)\Big( \frac{\varepsilon}{d}\Big)^{(\frac{1-5c}{2})\wedge{(\frac{\beta+\alpha-\alpha\beta}{(1+\beta)}}-2c)}\big)+O\big(\|h\|\Big(\frac{\varepsilon}{d}\Big)^{\beta{c}+\alpha(1-\beta)}\big)
    \end{equation*}
    To find the scale on which we can ensure Lipschitz regularity of $h$, we want to pick $c$ as large as possible subject to the constraits: 
    \begin{align*}
        1-5c&\geq{0}, & \frac{\beta+\alpha(1-\beta)}{1+\beta}-2c&\geq{0}
    \end{align*}
    Thus, taking $c=\min\{\frac{1}{5}, \frac{\beta+\alpha-\alpha\beta}{2(1+\beta)}\}$, our estimate for $|h(z)-h(w)|$ above tells us that there exists an absolute constant $C'>0$ such that:
    \begin{equation*}
        |h(z)-h(w)|\leq{C'\Big(\frac{|z-w|}{d}\Big)}
    \end{equation*}
    for any $z,w\in{V^{\bullet}}$ such that $|z-w|\geq{d\big({\small \frac{\varepsilon}{d}}\big)^{\beta\alpha'+(1-\beta)\alpha}}$, where $\alpha'=\min\{\frac{1}{5}, \frac{\beta+\alpha-\alpha\beta}{2(1+\beta)}\}$.
\end{proof}
\noindent Theorem \ref{thm: optimal Lipschitz regularity on mesoscopic scales} now follows as a straightforward corollary of Proposition \ref{prop: Lipschitz regularity on a mesoscopic scale (base case)}, which serves as our base case, and  Proposition \ref{prop: the bootstrap}, which tells us how we can successively refine our mesoscopic scale:
\begin{proof}
    (of Theorem \ref{thm: optimal Lipschitz regularity on mesoscopic scales}) Consider the sequence $(\alpha_{n})_{n\geq{0}}$ such that $\alpha_{0}=\frac{\beta^{2}}{2(1+\beta)}$, and $\alpha_{n+1}=(1-\beta)\alpha_{n}+\beta\cdot\min\{\frac{1}{5}, \frac{\beta+\alpha_{n}-\beta\alpha_{n}}{2(1+\beta)}\}$ for all $n\geq{0}$. Combining our results in Proposition \ref{prop: Lipschitz regularity on a mesoscopic scale (base case)} and Proposition \ref{prop: the bootstrap}, we have that for all $n\in{\mathbb{N}_{0}}$, if $G=(V^{\bullet}\sqcup{V^{\circ}}, E)$ is an orthodiagonal map with edges of length at most $\varepsilon$, $h:V^{\bullet}\rightarrow{\mathbb{R}}$ is harmonic on $G^{\bullet}$, and $z,w\in{V^{\bullet}}$ satisfy $|z-w|\geq{d\big(\frac{\varepsilon}{d}\big)^{\alpha_{n}}}$, where $d=\text{dist}(z,\partial{G})\wedge{\text{dist}(w,\partial{G})}$, then:
    \begin{equation*}
        |h(z)-h(w)|\leq{C_{n}\|h\|\Big(\frac{|z-w|}{d}\Big)}
    \end{equation*}
    where $C_{n}>0$ is some absolute constant. Hence, to prove the desired result, it suffices to show that  $\lim\limits_{n\rightarrow{\infty}}\alpha_{n}=\frac{\beta}{1+3\beta}$.\\ \\
    By our recursion, $\alpha_{n+1}$ is the weighted average of $\alpha_{n}$ and the minimum of $\frac{1}{5}$ and $\frac{\beta+\alpha_{n}-\beta\alpha_{n}}{2(1+\beta)}$. Hence, if $\alpha_{0}<\frac{1}{5}$, it follows that $\alpha_{n}<\frac{1}{5}$ for all $n\in{\mathbb{N}_{0}}$. Additionally, observe that: 
    \begin{equation*}
        \frac{1}{5}<\frac{\beta+\alpha_{n}-\beta\alpha_{n}}{2(1+\beta)} \Longleftrightarrow{\alpha_{n}>{\frac{2-3\beta}{5(1-\beta)}}}
    \end{equation*}
    Since the function $f(\beta)=\frac{2-3\beta}{5(1-\beta)}$ is strictly decreasing on $(0,1)$, $\beta\in{(0,1/2]}$ and $f(1/2)=1/5$, we see that the minimum of $\frac{1}{5}$ and $\frac{\beta+\alpha_{n}-\beta\alpha_{n}}{2(1+\beta)}$ being equal to $\frac{1}{5}$ requires that $\alpha_{n}\geq{\frac{1}{5}}$. Thus, for $\alpha_{0}=\frac{\beta^{2}}{2(1+\beta)}<\frac{1}{5}$, our recurrence simplifies to: 
    \begin{equation*}
        \alpha_{n+1}=(1-\beta)\alpha_{n}+\beta\Big(\frac{\beta+\alpha_{n}-\beta\alpha_{n}}{2(1+\beta)}\Big)
    \end{equation*}
    Again, since $\alpha_{n+1}$ is a weighted average of $\alpha_{n}$ and $\frac{\beta+\alpha_{n}-\beta\alpha_{n}}{2(1+\beta)}$, 
    \begin{equation*}
        \alpha_{n+1}\geq{\alpha_{n}} \hspace{5pt}\Longleftrightarrow \hspace{5pt}\frac{\beta+\alpha_{n}-\beta\alpha_{n}}{2(1+\beta)}\geq{\alpha_{n}} \hspace{5pt}\Longleftrightarrow\hspace{5pt}{\alpha_{n}\leq{\frac{\beta}{1+3\beta}}}
    \end{equation*}
    On the other hand, notice that if $\alpha_{n}<{\frac{\beta}{1+3\beta}}$ then $\frac{\beta+\alpha_{n}-\beta\alpha_{n}}{2(1+\beta)}=\frac{\beta+(1-\beta)\alpha_{n}}{2(1+\beta)}<{\frac{\beta+(1-\beta)\frac{\beta}{1+3\beta}}{2(1+\beta)}}=\frac{\beta}{1+3\beta}$ which tells us that $\alpha_{n+1}<{\frac{\beta}{1+3\beta}}$, since $\alpha_{n+1}$ is the weighted average of this quantity and $\alpha_{n}$. Since $\beta\in{(0,1/2)}$, $\alpha_{0}=\frac{\beta^{2}}{2(1+\beta)}<{\frac{\beta}{1+3\beta}}$. Thus, in our case, the sequence $(\alpha_{n})_{n\geq{0}}$ is strictly increasing and bounded above by $\frac{\beta}{1+3\beta}$, and so converges to a limit $\lambda$ as $n$ tends to $\infty$. This limit satisfies: 
    \begin{equation*}
        \lambda=(1-\beta)\lambda+\beta\Big(\frac{\beta+\lambda-\beta\lambda}{2(1+\beta)}\Big) \hspace{5pt}\Longleftrightarrow \hspace{5pt}\lambda=\frac{\beta}{1+3\beta}
    \end{equation*}
    This completes our proof.
\end{proof}    

\begin{appendices}
    \section{Random Walks and A Priori Regularity Estimates for Harmonic Functions on Orthodiagonal Maps}
\label{sec: Random Walks and A Priori Regularity Estimates for Harmonic Functions on Orthodiagonal Maps}

As per our discussion in Section \ref{subsec: outline of proof near boundary}, in this section, we will give self-contained proofs of Lemmas \ref{lem: weak Beurling on OD maps} and \ref{lem: discrete harmonic functions are beta Holder in the bulk} by following the strategy used by Chelkak to prove analogous results in \cite{Ch16}. Suppose $G=(V,E,c)$ is a  weighted plane graph, not necessarily finite. Let $(S_{n})_{n\geq{0}}$ be a simple random walk on $G$. For any $x\in{\mathbb{C}}$ and $r>0$ such that $B(x,r)\subseteq{\widehat{G}}$, let $T_{\partial{B(x,r)}}$ be the first time at which our simple random walk exits the ball $B(x,r)$: 
\begin{equation*}
    T_{\partial{B(x,r)}}:=\inf\{n\in\mathbb{N}_{0}:S_{n}\notin{B(x,r)}\}
\end{equation*}
We say that $G$ satisfies \textbf{Property (S)} if there exist constants $\eta\in{(0,\pi)}$ and $c>0$ such that for any $v\in{V}$ and $r>0$ satisfying $B(v,r)\subseteq{\widehat{G}}$,
\begin{equation*}
    \mathbb{P}^{v}\big(\text{arg}(S_{T_{\partial{B(v,r)}}}-v)\in{I}\big)\geq{c}
\end{equation*}
for any interval $I\subseteq{S^{1}}$ with $\text{length}(I)\geq{\eta}$. Informally, Property (S) tells us that a simple random walk started at the center of the disk $B(v,r)$ has probability at least $c$ of exiting this disk through a discrete arc of length $\eta$. Proposition 6.7 of \cite{CLR23} says that Property (S) holds for T-graphs satisfying the condition ``$Lip(\kappa,\delta)$" for some $\kappa\in{(0,1)}$ and $\delta>0$. As per our discussion in Section \ref{subsec: outline of proof near boundary}, for any $\kappa\in{(0,1)}$ there exists $c=c(\kappa)>1$ so that an orthodiagonal map with edges of length at most $\varepsilon$ satisfies the condition $Lip(\kappa,c\varepsilon)$. Hence, Property (S) for orthodiagonal maps follows immediately from Proposition 6.7 of \cite{CLR23}. For the convenience of the reader, we give a self-contained proof of Property (S) for orthodiagonal maps without reference to the language of t-embeddings. Having proven this, Lemmas \ref{lem: weak Beurling on OD maps} and \ref{lem: discrete harmonic functions are beta Holder in the bulk} readily follow.   

\subsection{Microscale Property (S) on Orthodiagonal Maps}
\label{subsec: property (S)}

Theorem 1.1 of \cite{GJN20} says the following:

\begin{thm}
\label{thm: GJN20}
    (Theorem 1.1 of \cite{GJN20}) Suppose $\Omega$ is a bounded simply connected domain, $g:\mathbb{R}^{2}\rightarrow{\mathbb{R}}$ is a $C^{2}$ function. Given $\delta, \varepsilon\in{(0,\text{diam}(\Omega))}$, let $G=(V^{\bullet}\sqcup{V^{\circ}}, E)$ be a simply connected orthodiagonal map, with edges of length at most $\varepsilon$ such that the Hausdorff distance between $\partial{V^{\bullet}}$ and $\partial{\Omega}$  is at most $\delta$. Let $h_{c}$ be the solution to the continuous Dirichlet problem on $\Omega$ with boundary data $g$ and let $h_{d}:V^{\bullet}\rightarrow{\mathbb{R}}$ be the solution to the discrete Dirichlet problem on $G^{\bullet}$ with boundary data $g|_{\partial{V^{\bullet}}}$. Set:
    \begin{align*}
        C_{1}&:=\sup_{x\in{\widetilde{\Omega}}}|\nabla{g(x)}|, & C_{2}&:=\sup_{x\in{\widetilde{\Omega}}}\|D^{2}g(x)\|_{2}
    \end{align*}
    where $\widetilde{\Omega}=\text{conv}(\overline{\Omega}\cup{\widehat{G}})$. Then there exists an absolute constant $C>0$ such that for all $x\in{V^{\bullet}\cap{\Omega}}$,
    $$
    |h_{d}(x)-h_{c}(x)|\leq{\frac{C\text{diam}(\Omega)(C_{1}+C_{2}\varepsilon)}{\log^{1/2}(\text{diam}(\Omega)/(\delta\vee\varepsilon))}}
    $$
\end{thm} 
\noindent Property (S) for simple random walks on orthodiagonal maps follows from Theorem \ref{thm: GJN20} by comparing $h_{d}$ and $h_{c}$ for a suitable choice of boundary data $g$:  
\begin{lem}
\label{lem: property (S)}
    (Microscale Property (S)) Suppose $G=(V^{\bullet}\sqcup{V^{\circ}},E)$ is an orthodiagonal map with edges of length at most $\varepsilon$, $u\in{V^{\bullet}}$ is a primal vertex of $G$, and $R>0$ is a positive real number so that $B(u,R)\subseteq{\widehat{G}}$. Let $(S_{n})_{n=0}^{\infty}$ be a simple random walk on $G^{\bullet}$ and let $T_{\partial{B(u,R)}}$ be the first time at which our random walk exits the open ball $B(u,R)$: 
    \begin{equation*}
        T_{\partial{B(u,R)}}:=\inf\{n\in{\mathbb{N}_{0}}: S_{n}\notin{B(u,R)}\}
    \end{equation*}
    Then there exist absolute constants $c,C>0$, $\eta\in{(0,\pi)}$ such that:
    \begin{equation*}
         \mathbb{P}^{u}\big(\text{arg}(S_{T_{\partial{B(u,R)}}}-u)\in{I}\big)\geq{c}
    \end{equation*}
    for any interval $I\subset{S^{1}}$ with $\text{length}(I)\geq{\eta}$, provided that $R\geq{C\varepsilon}$.
\end{lem}
\noindent We call this a microscale estimate because it holds for all balls whose radius is at least a constant multiple of the mesh of our orthodiagonal map.
\begin{proof}
    Suppose $G=(V^{\bullet}\sqcup{V^{\circ}},E)$ is an orthodiagonal map with edges of length at most $\varepsilon$, $u\in{V^{\bullet}}$, and $R>0$ is a positive real number so that $B(u,R)\subseteq{\widehat{G}}$.  Let $B=B_{G}(u,R)$ be the discrete ball of radius $R$ centered at $u$ in $G$. That is, $B=(V_{B}^{\bullet}\sqcup{V_{B}^{\circ}},E_{B})$ is the suborthodiagonal map of $G$ consisting of all the quadrilaterals of $G$ with a primal vertex in the open ball $B(u,R)$. Then: 
    \begin{align*}
        \text{Int}(V_{B}^{\bullet})&=V^{\bullet}\cap{B(u,R)} \\
        \partial{V_{B}}^{\bullet}&=\{v\in{V^{\bullet}}:v\notin{B(u,R)},\hspace{2pt}u\sim{v} \hspace{3pt}\text{for some  $u\in{V^{\bullet}\cap{B(u,R)}}$}\}
    \end{align*}
    Since the edges of $G$ all have length at most $\varepsilon$, the edges of $G^{\bullet}$ have length at most $2\varepsilon$. Hence, the boundary vertices of $B^{\bullet}$ all lie within $2\varepsilon$ of $\partial{B(u,R)}$. Let $\varphi:\mathbb{R}^{2}\rightarrow{\mathbb{R}}$ be a smooth function with compact support such that: 
\begin{itemize}
    \item $\varphi(x)=1$ if $\|x\|\leq{\frac{1}{2}}$
    \item $\varphi(x)\in{(0,1)}$ if $\frac{1}{2}<\|x\|<1$. 
    \item $\varphi(x)=0$ if $\|x\|\geq{1}$
\end{itemize} 
Suppose $v\in{\partial{B(u,R)}}$. Consider the function $g:\mathbb{R}^{2}\rightarrow{\mathbb{R}}$ defined as follows: 
\begin{equation*}
    g(x)=\varphi\Big(\frac{x-v}{R}\Big)
\end{equation*}
Observe that: 
\begin{itemize}
    \item $g(x)=1$ for $x\in{\partial{B(u,R)}}$ such that $|\text{arg}(x-v)|\leq{\arccos{(\frac{7}{8})}}$.
    \item $g(x)\in{(0,1)}$ for $x\in{\partial{B(u,R)}}$ such that $\arccos{(\frac{7}{8})}<|\text{arg}(x-v)|<\arccos{(\frac{1}{2})}$.
    \item $g(x)=0$ for $x\in{\partial{B(u,R)}}$ if $|\text{arg}(x-v)|\geq{\arccos{(\frac{1}{2})}}$.
\end{itemize}
Here $\arccos(x)$ takes values in $[0,\pi]$ for $x\in{[-1,1]}$. Let $h_{d}:V_{B}^{\bullet}\rightarrow{\mathbb{R}}$ be the solution to the following boundary value problem on $B^{\bullet}$: 
\begin{align*}
    \Delta^{\bullet}h_{d}(x)&=0 \hspace{27pt}\text{for all $x\in{\text{Int}(V^{\bullet}_{B}})$}\\ 
    h_{d}(x)&=g(x) \hspace{14pt}\text{for all $x\in{\partial{V^{\bullet}_{B}}}$}
\end{align*}
If $(S_{n})_{n=1}^{\infty}$ is a simple random walk on $G^{\bullet}$ and $T_{\partial{B(u,R)}}=\inf\{n\in{\mathbb{N}_{0}}:S_{n}\notin{B(u,R)}\}$ is the time at which our random walk exits $B(u,R)$, by the maximum principle for discrete harmonic functions:
\begin{equation*}
    \mathbb{P}^{u}\big(|\text{arg}(S_{T_{\partial{B(u,R)}}}-u)|\leq{\arccos{(1/2)}}\big)\geq{h_{d}(u)}
\end{equation*}
provided that $R\geq{2\varepsilon}$. Notice that:
$$ \|\nabla{g}\|=R^{-1}\|\nabla\phi\|, \hspace{20pt} \|D^{2}g\|=R^{-2}\|{D^{2}}\phi\|
$$
Let $h_{c}:\overline{B(u,R)}\rightarrow{\mathbb{R}}$ be the solution to the corresponding continuous boundary value problem on $B(u,R)$: 
\begin{align*}
    \Delta{h_{c}(x)}&=0 \hspace{17pt}\text{for all $x\in{B(u,R)}$} \\ 
    h_{c}(x)&=g(x) \hspace{5pt}\text{for all $x\in{\partial{B(u,R)}}$}
\end{align*}
Since $g\equiv{1}$ on a boundary arc along $\partial{B(u,R)}$ of length $2R\cdot\arccos{(7/8)}$, 
\begin{equation*}
    h_{c}(u)\geq{\frac{\arccos(7/8)}{\pi}}
\end{equation*}
Applying Theorem \ref{thm: GJN20} with $\Omega=B(u,R)$, $G=B_{G}(u,R)$, and $g$ as above, we have that: 
\begin{equation*}
    |h_{d}(u)-h_{c}(u)|\leq{\frac{C(2R)(R^{-1}\|\nabla{\varphi}\|+R^{-2}\|D^{2}\varphi\|\varepsilon)}{\log^{1/2}(2R/\varepsilon)}}=\frac{2C(\|\nabla{\varphi}\|+\|D^{2}\varphi\|\varepsilon{R^{-1}})}{\log^{1/2}(2R/\varepsilon)}
\end{equation*}
Taking $R\geq{C'\varepsilon}$ for some absolute constant $C'>0$ sufficiently large, we can ensure that: 
\begin{equation*}
    \frac{2C(\|\nabla{\varphi}\|+\|D^{2}\varphi\|\varepsilon{R^{-1}})}{\log^{1/2}(2R/\varepsilon)}\leq{\frac{\arccos(7/8)}{2\pi}}
\end{equation*}
Putting all this together, we get: 
\begin{align*}
    \mathbb{P}^{u}\big(|\text{arg}(S_{T_{\partial{B(u,R)}}}-u)|\leq{\arccos{(1/2)}}\big)\geq{h_{d}(u)}\geq{h_{c}(u)-|h_{c}(u)-h_{d}(u)|}\geq{\frac{\arccos(7/8)}{2\pi}}
\end{align*}
provided that $R\geq{C'\varepsilon}$ for some absolute constant $C'>0$.
\end{proof}

\subsection{Proof of Lemmas \ref{lem: weak Beurling on OD maps} and \ref{lem: discrete harmonic functions are beta Holder in the bulk}}
\label{subsec: toolbox}

The following result is an analogue of Lemma 2.10 of \cite{Ch16}:
\begin{lem}
\label{lem: crossings of annuli}
(Crossings of annuli). If $\eta\in{(0,\pi)}$ is the absolute constant from Lemma \ref{lem: property (S)} and $\tau=e^{3\pi\tan(\eta/2)+4\pi\cot(\eta/2)}$, there exist absolute constants $\rho>0$, $K>0$, so that if $G=(V,E,c)$ is orthodiagonal map with edges of length at most $\varepsilon$, $v,w\in{V^{\bullet}}$ are vertices of $G^{\bullet}$, $\mathcal{A}_{\tau}=\{z:\tau^{-1}|v-w|\leq{|z-w|}\leq{\tau|v-w|}\}$, and $\gamma=(x_{0}, x_{1}, x_{2}, ..., x_{m})$ is a simple, nearest-neighbor path in $G^{\bullet}$ so that $x_{i}\in{\mathcal{A}_{\tau}}$ for $i=1,2,...,m-1$, $|x_{0}|\leq{\tau^{-1}|v-w|}$, and $|x_{m}|\geq{\tau|v-w|}$, then a simple random walk on $G^{\bullet}$ started at $v$ will intersect $\gamma$ before exiting the annulus $\mathcal{A}_{\tau}$ with probability at least $\rho$, provided that $|v-w|\geq{K\varepsilon}$ and $B(w,\tau|v-w|)\subseteq{\widehat{G}}$.
\end{lem}
\begin{proof}
For any $\tau>1$, identifying the vertices and edges of $\gamma$ with the corresponding points and line segments in the complex plane, the domain $\mathcal{A}_{\tau}\setminus{\gamma}$ is simply connected. Let $arg(z-w)$ be the branch of the argument function on $\mathcal{A}_{\tau}\setminus{\gamma}$, centered at $w$, so that $arg(z-w)\in{[0,2\pi]}$ on $\{z:|z-w|=\tau|v-w|\}\subseteq{\partial{(\mathcal{A}_{\tau}\setminus{\gamma})}}$. Suppose $\tau_{0}=0$ and $(S_{n})_{n\geq{0}}$ is a simple random walk on $G^{\bullet}$ started at $v$. For $k\geq{1}$, we define the stopping times $(\tau_{k})_{k\geq{1}}$ inductively as follows:
\begin{equation*}
    \tau_{k}:=\inf\{n\geq{\tau_{k-1}}:|S_{n}-S_{\tau_{k-1}}|\geq{\alpha|S_{\tau_{k-1}}-w|}\}
\end{equation*}
Here, $\alpha\in{(0,1)}$ is some parameter whose value we will specify later. 
In other words, $\tau_{k}$ is the first time after $\tau_{k-1}$ at which our simple random walk exits the ball of radius $\alpha|S_{\tau_{k-1}}-w|$ about $S_{\tau_{k-1}}$. Define $v_{k}:=S_{\tau_{k}}$, $R_{k}:=|v_{k}-w|$, $\Delta\theta_{k}:=(arg(v_{k+1}-w)-arg(v_{k}-w))$. By Property (S), $arg(v_{k+1}-v_{k})\in{[arg(v_{k}-w)-\frac{\eta}{2}, arg(v_{k}-w)+\frac{\eta}{2}]}$ with probability at least $c$, where $c>0$ is the absolute constant from Lemma \ref{lem: property (S)}. In this case, we have that:
\begin{align*}
    \frac{R_{k+1}}{R_{k}}&\geq{\sqrt{1+2\alpha\cos\big(\eta/2\big)+\alpha^{2}}}, & \sin(|\Delta\theta_{k}|)&\leq{\frac{\big(\alpha+\frac{2\varepsilon}{|v_{k}-w|}\big)\sin(\eta/2)}{\sqrt{1+2\big(\alpha+\frac{2\varepsilon}{|v_{k}-w|}\big)\cos(\eta/2)+\big(\alpha+\frac{2\varepsilon}{|v_{k}-w|}\big)^{2}}}}
\end{align*}
If $|v_{k}-w|\geq{\alpha^{-2}\varepsilon}$, it follows that $\Delta\theta_{k}$ is at most some number that is asymptotic to $\alpha\sin(\eta/2)$ as $\alpha\rightarrow{0}$. Applying Property (S) in this way for $k=0,1,2,...,m-1$,\footnote{To apply Property (S) here, it is necessary that $\alpha|v_{k}-v|\geq{C\varepsilon}$ for $k=0,1,2,..., m-1$, where $C>0$ is the absolute constant in Lemma \ref{lem: property (S)}. Having fixed $\alpha\in{(0,1)}$, with each step we take, $|v_{k}-v|$ decreases by a factor of at most $1-\alpha$. Thus, taking $|v-w|\geq{\alpha^{-1}(1-\alpha)^{-m}C\varepsilon}\sim{Ce^{\frac{2\pi}{\sin(\eta/2)}}\varepsilon\alpha^{-1}}$ ensures that $\alpha|v_{k}-w|\geq{C\varepsilon}$ for $k=0,1,2,...,m-1$. We can similarly ensure that $|v_{k}-w|\geq{\alpha^{-2}\varepsilon}$ for $k=0,1,2,...,m-1$.} where $m\sim{\frac{2\pi}{\alpha\sin(\eta/2)}}$, we have that with probability at least $c^{m}$, $|arg(S_{n}-w)-arg(v-w)|\leq{2\pi}$ for $n=0,1,2,...,\tau_{m}$, and $|v_{m}-v|\geq{\big(1+2\alpha\cos\big(\eta/2\big)+\alpha^{2}\big)^{m}}\rightarrow{e^{\frac{4\pi\cos(\eta/2)}{\sin(\eta/2)}}}$ as $\alpha\rightarrow{0}$. In other words, for sufficiently large $\tau$ such that $1<\tau<e^{\frac{4\pi\cos(\eta/2)}{\sin(\eta/2)}}$, a random walk started at $v$ exits the annulus $\mathcal{A}_{\tau}$ through the outer boundary arc before making a full turn around $w$ (either clockwise or counterclockwise) with probability at least $\rho$, provided that $|v-w|\geq{K\varepsilon}$ and $B(w,\tau|v-w|)\subseteq{\widehat{G}}$.  Here, $\rho=\rho(\tau)>0$ and $K=K(\tau)>0$ are absolute constants, independent of the geometry of our orthodiagonal map. If $arg(v-w)\geq{4\pi}$ or $arg(v-w)\leq{-2\pi}$, this implies that with probability at least $\rho$, our random walk will intersect $\gamma$ before exiting $\mathcal{A}_{\tau}$.  \\ \\
If $-2\pi<arg(v-w)<4\pi$, we proceed as follows. Let $(v_{k})_{k\geq{0}}, (R_{k})_{k\geq{0}}, (\Delta\theta_{k})_{k\geq{1}}$ be as above. Without loss of generality, suppose $\pi\leq{arg(v-w)}<{4\pi}$. By Property (S), $arg(v_{k+1}-v_{k})\in{[arg(v_{k}-w)+\frac{\pi}{2}+\frac{\eta}{2}, arg(v_{k}-w)+\frac{\pi}{2}-\frac{\eta}{2}]}$ with probability at least $c>0$. In this case, we have that: 
\begin{equation*}
    \Delta\theta_{k}\geq{\sin(\Delta\theta_{k})}\geq{\frac{\alpha\cos(\eta/2)}{\sqrt{1+2\alpha\sin(\eta/2)+\alpha^{2}}}}
\end{equation*}
\begin{equation*}
     \sqrt{1+\Big(\alpha+\frac{2\varepsilon}{|v_{k}-w|}\Big)^{2}-2\Big(\alpha+\frac{2\varepsilon}{|v_{k}-w|}\Big)\sin(\eta/2)}\leq{\frac{R_{k+1}}{R_{k}}}\leq{\sqrt{1+\Big(\alpha+\frac{2\varepsilon}{|v_{k}-w|}\Big)^{2}+2\Big(\alpha+\frac{2\varepsilon}{|v_{k}-w|}\Big)\sin(\eta/2)}}
\end{equation*}
Taking $|v_{k}-w|\geq{\alpha^{-2}\varepsilon}$, $\Delta\theta_{k}$ is at least some number which is asymptotic to $\alpha\cos(\eta/2)$ as $\alpha\rightarrow{0}$, $R_{k+1}/R_{k}$ is at least some number which is asymptotic to $1-\alpha\sin(\eta/2)$, and at most some number which is asymptotic to $1+\alpha\sin(\eta/2)$, as $\alpha\rightarrow{0}$. Applying Property (S) in this way for $k=0,1,2,...,m_{1}-1$, where $m_{1}\sim{\frac{3\pi}{\alpha\cos(\eta/2)}}$, we have that with probability at least $c^{m_{1}}$, $arg(v_{m_{1}}-w)-arg(v-w)\geq{3\pi}$, and:
\begin{equation*}
    \underbrace{(1-\alpha\sin(\eta/2)+O(\alpha^{2}))^{m_{1}}}_{\rightarrow{e^{-\frac{3\pi\sin(\eta/2)}{cos(\eta/2)}}} \text{as }\alpha\rightarrow{0}}\leq{|S_{n}-w|}\leq{\underbrace{(1+\alpha\sin(\eta/2)+O(\alpha^{2}))^{m_{1}}}_{\rightarrow{e^{\frac{3\pi\sin(\eta/2)}{cos(\eta/2)}}} \text{as }\alpha\rightarrow{0}}} \hspace{5pt}\text{for $n=0,1,2,...,\tau_{m_{1}}$}
\end{equation*}
Thus, for any $\tau_{1}>e^{\frac{3\pi\sin(\eta/2)}{cos(\eta/2)}}$, a random walk on $G^{\bullet}$ started at $v$ makes a $3\pi$-turn counterclockwise around $w$ before exiting the annulus $\mathcal{A}_{\tau_{1}}$ with probability at least $\rho_{1}>0$, provided that $|v-w|\geq{K_{1}\varepsilon}$ and $B(w,\tau_{1}|v-w|)\subseteq{\widehat{G}}$. Here, $\rho_{1}=\rho_{1}(\tau)>0$ and $K_{1}=K_{1}(\tau)>0$ are absolute constants, independent of the geometry of our orthodiagonal map. Initially, $\pi\leq{arg(v-w)}<4\pi$. Having made a $3\pi$-turn counterclockwise around $w$, $arg(v_{m_{1}}-w)\geq{4\pi}$. Hence, applying our argument from the first case to a simple random walk on $G^{\bullet}$ started from $v_{m_{1}}$, we have that for any $\tau_{2}\geq{e^{3\pi\tan(\eta/2)+4\pi\cot(\eta/2)}}$, a random walk on $G^{\bullet}$ started at $v$ will intersect $\gamma$ before exiting the annulus $\mathcal{A}_{\tau_{2}}$ with probability at least $\rho_{2}$, provided that $|v-w|\geq{K_{2}\varepsilon}$ and $B(w,\tau_{2}|v-w|)\subseteq{\widehat{G}}$. Here, $\rho_{2}=\rho_{2}(\tau)>0$ and $K_{2}=K_{2}(\tau)>0$ are absolute constants, independent of the geometry of our orthodiagonal map. This concludes our proof.   
\end{proof}
\noindent Armed with this lemma, we are ready to prove the weak Beurling estimate, Lemma \ref{lem: weak Beurling on OD maps}:
\begin{proof}
(of Lemma \ref{lem: weak Beurling on OD maps}) Let $w$ be the point of $\partial{V^{\bullet}}$ closest to $u$ so that $d_{u}=|u-w|$. Suppose $\tau_{0}=0$, and $(S_{n})_{n\geq{0}}$ is a simple random walk on $G^{\bullet}$, started at $u$. For $k\geq{1}$, we define the stopping times $(\tau_{k})_{k\geq{1}}$ inductively as follows:
\begin{equation*}
    \tau_{k}:=\inf\{n\geq{\tau_{k-1}}: |S_{n}-w|\geq{\tau|S_{\tau_{k-1}}-w|}\}
\end{equation*}
In other words, $\tau_{k}$ is the first time that the distance between $S_{n}$ and $w$ is at least $\tau$ times the distance between $S_{\tau_{k-1}}$ and $w$. Observe that for each $k$:
\begin{equation*}
    \tau|S_{\tau_{k-1}}-w|\leq{|S_{\tau_{k}}-w|}\leq{\tau|S_{\tau_{k-1}}-w|+2\varepsilon}
\end{equation*}
This is because the edges of $G^{\bullet}$ have length at most $2\varepsilon$. Since $|S_{\tau_{0}}-w|=|u-w|$, it follows that for each $k$: 
\begin{equation}
\label{eqn: bounds on S-tau-k}
    \tau^{k}|u-w|\leq{|S_{\tau_{k}}-w|}\leq{\tau^{k}|u-w|+\frac{2\tau^{k}}{\tau-1}\varepsilon}
\end{equation}
By Lemma \ref{lem: crossings of annuli}, a simple random walk started at $S_{\tau_{k-1}}$ will exit the ball $B(w,\tau|S_{\tau_{k-1}}-w|)$ without hitting $\partial{V}^{\bullet}$ first, with probability at most $1-\rho$, provided that $|S_{\tau_{k-1}}-w|\geq{K\varepsilon}$, where $\rho,\tau,K$ are the absolute constants in the statement of Lemma \ref{lem: crossings of annuli}. By Equation \ref{eqn: bounds on S-tau-k}, $|u-w|\geq{K\varepsilon}$ implies that $|S_{\tau_{k-1}}-w|\geq{K\varepsilon}$ for all $k$. Hence, if $\tau_{\partial{V^{\bullet}}}$ is the hitting time of $\partial{V^{\bullet}}$ by our random walk, by the strong Markov property and Lemma \ref{lem: crossings of annuli}, we have that:
\begin{align*}
    \mathbb{P}^{u}(\tau_{m}<\tau_{\partial{V^{\bullet}}})&=\underbrace{\mathbb{P}^{u}(\tau_{1}<\tau_{\partial{V^{\bullet}}})}_{\leq{(1-\rho)}}\mathbb{E}^{u}\Big[\prod_{k=1}^{m-1}\underbrace{\mathbb{P}^{S_{\tau_{k}}}(\tau_{k+1}<\tau_{\partial{V}^{\bullet}}|\tau_{k}<\tau_{\partial{V}^{\bullet}})}_{\leq{(1-\rho)}}\Big] \\ 
    &\leq{(1-\rho)^{m}}
\end{align*}
provided that $|u-w|\geq{K\varepsilon}$. Let $m=\Big\lfloor\frac{\log(\frac{r}{|u-w|+\frac{2\varepsilon}{\tau-1}})}{\log(\tau)}\Big\rfloor$ so that $m$ is the largest integer such that $\tau^{m}+\frac{2\tau^{m}}{\tau-1}\varepsilon\leq{r}$. Then:
\begin{align*}
    \mathbb{P}^{u}(|S_{\tau_{\partial{V^{\bullet}}}}-w|\geq{r})&\leq{ \mathbb{P}^{u}(\tau_{m}<\tau_{\partial{V^{\bullet}}})}\leq{(1-\rho)^{m}}\leq{(1-\rho)^{-1}\left(\frac{|u-w|+\frac{2\varepsilon}{\tau-1}}{r}\right)^{\frac{\log(\frac{1}{1-\rho})}{\log(\tau)}}} \\
    &\leq{C\left(\frac{|u-w|}{r}\right)^{\frac{\log(\frac{1}{1-\rho})}{\log(\tau)}}}=C\Big(\frac{d_{u}}{r}\Big)^{\frac{\log(\frac{1}{1-\rho})}{\log(\tau)}}
\end{align*}
where $C>0$ is an absolute constant, assuming that $d_{u}=|u-w|\geq{K\varepsilon}$. If on the other hand $d_{u}<K\varepsilon$, by the maximum principle we have that: 
\begin{equation*}
    \mathbb{P}^{u}(|S_{\tau_{\partial{V^{\bullet}}}}-w|\geq{r})\leq{C\Big(\frac{K\varepsilon}{r}\Big)^{\frac{\log(\frac{1}{1-\rho})}{\log(\tau)}}}
\end{equation*}
This completes our proof. Furthermore, we see that we can take $\beta=\frac{\log(\frac{1}{1-\rho})}{log(\tau)}$, where $\tau$ and $\rho$ are the absolute constants in the statement of Lemma \ref{lem: crossings of annuli}.
\end{proof}
\noindent Having established Lemma \ref{lem: weak Beurling on OD maps}, Lemma \ref{lem: discrete harmonic functions are beta Holder in the bulk}, which tells us that discrete harmonic functions on orthodiagonal maps are $\beta$- H\"older in the bulk for some absolute constant $\beta\in{(0,1)}$, follows as an immediate corollary:
\begin{proof}
    (of Lemma \ref{lem: discrete harmonic functions are beta Holder in the bulk}) Let $K>0$ be the absolute constant in Lemma \ref{lem: weak Beurling on OD maps}. Observe that if $|x-y|\vee{K\varepsilon}>{\frac{1}{2}d}$, the desired result holds trivially, since $|h(y)-h(x)|\leq{2\|h\|_{\infty}}$. With this in mind, suppose that $|y-x|\vee{K\varepsilon}\leq{\frac{1}{2}d}$ and WLOG, $h(y)\geq{h(x)}$. Additionally, suppose that $|x-y|>K\varepsilon$. By the maximum principle, there exists a path $\gamma=(w_{0}, w_{1}, ..., w_{m})$ of vertices in $G^{\bullet}$ so that $w_{0}=y$, $w_{m}\in{\partial{V^{\bullet}}}$, and $h(w_{i})\geq{h(y)}$ for $i=0,1,2,...,m$. By Lemma \ref{lem: weak Beurling on OD maps}, with high probability, a simple random walk started at $x$ will hit $\gamma$ before exiting the ball of radius $d$ centred at $y$. Thus: 
    \begin{align*}
        h(x)&\geq{\left(1-C\Big(\frac{|y-x|}{d-|y-x|}\Big)^{\beta}\right)h(y)-C\Big(\frac{|y-x|}{d-|y-x|}\Big)^{\beta}\|h\|_{\infty}}  \\ 
        &\geq{\Big(1-C\Big(\frac{2|y-x|}{d}\Big)^{\beta}\Big)h(y)-C\Big(\frac{2|y-x|}{d}\Big)^{\beta}\|h\|_{\infty}}
    \end{align*}
    Since $h(y)\geq{h(x)}$, it follows that:
    \begin{equation*}
        |h(y)-h(x)|\leq{C'\|h\|_{\infty}\Big(\frac{|y-x|}{d}\Big)^{\beta}}    
    \end{equation*}
    where $C'>0$ is an absolute constant. If on the other hand $|x-y|\leq{K\varepsilon}$, where $K\varepsilon\leq{\frac{1}{2}d}$, the maximum principle tells us that:
    \begin{equation*}
        |h(y)-h(x)|\leq{C'\|h\|_{\infty}\Big(\frac{K\varepsilon}{d}\Big)^{\beta}} 
    \end{equation*}
    This completes our proof. 
\end{proof}
\end{appendices}

\end{document}